\documentclass[twoside,a4paper,11pt]{article} 

\usepackage[left=2.5cm, right=2.5cm, top=2.5cm, bottom=3.0cm]{geometry}

\usepackage[T1]{fontenc}
\usepackage[utf8]{inputenc}
\usepackage[french,english]{babel}

\usepackage{bm}

\usepackage{amsmath,amssymb,amsfonts,amsthm}
\usepackage[dvipsnames]{xcolor}

\DeclareMathAlphabet{\matcal}{OMS}{zplm}{m}{n}

\usepackage[normalem]{ulem}
\usepackage{cancel}

\usepackage[pdftex]{graphicx}	
\usepackage{caption}
\usepackage{url}

\usepackage[colorlinks=true]{hyperref}
\hypersetup{colorlinks,
            citecolor=[rgb]{0.4,0,0},  
            linkcolor=[rgb]{0.05,0.20,0.55}}

\usepackage{bbm}

\usepackage{sectsty}		
\sectionfont{\normalfont } 

\subsectionfont{\normalfont \bfseries }

\subsubsectionfont{\normalfont \bfseries }

\usepackage{enumitem}
\newlist{hyp0}{enumerate}{1}
\setlist[hyp0]{label=(C\arabic*)}
\newlist{hyp0'}{enumerate}{1}
\setlist[hyp0']{label=(H\arabic*)}

\newlist{hyp1}{enumerate}{1}
\setlist[hyp1]{label=(A\arabic*)}

\newlist{condition}{enumerate}{1}
\setlist[condition]{label=(C)}

\usepackage[numbers,comma,sort]{natbib}
\usepackage{amsmath}	

\numberwithin{equation}{section}		

\newtheorem{theorem}{Theorem}[section]

\newtheorem*{theorem*}{Theorem}
\newtheorem{lemma}[theorem]{Lemma}

\newtheorem{proposition}[theorem]{Proposition}
\newtheorem{remark}[theorem]{Remark}
\newtheorem{convention}[theorem]{Notational convention}

\newtheorem{innercustomthm}{Theorem}
\newenvironment{customthm}[1]
  {\renewcommand\theinnercustomthm{#1}\innercustomthm}
  {\endinnercustomthm}

\newtheorem{innercustomprop}{Proposition}
\newenvironment{customprop}[1]
  {\renewcommand\theinnercustomprop{#1}\innercustomprop}
  {\endinnercustomprop}

\usepackage{fancyhdr}
\fancypagestyle{empty}{\fancyfoot[C]{\vspace*{-1.0\baselineskip}\thepage}}

\usepackage[titletoc]{appendix}
\AtBeginDocument{}

\DeclareMathOperator{\Tr}{Tr}

\newcommand{\Cb}{\matcal{C}_b}
\newcommand{\Yc}{\mathcal{Y}^H}

\newcommand{\R}{\mathbb{R}}

\newcommand{\N}{\mathbb{N}}
\newcommand{\EE}{\mathbb{E}}
\newcommand{\PP}{\mathbb{P}}

\newcommand{\HH}{\matcal{H}_H}

\newcommand{\wl}{\mathbf{W}_\lambda}
\newcommand{\rate}{\mathcal{M}}
\newcommand{\ratet}{\widetilde{\mathcal{M}}}

\newcommand{\sbullet}{{\scalebox{0.4}{$\bullet$}}}

\newcommand\indi[1]{\mathbb{I}_{{\mathnormal{#1}}}}

\begin{document}

\title{\normalfont \Large  {Lipschitz continuity in the Hurst parameter 
of 
functionals of stochastic differential equations driven by a frac{t}ional 
Brownian motion}}
\vspace{-1cm}

\author{Alexandre Richard\footnote{Universit\'e Paris-Saclay, CentraleSup\'elec and CNRS FR-3487, France.\newline  E-mail: \texttt{alexandre.richard@centralesupelec.fr}.} 
\and Denis Talay\footnote{Equipe-projet ASCII, Centre Inria de Saclay and Ecole Polytechnique, 
France.\newline E-mail: \texttt{denis.talay@inria.fr}.}}

\date{\small \today}

\maketitle

\begin{abstract}
 
Sensitivity analysis w.r.t. the long-range/memory noise parameter for 
probability distributions of functionals of solutions to stochastic 
differential equations is an important stochastic modeling issue in 
many applications.

In this paper we consider solutions $\{X^H_t\}_{t\in 
\R_+}$ to stochastic differential equations driven 
by frac{t}ional Brownian motions.
We develop two innovative sensitivity analyses
when the Hurst parameter~$H$ of the noise tends to the 
critical Brownian parameter $H=\tfrac{1}{2}$ from above or from 
below. First, we examine expected smooth functions of $X^H$ at a 
fixed time horizon~$T$. 
Second, we examine Laplace transforms of functionals which are 
irregular with regard to Malliavin calculus, namely, first passage 
times of $X^H$ at a given threshold.

In both cases we exhibit the Lipschitz continuity w.r.t.~$H$ around the 
value $\tfrac{1}{2}$. Therefore, our results show that the
Markov Brownian model is a good proxy 
model as long as the Hurst parameter remains close to~$\tfrac{1}{2}$.
\end{abstract}


\tableofcontents

\section{Introduction} \label{sec:Intro}

In many applied situations where stochastic 
differential equations are used, one chooses Markovian dynamics, in particular for 
the following reasons. A huge literature in stochastic analysis exists on the well-posedness, 
calibration and simulation of Markov models. Their probability distributions can be 
characterized by extensively studied partial differential equations
or integro-differential equations.  In addition, well developed 
techniques allow one to describe the convergence of Markov processes
to homogenized systems, mean-field limits or equilibrium regimes.

However, Markov models may sometimes be seen as questionable
idealizations of the reality. Empirical studies actually tend to show
memory effects in biological, financial, physical data: see
e.g.~\citet{Rypdal} for a statistical evidence in climatology. Such
empirical results justify to consider non-Markov models driven by noises
with long-range memory such as frac{t}ional Brownian motions rather than
by L\'evy noises. But \citet{JolisViles} emphasise that
choosing a noise with long-range memory does not close the modeling
problem since the parametric estimation of the model may be difficult
and crude (see~\citet{BerzinEtAl} for the statistics of stochastic
models with long-range memory).

Therefore, one often needs to balance tractable Markov models against
more realistic but complex non-Markov models. A natural question then
arises: Is it really worth introducing complex models?

To tackle this issue, we consider the case of solutions $X^H$ to 
one-dimensional stochastic differential equations driven by 
frac{t}ional Brownian 
motions~$B^H$, where the Hurst constant~$H$ parameterizes the 
covariance function of~$B^H$. The case $H=\tfrac{1}{2}$ corresponds to 
the \textsl{pure standard Brownian case}. The driving noise and 
corresponding solution are then respectively denoted by 
$\mathbf{B}\equiv 
B^{\frac{1}{2}}$ and $\mathbf{X}\equiv X^{\frac{1}{2}}$.

We develop a sensitivity analysis, w.r.t. $H$ around the reference 
value $H=\frac{1}{2}$, of probability distributions of certain 
functionals of the solutions~$X^H$. We examine the two following cases 
which respectively are regular and singular with regard to 
Malliavin calculus: on the one hand, the time marginal distributions of 
the solutions; on the other hand, the Laplace transform of first 
passage times. 

We have three motivations to consider first passage times of~$X^H$.
First, the analysis of first passage times 
is an important issue in physical sciences~(\citet{M-R-O-14}),
in the evaluation of default risks and ruin probabilities 
(\citet{Jeanblanc-al}), and in the study of neuron spike 
trains~(see \citet{ROT} and references therein). 
Second, our discussion on Markovian or non-Markovian modeling applies 
in 
force to the study of hitting times. Indeed, the Markov property of the 
process is 
essential to calculate exact probability distributions of first passage 
times, characterize these distributions by means of partial 
differential equations, or construct numerical algorithms to simulate 
them: See e.g. 
\citet{salminen-yor}, \citet{alili-patie}, 
\citet{deaconu-herrmann} and citations therein. On the contrary, the 
long-range dependence leads to analytical and numerical difficulties: 
See e.g. \citet{Jeon-al}. 
Last, it seems to us 
worthy of 
showing that an accurate sensitivity analysis is possible even in a 
case which is singular with regard to Malliavin calculus.

For $H=\frac{1}{2}$ and $X^H$ reduced to the standard Brownian motion 
$\mathbf{B}$, the Laplace transform of the first hitting time 
$\tau_{\mathbf{B}}$ of the threshold~1 satisfies 
\begin{equation} \label{Laplace:brownien}
\forall x_0\leq 1,~~\EE\left[\left(e^{-\lambda 
\tau_{\mathbf{B}}}\right)~|~
\mathbf{B}_0=x_0\right]= e^{-(1-x_0)\sqrt{2\lambda}}.
\end{equation}

For $H\neq\frac{1}{2}$, even in the simple case where $X^H$ reduces to 
$B^H$, the probability distribution of $\tau_B^H$ is 
not explicitly known. \citet{Molchan} 
obtained the asymptotic behaviour of its tail 
distribution function and \citet{DecreusefondNualart08} obtained an 
upper bound on the Laplace transform of $(\tau^H_B)^{2H}$. 
Related works are those of~\citet{NourdinViens} on the density of 
$\sup_{t\in [a,b]} B^H_t - \EE\left(\sup_{t\in [a,b]} B^H_t\right)$ 
where $0<a<b$, and of~\citet{BaudoinNualartOuyangTindel14} on hitting 
probabilities of multidimensional fractional diffusions. 
The recent work of \citet{DelormeWiese} proposes an asymptotic 
expansion (in terms of $H-\tfrac{1}{2}$) of the density of
$\sup_{t\in [0,b]} B^H_t$; this expansion is
formally obtained by perturbation analysis techniques. 

Below, we obtain an accurate estimate for 
$$ \left|\EE\left(e^{-\lambda \tau^H_X}\right) - \EE\left(e^{-\lambda 
\tau_{\mathbf{X}}}\right)\right|, $$
with explicit rates in terms of $|H-\frac{1}{2}|$, $\lambda>0$,
and the distance from the initial condition~$X^H_0$ to the threshold.
This result implies the vague convergence of $\tau^H_X$ to 
$\tau_{\mathbf{X}}$ when $H$ tends to $\tfrac{1}{2}$. In addition,
in Section~\ref{sec:examples} we exhibit sufficient conditions to get
the weak convergence and show how to use the preceding convergence rate
of the Laplace transforms to quantify the weak convergence of $\tau^H_X$.

Our sensitivity analyses of time marginal and first passage time 
distributions tend to show that the Markov Brownian model is a good 
proxy model as long as the Hurst parameter remains close 
to~$\tfrac{1}{2}$. This is an important information for modeling and 
simulation purposes: Whenever statistical or calibration procedures 
lead to estimated values of~$H$ close to $\tfrac{1}{2}$, it then is 
reasonable to work with Brownian SDEs and standard stochastic 
integration theory.

\paragraph{Why do we limit ourselves to a sensitivity analysis 
around~$H=\frac{1}{2}$?}
In this paper, contrary to~\citet{Giordano-al} and \citet{JolisViles} 
we do not develop a sensitivity analysis of the model 
around $H'\neq\frac{1}{2}$. Our reasons are as follows. First, as 
already explained, it seems interesting to us to obtain as good as 
possible sensitivity estimates around a Markov proxy model: We actually 
get Lipschitz continuity properties. 
Second, the fact that the proxy model has the Markov property allows us 
to apply the It\^o--Skorokhod formula proven in 
Section~\ref{sec:prelim} to the solution of a suitable ordinary 
differential equation, which allows us to transform the sensitivy 
analysis of the Laplace transform of $\tau^H_X$ around 
$H=\tfrac{1}{2}$ to the sensitivity analysis of Skorokhod integrals 
depending on~$X^H$ and stopped at~$\tau^H_X$. We thus can 
benefit from the fact that $X^H$ is a smooth functional on a suitable 
Wiener space. We do not see how to extend this stategy when the proxy 
model is not Markov. 
Finally, the equality 
\begin{equation} \label{eq:intro-pivot-H'}
\EE\left(e^{-\lambda \tau^H_X}\right) 
- \EE\left(e^{-\lambda \tau^{H'}_X}\right)
= \EE\left(e^{-\lambda \tau^H_X}\right) 
- \EE\left(e^{-\lambda \tau_{\mathbf{X}}}\right) 
+ \EE\left(e^{-\lambda \tau_{\mathbf{X}}}\right)
- \EE\left(e^{-\lambda \tau^{H'}_X}\right)
\end{equation}
does not seem to help to get a sharp estimate in terms of~$|H-H'|$: See 
Remark~\ref{rk:commentaire-H'} below.

\paragraph{Organization of the paper.} In Section~\ref{sec:main} we 
state and comment our main results. In Section~\ref{sec:prelim} we 
review elements of stochastic calculus for frac{t}ional Brownian 
motion and we prove an It\^o formula for drifted fractional Brownian 
motions. In Section~\ref{sec:fracDiff} we prove 
Proposition~\ref{th:gapLawXt^H} which concerns the sensitivity 
w.r.t.~$H$ of expected smooth functions of~$X^H_t$ for every~$t>0$. Our 
proof of this proposition allows 
us to smoothly introduce our strategy to analyse Laplace transforms of 
first passage times. In Section~\ref{sec:majgap_0} we prove
our main result, namely, Theorem~\ref{prop:majgap_0}. 
In Section~\ref{sec:examples} we exhibit sufficient conditions for the weak convergence
of $\tau_X^H$ to $\tau_{\mathbf{X}}$ and we
apply our main results to quantify this weak convergence.
Various technical lemmas are gathered in Appendix~\ref{app:KH-etoile}, 
Appendix~\ref{App:tecLem} and Appendix~\ref{App:moments}. 
Estimates on the
derivatives of the Laplace transform of $\tau_{\mathbf{X}}$ are proven 
in Appendix~\ref{app:boundW}. Finally, the reader can find a glossary 
of our various processes, functions, etc. in 
Appendix~\ref{App:glossary}. 

\paragraph{Notations.}
For any random variable in $L^p(\Omega)$ we set
$$ \|F\|_p := \left\{ \EE(\|F\|^p) \right\}^{\frac{1}{p}}. $$

We denote by $C$ any constant which may change from line  to line. 
It may depend on the Hurst parameter~$H$ but, in that case, it is a 
bounded function of~$H$ in~$[\frac{1}{4},1]$. 

We denote by $C_H$ any constant depending on~$H$ which tends to 
infinity when $H$ tends to $\frac{1}{4}$ or $1$ and is bounded on any 
closed subinterval of $(\frac{1}{4},1)$. Such a constant may depend on 
various parameters except on the parameter~$\lambda$ of the Laplace 
transform and the time horizon $N$ considered in 
Section~\ref{sec:majgap_0}.

\section{Main results} \label{sec:main}
We are given a frac{t}ional Brownian motion $\{B^H_t\}_{t\in \R_+}$ 
with 
Hurst parameter $H\in(0,1)$. This process is the only Gaussian process 
with stationary increments which is self-similar 
of order $H$ (up to centering and normalization of the 
variance). Its covariance reads:
\begin{equation}\label{eq:defCov}
R_H(s,t) = \tfrac{1}{2}\left(s^{2H} + t^{2H} - |t-s|^{2H}\right), \quad 
\forall s,t\in \R_+.
\end{equation}

We also are given two functions $b$ and $\sigma$ which satisfy:
\begin{hyp0'}
\item\label{hyp:h1'} $b$ belongs to $\Cb^1(\R)$ and $\sigma$ 
belongs to $\Cb^2(\R)$.
\item\label{hyp:h2'} The function $\sigma$ satisfies the strong 
ellipticity condition: $\exists \sigma_0>0$ such that $|\sigma(x)|\geq 
\sigma_0$ for every $x\in \R$.
\end{hyp0'}

As it will be recalled in Subsection~\ref{subsec:SDE_Ito}, the 
preceding hypotheses imply that for every 
$H\in(\tfrac{1}{4},1)$ and $x_0\in\R$ there exists a unique solution 
$\{X^H_t\}_{t\in \R_+}$ to the stochastic differential equation driven 
by~$\{B^H_t\}_{t\in \R_+}$:
\begin{equation}\label{eq:fSDE_0}
X^H_t = x_0 + \int_0^t b(X^H_s)\ ds + \int_0^t \sigma(X^H_s)
~\circ dB^H_s. 
\end{equation}
In particular, for $H=\tfrac{1}{2}$, there exists a unique
square integrable strong solution~$\mathbf{X}\equiv X^{\frac{1}{2}}$ 
to the Brownian SDE in the Stratonovich sense
\begin{equation}\label{eq:SDE_0}
\mathbf{X}_t = x_0 + \int_0^t b(\mathbf{X}_s)\ ds 
+ \int_0^t \sigma(\mathbf{X}_s)~\circ d\mathbf{B}_s. 
\end{equation}

Our first result is easy to prove but instructive. It will be proven in 
Section~\ref{subsec:proof_gapLawXt^H}. It concerns the sensitivity 
w.r.t. $H$ around the critical Brownian parameter $H=\tfrac{1}{2}$ of 
$\EE\varphi(X^H_t)$, where $\varphi$ is a smooth function.

\begin{customprop}{\ref{th:gapLawXt^H}}
Let $X^H$ and $\mathbf{X}$ be as above. 
Suppose that $b$ and $\sigma$ satisfy \ref{hyp:h1'} and \ref{hyp:h2'}, 
and that $\varphi$ is 
bounded and H\"older continuous of order $2+\beta$ for some $\beta>0$. 
Then, for any $T>0$, there exists $C>0$ such that 
\begin{equation*}
\forall H\in (\tfrac{1}{4},1),~~\sup_{t\in [0,T]} 
\left|\EE \varphi(X_t^H) - \EE \varphi(\mathbf{X}_t) \right| 
\leq C~|H-\tfrac{1}{2}|.
\end{equation*}
\end{customprop}

Our second result concerns a singular functional of the paths,
namely, the first passage time of $X^H$ at a given threshold (1, say).
Given $x_0<1$, set
\begin{equation}\label{eq:defTauH}
\tau^H_X := \inf\{t\geq 0: X^H_t=1\}.
\end{equation}

The precise formulation of our result is obtained by combining
the theorem~\ref{prop:majgap_0} and the proposition~\ref{prop:moments}.
Part of the difficulties overcome in the lengthy proof of~Theorem 
\ref{prop:majgap_0} come from the fact that we aim to get a 
sensitivity estimate which tends to~0 as fast as possible 
when $H$ tends to~$\tfrac{1}{2}$ and decays to~0 at the same rates as 
in the exact formula~\eqref{Laplace:brownien} when $\lambda$ and 
$|1-x_0|$ tend to infinity.

\begin{customthm} {\ref{prop:majgap_0}}
Let $X^H$ and $\mathbf{X}$ be the solutions to (\ref{eq:fSDE_0}) and 
(\ref{eq:SDE_0}) respectively, both with initial condition~$x_0<1$. 
Assume that $b$ and $\sigma$ satisfy \ref{hyp:h1'}-\ref{hyp:h2'}. 
For the monotone function~$F$ defined as in 
Proposition~\ref{prop:Lamperti}
let $\widetilde{b} := \frac{b\circ F^{-1}}{\sigma\circ F^{-1}}$,
$\mathbb{Y}:=F(1)$ and $y_0:=F(x_0)$.

For any $p\geq1$ and $\lambda>|\widetilde{b}'|_\infty$ set
\begin{equation*} %
\rate_p(\mathbb{Y}-y_0,\lambda) := \sup_{s\in \R_+} 
\left( e^{-\frac{1}{2} (\lambda-|\widetilde{b}'|_\infty) p 
s}
~\EE~e^{-|\mathbb{Y}-Y^H_s|p\mathcal{R}(\lambda)} \right),
\end{equation*} 
where
\begin{equation*} %
\mathcal{R}(\lambda):=\sqrt{2\lambda +\mu^2} - \mu
~~\text{with}~~\mu:=|\widetilde{b}|_\infty.
\end{equation*}

Suppose $x_0<1$ and $\lambda >|\widetilde{b}'|_\infty$. 
Set~$\widetilde{\lambda}:= \lambda - |\widetilde{b}'|_\infty$. 
For any $H\in(\tfrac{1}{4},1)$ we have  
\begin{multline*} \label{ineq:Laplace-transform-main-theorem}
\Big|\EE\left(e^{-\lambda \tau^H_X}\right) 
- \EE\left(e^{-\lambda \tau_{\mathbf{X}}}\right)\Big| \\
\leq C_H~|H-\tfrac{1}{2}|
~\frac{(1+\lambda)^2}{1\wedge \widetilde{\lambda}^3}
~\Big( \rate_{1}(\mathbb{Y}-y_0,{\lambda}) + 
\left(\rate_{2}(\mathbb{Y}-y_0,{\lambda}) 
\right)^{\frac{H\wedge\frac{1}{2}}{6}} + 
\left(\rate_{4}(\mathbb{Y}-y_0,{\lambda})\right)
^{\frac{H\wedge\frac{1}{2}}{12}} \Big). 
\end{multline*}
\end{customthm}

\begin{customprop}{\ref{prop:moments}}
Let $\lambda >|\widetilde{b}'|_\infty$. Let $m:=\mathbb{Y}-y_0$,
$\mu:=|\widetilde{b}|_\infty$, $q:=p\mathcal{R}(\lambda)$ and 
$\widetilde{\lambda}:= \lambda - |\widetilde{b}'|_\infty$. 

One has
\begin{equation*}  %
\rate_p(\mathbb{Y}-y_0,\lambda)
\leq C\left( e^{-\frac{q}{2}m}
+ e^{-\frac{\widetilde{\lambda}}{2} \Psi^H_q(m)} 
+\exp\left(-2^{-\frac{8}{3}}~m^{\frac{2}{1+2H}}~
\widetilde{\lambda}^{\frac{2H}{1+2H}}\right)
+ \exp\left(-\widetilde{\lambda} \tfrac{m}{2\mu} \right) \right),
\end{equation*}
where
\begin{equation*} %
\Psi^H_q(m) := \frac{m}{\mu+q}
~\indi{\left[\left(\frac{m}{\mu+q}\right)^{2H-1} < 1\right]}
+ \left(\frac{m}{\mu+q}\right)^{\frac{1}{2H}}
~\indi{\left[\left(\frac{m}{\mu+q}\right)^{2H-1} \geq 1\right]}.
\end{equation*}
\end{customprop}

The bound on $\rate_p(\mathbb{Y}-y_0,\lambda)$ given in Proposition \ref{prop:moments} yields the following asymptotic expressions when $\mu>0$:\\
If $H> \frac{1}{2}$,
\begin{equation*}
\text{for fixed $\mathbb{Y}-y_{0}$,}~
\begin{cases}
\rate_p(\mathbb{Y}-y_0,\lambda) \underset{\lambda\to \infty}{\lesssim} 
e^{-\frac{\mathbb{Y}-y_0}{2p} \sqrt{\frac{\lambda}{2}}} ,\\
\rate_p(\mathbb{Y}-y_0,\lambda) \underset{\widetilde{\lambda}\to 0}{\lesssim} e^{-\frac{1}{2} \big(\frac{\mathbb{Y}-y_0}{\mu} \wedge (\frac{\mathbb{Y}-y_0}{\mu})^{\frac{1}{2H}}\big) \widetilde{\lambda}},
\end{cases}
\end{equation*}
\begin{equation*}
\text{for fixed $\lambda$,}~
\begin{cases}
\rate_p(\mathbb{Y}-y_0,\lambda) \underset{\mathbb{Y}-y_{0}\to 
\infty}{\lesssim} e^{-2^{-\frac{8}{3}} 
\widetilde{\lambda}^{\frac{2H}{1+2H}} (\mathbb{Y}-y_0)^{\frac{2}{1+2H}} 
}, \\
\rate_p(\mathbb{Y}-y_0,\lambda) \underset{\mathbb{Y}-y_{0}\to 
0}{\lesssim} e^{-\frac{p\mathcal{R}(\lambda)}{2}(\mathbb{Y}-y_{0})} + 
e^{-\frac{\widetilde{\lambda}}{2(\mu+p\mathcal{R}(\lambda))}
(\mathbb{Y}-y_{0})}.
\end{cases}
\end{equation*}
If $H< \frac{1}{2}$,
\begin{equation*}
\text{for fixed $\mathbb{Y}-y_{0}$,}~
\begin{cases}
\rate_p(\mathbb{Y}-y_0,\lambda) \underset{\lambda\to \infty}{\lesssim} 
e^{-\frac{1}{2} \big(\frac{\mathbb{Y}-y_{0}}{p}\big)^{\frac{1}{2H}} 
\lambda^{\frac{4H-1}{4H}}}, \\
\rate_p(\mathbb{Y}-y_0,\lambda) \underset{\widetilde{\lambda}\to 0}{\lesssim} e^{-\frac{1}{2} \big(\frac{\mathbb{Y}-y_0}{\mu} \wedge (\frac{\mathbb{Y}-y_0}{\mu})^{\frac{1}{2H}}\big) \widetilde{\lambda}},
\end{cases}
\end{equation*}
\begin{equation*}
\text{for fixed $\lambda$,}~
\begin{cases}
\rate_p(\mathbb{Y}-y_0,\lambda) \underset{\mathbb{Y}-y_{0}\to 
\infty}{\lesssim} 
e^{-\frac{p\mathcal{R}(\lambda)}{2}(\mathbb{Y}-y_{0})} + 
e^{-\frac{\widetilde{\lambda}}{2(\mu+p\mathcal{R}(\lambda))}
(\mathbb{Y}-y_{0})}, \\
\rate_p(\mathbb{Y}-y_0,\lambda) \underset{\mathbb{Y}-y_{0}\to 
0}{\lesssim} 
e^{-\frac{\widetilde{\lambda}}{2(\mu+p\mathcal{R}(\lambda))^{1/(2H)}} 
(\mathbb{Y}-y_{0})^{\frac{1}{2H}} }.
\end{cases}
\end{equation*}

\section{Stochastic calculus for frac{t}ional Brownian motions and 
stochastic differential equations}\label{sec:prelim}

\subsection{Elements of stochastic calculus for fractional Brownian 
motion}\label{subsec:elements}

In this section, we briefly review the definition of Skorokhod
integrals w.r.t. frac{t}ional Brownian 
motions. The material mainly comes from \cite{Nualart}.

\begin{convention} \label{convention-of-writing}
In all this section we let the time horizon $T>0$ be fixed. 
This parameter~$T$ enters the definitions below of the operators 
$K^*_H$, $D^H$ and $\delta_H$, of the spaces $\HH$ and $|\HH|$, and of 
the corresponding norms on that spaces. For the sake of simplicity, the 
notation does not reflect the dependency on~$T$. However, when 
necessary, we will change the notation e.g. from $\delta_H$ to 
$\delta_H^{(T)}$.
\end{convention}

\paragraph{The integral kernels $K_H$.}
For any $H\in (0,1)\setminus \{\tfrac{1}{2}\}$ define the
kernel $K_H(s,u)$ as
\begin{equation} \label{def:KH}
\begin{cases}
\forall 0< s\leq r,~K_H(s,r) := 0, \\
\forall 0<r<s,~K_H(s,r) := \chi_H\ \left\{ 
\left(\frac{s(s-r)}{r}\right)^{H-\frac{1}{2}} - 
(H-\tfrac{1}{2})~r^{\frac{1}{2}-H} \int_r^s 
\theta^{H-\frac{3}{2}} (\theta-r)^{H-\frac{1}{2}}~d\theta\right\},
\end{cases}
\end{equation}
where $\chi_H$ is the bounded function of~$H$ on $(\frac{1}{4},1)$ 
defined by
\begin{equation} \label{def:chi-H}
\chi_H := \Big( \frac{2H\ \Gamma(3/2-H)}{\Gamma(H+\tfrac{1}{2})\ 
\Gamma(2-2H)}\Big)^{\frac{1}{2}}.
\end{equation}

Recall that for any $H\in(0,1)$, the covariance of the frac{t}ional 
Brownian motion is defined by~(\ref{eq:defCov}). The next equality 
explains the reason for which the kernel $K_H$ is introduced:
\begin{equation}\label{eq:decompCov}
R_H(s,t) = \int_0^{s\wedge t} K_H(s,u)~K_H(t,u)\ du.
\end{equation}

\paragraph{Useful properties of $K_H$.} In the sequel we will need the 
following basic properties of the kernel~$K_H$. 

First, since $K_H(\theta,r) = 0$ for 
$r\geq \theta$, we have 
\begin{equation} \label{eq-RH-KH}
\forall N>0,~\forall0<s<N,~\forall0<t<N,~~
R_H(s,t) = \int_0^N K_H(s,r)~K_H(t,r)\ dr.
\end{equation}
In particular, 
\begin{equation} \label{eq:norme-L2-K_H}
\forall 0<\theta<N,~~\theta^{2H} = \int_0^N K_H(\theta,v)^2 dv
= \int_0^\theta K_H(\theta,v)^2 dv.
\end{equation} 
Second, in view of the preceding equality and~(\ref{eq:defCov}), for 
any $0<v<\theta<N$ one has
\begin{align} \label{eq-delta-KH-carre}
\int_0^N (K_H(\theta,r)- K_H(v,r))^2 dr &= \int_0^N 
K_H(\theta,r)^2 dr + \int_0^N K_H(v,r)^2 dr -2 R_H(\theta,v) 
\nonumber \\
&= \theta^{2H} + v^{2H} - \left(\theta^{2H} + v^{2H} - 
(\theta-v)^{2H}\right) \nonumber \\
&= (\theta-v)^{2H}.
\end{align}

The last property of $K_H$ we will need is obtained by an easy 
calculation:
\begin{equation} \label{eq:deriv-K_H}
\partial_s K_H (s,r)  = 
\indi{\{s>r\}}~\chi_H~(H-\tfrac{1}{2})~
\left(\tfrac{s}{r}\right)^{H-\frac{1}{2}}
(s-r)^{H-\frac{3}{2}}.
\end{equation}

\paragraph{The operators $K_{H,s}^*$ and the spaces $\HH$, $|\HH|$.}
Given $s>0$ we define the operator $K_{H,s}^*$ as the dual in 
$L^2([0,s])$ of the integral operator with kernel $K_H$. For
step functions on $[0,s]$ this operator is defined by
$$ K_{H,s}^*\varphi(r) := K_H(s,r)~\varphi(r) 
+ \int_r^s \partial_\theta K_H(\theta,r) 
~(\varphi(\theta)-\varphi(r))~d\theta. $$

We now fix a time horizon $T$. According to our notational 
convention~\ref{convention-of-writing}, when no risk of confusion is 
possible we set $K_H^*\equiv K_{H,T}^*$. Denote by $\HH$ the 
Hilbert space defined as the closure of the space of step functions 
w.r.t. the scalar product
\begin{equation*}
\langle \varphi, \psi \rangle_{\HH} = \langle K_H^*\varphi, 
K_H^*\psi\rangle_{L^2([0,T])}.
\end{equation*}
This extension of $K_H^*$ as an isometric operator from $\HH$ to 
$L^2([0,T])$ is also denoted by $K_H^*$. In particular, we have
\begin{equation*}
\langle \indi{[0,s]}, \indi{[0,t]} \rangle_{\HH} = R_H(s,t).
\end{equation*}
Note that $B^H \in \HH$ iff $H>\tfrac{1}{4}$ 
(see Nualart~\cite[p.301]{Nualart}).

A natural subspace of $\HH$ will be used in the sequel: $|\HH|$ is 
the Banach space of measurable functions $\varphi$ on $[0,T]$ such that
\begin{itemize}
\item if $H<\frac{1}{2}$,
\begin{equation}\label{eq:defAbsH<}
\|\varphi\|_{|\HH|}^2 := \int_0^T \varphi_t^2~K_H(T,t)^2~dt
+ \int_0^T \left(\int_t^T |\varphi_s-\varphi_t| 
(s-t)^{H-\frac{3}{2}}\ ds \right)^2dt <\infty,
\end{equation}
\item If $H>\frac{1}{2}$,
\begin{equation}\label{eq:defAbsH}
\|\varphi\|_{|\HH|}^2 := \alpha_H \int_0^T \int_0^T 
|\varphi_s|\ |\varphi_t|\ |s-t|^{2H-2}\ ds~dt <\infty,
\end{equation}
\end{itemize}
where
\begin{equation} \label{def:alpha_H}
\alpha_H := 2H~(H-\tfrac{1}{2}).
\end{equation}

\paragraph{Useful properties of $K_H^*$.}
Below we will use the following properties of $K_H^*$ acting on $|\HH|$.

We show in Appendix~\ref{app:KH-etoile} that the following 
extension of 
$K_H^*$ from step functions to the space $|\HH|$ is well defined:
\begin{equation} \label{eq:defKH-avec-derivee}
\forall\varphi\in|\HH|,~~K_H^*\varphi(r)
:= K_H(T,r)~\varphi(r) 
+ \int_r^T \partial_\theta K_H(\theta,r) 
~(\varphi(\theta)-\varphi(r))~d\theta.
\end{equation}

One easily deduces the following from~\eqref{eq:defKH-avec-derivee}:
\begin{equation}\label{eq:KHbis}
\begin{cases}
& \text{For any}~0<t\leq T~~\text{and}~~
\varphi\in |\HH| \text{ such that } \varphi(\theta) = 0 
\text{ when 
}\theta>t,~\text{one has:} ~~\forall r<t\leq T, \\
&K_H^*\varphi(r)= K_H(t,r)~\varphi(r) + 
\int_r^t
\partial_\theta K_H(\theta,r)
~(\varphi(\theta)-\varphi(r))~d\theta \\
&\phantom{K_H^*\varphi(r)}=K_{H,t}^*\varphi(r).
\end{cases}
\end{equation}

When $H>\tfrac{1}{2}$ (and thus $H-\frac{3}{2}>-1$) and 
when $\varphi$ is in $|\HH|$, 
it also comes from~(\ref{eq:defKH-avec-derivee})  that 
\begin{equation}\label{eq:defKHgeq12}
K_H^*\varphi(r)= \int_r^{T} 
\partial_\theta K_H(\theta,r)\ \varphi(\theta)~d\theta.
\end{equation}

Finally, when $H\rightarrow \tfrac{1}{2}$, $\chi_H$ tends to~1 and thus 
$\frac{\partial}{\partial \theta} K_H(\theta,\sigma)$ converges in 
the distributional sense to the Dirac measure at point $\sigma$. 
Therefore, for any $\varphi\in|\HH|$ and $0<r<T$, 
$K_H^*\varphi(r)$ tends to $\varphi(r)$ 
when $H\rightarrow \tfrac{1}{2}$.

\paragraph{Representation of fBm as non-anticipating stochastic 
integrals.}

From the equality~\eqref{eq:decompCov} one can deduce the following 
representation of the fBm $B^H$: for some 
standard Brownian motion $\mathbf{B}\equiv B^{1/2}$,
\begin{equation} \label{BH-B}
\forall t\geq0,~~B^H_t = \int_0^t K_H(t,u)~d\mathbf{B}_u.
\end{equation}

\paragraph{Malliavin calculus for frac{t}ional Brownian motion.} 

We are given a fBm~$B^H$ and the corresponding Brownian 
motion~$\mathbf B$ as in~\eqref{BH-B}. Similarly to the Malliavin derivative 
$\mathbf{D}$ associated to the Brownian motion~$\mathbf{B}$, 
the Malliavin derivative $D^H$ is defined as an operator acting on the 
smooth random variables with values in $\HH$. The domain of $D^H$ in 
$L^p(\Omega)$ ($p>1$) is denoted by $\mathbb{D}^{1,p}$ and is the 
closure of the space of smooth random variables with respect to the norm
\begin{equation*}
\left\{\EE(|F|^p) 
+ \EE\left(\|D^H F\|_{\HH}^p\right)\right\}^{\frac{1}{p}}.
\end{equation*}
Equivalently (cf \cite[p.288]{Nualart}), $D^H$ can be defined as 
\begin{equation}\label{eq:deriveeMalliavin}
D^H = (K_H^*)^{-1}\mathbf{D}.
\end{equation}
In particular, for all $s,t\in[0,T]$ we have that
\begin{equation*}
D^H_s B^H_t=  \indi{[0,t]}(s)\ \text{ and }~\mathbf{D}_s B^H_t = 
K_H^*\left(D_\cdot^H B^H_t\right)(s) = 
K_H^*(\indi{[0,t]}(\cdot))(s) = K_H(t,s).
\end{equation*}

We denote by $\mathbb{D}^{1,2}(|\HH|)$ the set of the $|\HH|$-valued 
random variables such that
\begin{equation}\label{eq:defD12|H|}
\EE\|\xi\|^2_{|\HH|} 
+ \EE \int_0^T\|\mathbf{D}_r \xi_\sbullet\|_{|\HH|}^2~dr 
< \infty  \quad \text{if } H<\tfrac{1}{2},
\end{equation}
and 
\begin{equation}\label{eq:defD12|H|+}
\EE\|\xi\|^2_{|\HH|} 
+ \EE 
\int_{[0,T]^4} |D^H_r \xi_\theta|~|D^H_s \xi_\eta|
~|s-r|^{2H-2}~|\theta-\eta|^{2H-2}~dr~ds~d\theta~d\eta
< \infty  \quad \text{if } H>\tfrac{1}{2}.
\end{equation}
See \cite[Sec.3]{AlosMazetNualart} and \cite[p.295]{Nualart} when 
$H<\tfrac{1}{2}$ and \cite[p.288]{Nualart} when 
$H>\tfrac{1}{2}$.

The divergence operator or Skorokhod integral $\delta_H$ is defined by 
the following duality relation: for any~$F$ in $\mathbb{D}^{1,2}$ and 
any~$\xi$ in the domain $\text{dom}(\delta_H)\subset  L^2(\Omega,\HH)$ 
of~$\delta_H$, one has
$$\EE\left(\langle \xi, D^H 
F\rangle_{\HH}\right) = \EE\left(F~\delta_H(\xi)\right). $$ 
The Skorokhod integral $\delta_H$ is related to the ordinary 
Skorokhod integral $\bm{\delta}$ w.r.t. the Brownian motion $B$ as 
follows: 
for any $\xi$ such that $K_H^* \xi\in \text{dom}(\bm{\delta})$,
\begin{equation*}
\delta_H(\xi) = \bm{\delta}(K_H^* \xi).
\end{equation*}
It can be shown that $\text{dom}(\delta_H) = (K_H^*)^{-1}(\text{dom} 
(\bm{\delta}))$ and that~$\text{dom} (\delta_H)$ contains 
$\mathbb{D}^{1,2}(|\HH|)$ (see Nualart~\cite[Sec.5.2.2 and 
p.295]{Nualart} and references therein). 

We again emphasise that the preceding operators implicitly depend 
on~$T$. According to our notational 
convention~\ref{convention-of-writing} we will write $\delta_H^{(T)}$
rather than $\delta_H$ when it is necessary to take care of that 
dependency.

\subsection{Solutions to the 
SDE~(\ref{eq:fSDE_0})}\label{subsec:SDE_Ito}

Consider the SDE~(\ref{eq:fSDE_0}) in the Stratonovich sense under our 
hypotheses \ref{hyp:h1'} and \ref{hyp:h2'} on $b$ and $\sigma$.

For $H>\tfrac{1}{2}$, we consider the unique solutions in 
the sense of~\citet{Young} which are studied in 
\citet{NualartRascanu}. They are based on the generalized Stieltjes 
integrals defined in~\citet{Zaehle}. They 
coincide with Stratonovich solutions since the regularity conditions 
for Stratonovich and Young integrals to coincide are met in our context.
Their sample paths are H\"older continuous with
order~$H-\epsilon$ for any $0<\epsilon<H$.

For $H\in(\tfrac{1}{4},\tfrac{1}{2})$ we deal with the notion of 
Stratonovich solution studied by Al\`os et al~\cite{AlosLeonNualart}. 
In~\cite[Prop.6]{AlosLeonNualart} 
it is shown that for $b\in \Cb^1(\R)$ and 
$\sigma\in\Cb^2(\R)$ there exists a pathwise unique solution 
to~(\ref{eq:fSDE_0}). This Stratonovich solution admits the 
Doss-Sussman 
representation:
\begin{equation*}
\begin{cases}
X^H_t &= \alpha(B^H_t,Z^H_t), \\
Z^H_t &=  x_0 + \int_0^t b\circ\alpha(B^H_s,Z^H_s) 
\exp\left(-\int_0^{B^H_s}\sigma'\circ\alpha(u,Z^H_s) du\right) ds,
\end{cases}
\end{equation*}
where $\alpha(x,z)$ solves
\begin{equation} \label{def:alpha}
\begin{cases}
\frac{\partial \alpha}{\partial x}(x,z) &= \sigma\circ\alpha(x,z), \\
\alpha(0,z) &= z.
\end{cases}
\end{equation}
The uniqueness results from~\citet[Lem.2]{Doss}.

In both cases, $H>\tfrac{1}{2}$ and $H\in(\tfrac{1}{4},\tfrac{1}{2})$,
we will need to apply an It\^o type formula to processes of the type 
$(\Phi(t,X^H_t))$, where $\Phi$ is a smooth function. For reasons which 
will be apparent in the sequel, we need that the formula involves 
stochastic integrals with zero expectation, and thus Skorokhod 
integrals rather than Stratonovich integrals. 

A natural approach would consist in extending previous works, namely,
the It\^o-Skorokhod formula in~\cite[Thm.8]{AlosNualart}
for $H>\tfrac{1}{2}$ and the It\^o-Stratonovich formula 
in~\cite[Thm.4]{AlosLeonNualart} for $\frac{1}{4}<H<\tfrac{1}{2}$. We 
would have to strenghten our 
hypotheses to ensure that the process $\sigma(X^H_t)$ belongs 
to~$\mathbb{D}^{2,2}(|\HH|)$, however. We also would have to develop 
heavy calculations to get needed estimates on $\EE\sup_{s\leq 
t}|X^H_s|$ and $\EE|D^H_r X^H_t|^p$. In addition, the It\^o-Skorokhod 
formula would involve integrals with algebraically complex integrands.

For all these reasons, we follow another way which is allowed by the 
ellipticity condition~\ref{hyp:h2'}. 

First, we show that $X^H$ is a one-to-one 
transform of a
drifted fractional Brownian motion for any~$H\in(\frac{1}{4},1)$. That 
amounts to prove an 
It\^o formula for a specific smooth map, namely, the Lamperti 
transform. The formula is easy to prove since the dynamics of the 
transformed 
process does not involve any stochastic integral.

Second, we establish an It\^o-Skorokhod formula for general functions 
of time and drifted fractional Brownian motions. We
here 
benefit from the fact that the dynamics of the process under 
consideration does not involve a stochastic integral.

\subsection{The Lamperti process $Y^H$}\label{subsec:Lamperti}
In this section we show that a one-to-one transform of~$X^H$ is 
a fractional diffusion~$Y^H$ with constant diffusion coefficient, and 
we prove regularity properties of~$Y^H$.

\begin{proposition}\label{prop:Lamperti}
Let $H\in (\tfrac{1}{4},1)$. Assume that $b$ and $\sigma$ satisfy the 
hypotheses~\ref{hyp:h1'} and~\ref{hyp:h2'}. Let $F(x) := \int_0^x 
\frac{1}{\sigma(z)}\ dz$ be the Lamperti transform. Set 
$\widetilde{b} := \frac{b\circ F^{-1}}{\sigma\circ F^{-1}}$. 

Then the process $Y^H := F(X^H)$ is the unique pathwise solution 
to the following SDE:
\begin{equation}\label{eq:fSDE_Lamperti}
\forall t\geq 0,\quad Y^H_t = F(x_0) + B_t^H + \int_0^t 
\widetilde{b}(Y^H_s)\ ds.
\end{equation}
\end{proposition}

\begin{proof}
When $H>\tfrac{1}{2}$: The desired result is obtained by means of the 
classical chain rule since the Stratonovich integral coincides with a 
Stieltjes integral.

When $H\in(\tfrac{1}{4},\tfrac{1}{2})$:
To prove (\ref{eq:fSDE_Lamperti}), fix an arbitrary time horizon $T>0$ 
and for any $\epsilon>0$ consider the regularised process 
\begin{equation} \label{def:BH-regularise}
B^{H,\epsilon}_t = \frac{1}{2\epsilon} \int_0^t 
\left(B^H_{(s+\epsilon)\wedge T} 
- B^H_{(s-\epsilon)\vee 0}\right) ds.
\end{equation}
Set also $X^{H,\epsilon}_t = \alpha(B^{H,\epsilon}_t,Z_t)$, where 
$\alpha$ is defined by~(\ref{def:alpha}). Then the 
usual chain rule leads to
\begin{align*}
F(X^{H,\epsilon}_t) &= F(x_0) + \frac{1}{2\epsilon} 
\int_0^t \partial_x (F\circ\alpha)(B^{H,\epsilon}_s,Z_s) \times 
\left(B^H_{(s+\epsilon)\wedge T} - B^H_{(s-\epsilon)\vee 0}\right)\ d
s \\
&\hspace{1cm} + \int_0^t \partial_z 
(F\circ\alpha)(B^{H,\epsilon}_s,Z_s)\ b\circ\alpha(B^H_s,Z^H_s) 
\exp\left(-\int_0^{B^H_s} \sigma'\circ\alpha(z,Z^H_s) dz\right) d
s.
\end{align*}
The definition of $\alpha$ implies that $\partial_x (F\circ\alpha)=1$ 
and  $\partial_z \alpha(x,z) = \exp\left(\int_0^{x} 
\sigma'\circ\alpha(u,z) du\right)$. Thus 
\begin{align*}
F(X^{H,\epsilon}_t) &= F(x_0) + \frac{1}{2\epsilon}  \int_0^t 
\big(B^H_{(s+\epsilon)\wedge T} - B^H_{(s-\epsilon)\vee 0}\big)~ds 
+ \int_0^t \frac{ \frac{\partial\alpha}{\partial z} 
(B^{H,\epsilon}_s,Z_s)}{\frac{\partial\alpha}{\partial z} 
(B^H_s,Z_s)} \times 
\frac{b\circ\alpha(B^H_s,Z^H_s)}
{\sigma\circ\alpha(B^{H,\epsilon}_s,Z^H_s)}\ ds.
\end{align*}
As $B^H_0=0$ one can readily show that $\frac{1}{2\epsilon}  
\int_0^\theta \big(B^H_{(s+\epsilon)\wedge T} - B^H_{(s-\epsilon)\vee 
0}\big)~ds$ almost surely converges to $B^H_\theta$ when $\epsilon$ 
tends to~0 and this convergence is uniform on~$[0,T]$. 
The almost sure convergence of each side of the preceding equality 
yields~Equation (\ref{eq:fSDE_Lamperti}). Pathwise uniqueness of the 
solution
results from the Lipschitz property of~$\widetilde{b}$: See 
Subsection~\ref{subsec:SDE_Ito}.
\end{proof}

\paragraph{Properties of the Lamperti process $Y^H$.} 
We now state useful estimates on $Y^H$ and its 
Malliavin derivatives. We use the following representation which is 
valid for any $H\in (\tfrac{1}{4},1)$ (see~\citet{NualartSaussereau}):
\begin{equation}\label{eq:DrXH}
\begin{cases}
\forall r>t,~~D^H_r Y^H_t = 0, \\
\forall r\leq t,~~D^H_r Y^H_t 
= 1+\int_r^t D^H_r Y^H_u\ \widetilde{b}'(Y^H_u)~du.
\end{cases}
\end{equation}
From~\eqref{eq:DrXH} one readily gets
\begin{equation}\label{eq:DY}
\forall r>0,~\forall t>0,~~~
D_r^H Y_t^H = \indi{[0,t]}(r) 
\exp\left(\int_r^t \widetilde{b}'(Y^H_u)~du\right).
\end{equation}

The following proposition is an obvious consequence of 
\eqref{eq:fSDE_Lamperti} and \eqref{eq:DY}.
\begin{proposition}\label{lem:HolderFSDE}
Let $b$ and $\sigma$ satisfy hypotheses~\ref{hyp:h1'}-\ref{hyp:h2'}. 
It a.s. holds that
\begin{align}
& 0\leq D_r^H Y_t^H
\leq \indi{[0,t]}(r)~e^{|\widetilde{b}'|_\infty (t-r)}
\label{ineq:DrYH} \\
& \forall 0\leq r\leq s\leq t,~|D^H_r Y^H_t - D^H_r Y^H_s|
\leq e^{|\widetilde{b}'|_\infty (t-r)}~|\widetilde{b}'|_\infty~(t-s).
\label{ineq:DrYHt-DrYHs}\\
& \forall 0\leq r\leq r'\leq t,~ |D^H_r Y^H_t - D^H_{r'} Y^H_t|
\leq e^{|\widetilde{b}'|_\infty (r'-r)}\, |\widetilde{b}'|_\infty~(r'-r) .
\label{ineq:DrYHt-DrprimeYHt}
\end{align}
\end{proposition}

In addition to the preceding estimates on~$D^H_\sbullet Y^H_\sbullet$
we will need accurate estimates on~$\mathbf{D}_\sbullet Y^H_\sbullet$. 
The next proposition provides two 
such useful estimates. The upper bounds are expressed in terms of
the kernel~$K_H$ because, in the sequel, we will either use pointwise 
estimates on $K_H$ or the $L^2(0,T)$ 
properties~\eqref{eq:decompCov} and~\eqref{eq-delta-KH-carre} 
(see e.g. the proof of~Proposition~\ref{lem:YHinD12} and the 
calculations in Appendix \ref{App:tecLem}).
Notice that~\eqref{eq:deriveeMalliavin} and~\eqref{eq:KHbis} imply
\begin{equation} \label{eq:deriv-Malliavin-Y^H}
\forall 0<r\leq t,~~
\mathbf{D}_r Y^H_t = K_H^*\left(D^H_\sbullet Y^H_t\right)(r)
= K_{H,t}^*\left(D^H_\sbullet Y^H_t\right)(r).
\end{equation}

\begin{proposition}\label{lem:DYH}
Let $H\in(0,1)\setminus \{\tfrac{1}{2}\}$. One then has
\begin{equation} \label{D_rY}
\begin{cases}
\forall r>t,~ |\mathbf{D}_r Y^H_t| = 0, \\
\forall r\leq t,~|\mathbf{D}_r Y^H_t| 
\leq C~e^{|\widetilde{b}'|_\infty(t-r)} \left\{  |K_H(t,r)|  
+ (t-r)^{H+\frac{1}{2}} 
~\indi{\{H<\frac{1}{2}\}} \right\}.
\end{cases}
\end{equation}
In addition, for any $r\leq s < t$ it holds that
\begin{multline} \label{ineq:diff-D_rY}
|\mathbf{D}_r Y^H_t - \mathbf{D}_r Y^H_s| \\
\leq C~e^{|\widetilde{b}'|_\infty(t-r)} 
\left\{ |K_H(t,r) - K_H(s,r)|
+ (t-s) \left(|K_H(s,r)|+(s-r)^{H+\frac{1}{2}} 
~\indi{\{H<\frac{1}{2}\}}\right)\right\}.
\end{multline}
\end{proposition}

\begin{proof}
\textbf{(A)} To prove~\eqref{D_rY}
we successively examine the cases $H>\tfrac{1}{2}$ 
and~$H<\tfrac{1}{2}$.

\vspace{0.2cm}

\underline{The case $H>\tfrac{1}{2}$.} 

For $r>t$ we deduce~\eqref{D_rY} from~\eqref{eq:deriv-Malliavin-Y^H} 
and~\eqref{eq:DY}.

For $r\leq t$ we start with using~\eqref{eq:deriv-Malliavin-Y^H} 
and~\eqref{eq:defKHgeq12} to get
$$ |\mathbf{D}_r Y^H_t|
= \left|\int_r^t \partial_\theta K_H (\theta,r) 
~D^H_\theta Y^H_td\theta \right|. $$
Observe that~\eqref{eq:deriv-K_H} implies that
$K_H$ and $\partial_\theta K_H$ are non-negative when 
$H>\tfrac{1}{2}$. By using~\eqref{ineq:DrYH} we deduce from the 
preceding that
$$ |\mathbf{D}_r Y^H_t| \leq e^{|\widetilde{b}'|_\infty (t-r)} \int_r^t 
\partial_\theta K_H (\theta,r)~d\theta
\leq e^{|\widetilde{b}'|_\infty (t-r)}~K_H(t,r), $$
which is the desired inequality.

\underline{The case $H<\tfrac{1}{2}$.} 

For $r>t$ the inequality~\eqref{D_rY} follows from~\eqref{eq:DY}.

For $r\leq t$, we use~\eqref{eq:deriv-Malliavin-Y^H} to get
\begin{equation*}
|\mathbf{D}_r Y^H_t|
= \left| K_H(t,r) D^H_r Y^H_t + \chi_H~(H-\tfrac{1}{2}) \int_r^t 
\left(\frac{\theta}{r}\right)^{H-\frac{1}{2}} 
(\theta-r)^{H-\frac{3}{2}}  \left(D^H_\theta Y^H_t - D^H_r Y^H_t 
\right) d\theta \right|.
\end{equation*}
In view of~\eqref{eq:DY} and \eqref{ineq:DrYHt-DrprimeYHt}
one thus has
\begin{align*}
|\mathbf{D}_r Y^H_t|&\leq e^{|\widetilde{b}'|_\infty (t-r)} 
\Big\{|K_H(t,r)| + \chi_H 
~|\widetilde{b}'|_\infty~|H-\tfrac{1}{2}| \int_r^t 
\left(\frac{\theta}{r}\right)^{H-\frac{1}{2}} 
(\theta-r)^{H-\frac{1}{2}}~d\theta \Big\}.
\end{align*}
For $0<r<\theta$ and $H<\frac{1}{2}$ one has 
$\left(\frac{\theta}{r}\right)^{H-\frac{1}{2}}<1$. 
The inequality~\eqref{D_rY} follows.

\vspace{1cm}

\textbf{(B)}
 To prove~\eqref{ineq:diff-D_rY}, let $r\leq s <t$. We again 
successively examine the cases $H>\tfrac{1}{2}$ and~$H<\tfrac{1}{2}$.

\vspace{0.2cm}

\underline{The case $H>\tfrac{1}{2}$.} 

We use~\eqref{eq:deriv-Malliavin-Y^H} and~\eqref{eq:defKHgeq12} to get
\begin{equation*}\label{eq:DY^H-H-gt-1/2}
\begin{split}
\mathbf{D}_r Y^H_t - \mathbf{D}_r Y^H_s &= K_H^*\left(D^H_\sbullet 
Y^H_t - D^H_\sbullet Y^H_s\right)(r) \\
&= \int_r^t \partial_\theta K_H(\theta,r) 
\left(D^H_\theta Y^H_t - D^H_\theta Y^H_s\right)~d\theta \\
&= \int_r^s \partial_\theta K_H (\theta,r) 
\left(D^H_\theta Y^H_t - D^H_\theta Y^H_s\right) d\theta + 
\int_s^t \partial_\theta K_H (\theta,r) 
D^H_\theta Y^H_t d\theta.
\end{split}
\end{equation*}
We now combine~\eqref{ineq:DrYHt-DrYHs}, ~\eqref{ineq:DrYH} and 
the non-negativity of $\partial_\theta K_H$ when 
$H>\frac{1}{2}$. It comes:
\begin{equation*}
\begin{split}
|\mathbf{D}_r Y^H_t - \mathbf{D}_r Y^H_s|
&\leq  e^{|\widetilde{b}'|_\infty(t-r)}  
\left\{|\widetilde{b}'|_\infty (t-s)\int_r^s 
\partial_\theta K_H (\theta,r)~d\theta + \int_s^t 
\partial_\theta K_H (\theta,r)~d\theta\right\} \\
&\leq C~e^{|\widetilde{b}'|_\infty(t-r)}  
\left\{ (t-s)~K_H(s,r) 
+ K_H(t,r) - K_H(s,r) \right\}.
\end{split}
\end{equation*}
We thus have obtained~\eqref{ineq:diff-D_rY} when $H>\frac{1}{2}$.

\underline{The case $H<\tfrac{1}{2}$.} 

In view of~\eqref{eq:deriv-Malliavin-Y^H} one has
\begin{align} \label{eq:DY^H}
\mathbf{D}_r Y^H_t - \mathbf{D}_r Y^H_s &= K_H^*\left(D^H_\sbullet 
Y^H_t - D^H_\sbullet Y^H_s\right)(r) \nonumber \\
&= K_H(t,r) \left(D^H_r Y^H_t - D^H_r Y^H_s\right) \nonumber \\
&~~~~~~+ \int_r^s 
\partial_\theta K_H(\theta,r) \left(D^H_\theta Y^H_t 
- D^H_\theta Y^H_s - (D^H_r Y^H_t - D^H_r Y^H_s)\right) d\theta 
\nonumber \\
&~~~~~~+ \int_s^t
\partial_\theta K_H(\theta,r) \left(D^H_\theta Y^H_t 
- (D^H_r Y^H_t - D^H_r Y^H_s)\right) d\theta \nonumber \\
&= \int_r^s 
\partial_\theta K_H(\theta,r) \left(D^H_\theta Y^H_t 
- D^H_\theta Y^H_s - (D^H_r Y^H_t - D^H_r Y^H_s)\right) d\theta 
\nonumber \\
&~~~~~~+ \int_s^t \partial_\theta K_H(\theta,r) 
D^H_\theta Y^H_t~d\theta + K_H(s,r) \left(D^H_r Y^H_t - D^H_r 
Y^H_s\right) \nonumber \\
&=: A_1 + A_2 + A_3.
\end{align}

Use \eqref{eq:DY} and apply the Mean Value theorem to the map
$$ v\in[r,\theta] \mapsto 
\exp\big(\int_v^t \widetilde{b}'(Y^H_u)~du\big)
- \exp\big(\int_v^s \widetilde{b}'(Y^H_u)~du\big). $$
For $r<\theta<s$ it comes
$$ |D^H_\theta Y^H_t - D^H_\theta Y^H_s - (D^H_r Y^H_t - D^H_r 
Y^H_s)| 
\leq |\widetilde{b}'|_\infty^2~e^{|\widetilde{b}'|_\infty (t-r)} 
(\theta - r)(t-s). $$
By successively using~\eqref{eq:deriv-K_H} and $H<\frac{1}{2}$ we 
obtain 
\begin{align} \label{ineq:A1}
|A_1| &\leq
C~e^{|\widetilde{b}'|_\infty (t-r)}~(t-s)\int_r^s 
\Big(\frac{\theta}{r}\Big)^{H-\frac{1}{2}} 
(\theta-r)^{H-\frac{1}{2}}~d\theta \nonumber \\
&\leq C~e^{|\widetilde{b}'|_\infty (t-r)}~(t-s) \int_r^s 
(\theta-r)^{H-\frac{1}{2}}~d\theta \nonumber \\
&\leq
C~e^{|\widetilde{b}'|_\infty (t-r)}~(t-s)~(s-r)^{H+\frac{1}{2}}.
\end{align}

We now estimate $|A_2|$. The equality $\eqref{eq:deriv-K_H}$ shows that 
$\partial_{\theta} K_H(\theta,r) \leq 0$ when 
$H<\frac{1}{2}$. Therefore,
\begin{equation} \label{ineq:A2}
|A_2| \leq e^{|\widetilde{b}'|_\infty (t-r)}~
\int_s^t (-\partial_\theta K_H(\theta,r))~d\theta
= e^{|\widetilde{b}'|_\infty (t-r)}~|K_H(t,r) - K_H(s,r)|.
\end{equation}

We finally consider $|A_3|$. In view of~\eqref{ineq:DrYHt-DrYHs} we have
\begin{equation} \label{ineq:A3}
|A_3| \leq e^{|\widetilde{b}'|_\infty 
(t-r)}~|\widetilde{b}'|_\infty~(t-s)~|K_H(s,r)|. 
\end{equation}

It now remains to combine~\eqref{eq:DY^H}, \eqref{ineq:A1},
\eqref{ineq:A2} and \eqref{ineq:A3}. We deduce~\eqref{ineq:diff-D_rY}
for $H<\frac{1}{2}$.
\end{proof}

\begin{remark}\label{rk:boundb'}
In Section~\ref{sec:examples} we reinforce our hypotheses on the drift $\widetilde{b}$ and obtain weak convergence rate estimates on the law of~ $\tau^H_X$ when $H$ tends to $\frac{1}{2}$. 
To carry out this program we will need to deal with bounds from above on moments of
$\mathbf{D}_r Y^H_t $ and of its time increments rather than the univeral a.s. bounds in Lemma~\ref{lem:DYH}.

In the preceding proof,  we used an a.s. bound on
$$ \sup_{\theta\in (r,t)} \exp\left( \int_{\theta}^t \widetilde{b}'(Y^H_{u})\, du \right). $$
In addition,
\begin{itemize}
\item For $H>\tfrac{1}{2}$  we used an a.s. bound on
$$ \sup_{\theta\in (r,s)} \frac{1}{t-s} \Big( \exp\big( \int_{\theta}^t \widetilde{b}'(Y^H_{u})\, du \big) - \exp\big( \int_{\theta}^s \widetilde{b}'(Y^H_{u})\, du \big) \Big) $$
to prove~\eqref{ineq:diff-D_rY}.
\item For $H<\tfrac{1}{2}$ we used a.s. bounds 
$$ \sup_{\theta\in (r,t)} \frac{1}{\theta-r} \Big( \exp\big( \int_{\theta}^t \widetilde{b}'(Y^H_{u})\, du \big) - \exp\big( \int_{r}^t \widetilde{b}'(Y^H_{u})\, du \big) \Big) $$
and
\begin{align*}
\sup_{\theta\in (r,s)} \frac{1}{(\theta-r)(t-s)} \Big( &\exp\big( \int_{\theta}^t \widetilde{b}'(Y^H_{u})\, du \big) - \exp\big( \int_{\theta}^s \widetilde{b}'(Y^H_{u})\, du \big) \\
 &- \exp\big( \int_{r}^t \widetilde{b}'(Y^H_{u})\, du \big) 
+ \exp\big( \int_{r}^s \widetilde{b}'(Y^H_{u})\, du \big)\Big) 
\end{align*}
to prove~\eqref{D_rY} and~\eqref{ineq:diff-D_rY} respectively.
\end{itemize}

Using Taylor expansions, all these quantities are bounded by the following function defined for any $r\leq t$:
\begin{align}\label{eq:varpiH}
\varpi_{H}(r,t) &:= (1\vee |\widetilde{b}'|_\infty^2)~\sup_{r<\theta<s<t} D^H_{\theta} Y^H_{s} = (1\vee |\widetilde{b}'|_\infty^2)~\sup_{r<\theta<s<t} \exp\big( \int_{\theta}^s \widetilde{b}'(Y^H_{u})\, du \big) .
\end{align}
The calculations in the preceding proof show that 
\begin{align}
& \forall r\leq t,~|\mathbf{D}_r Y^H_t| 
\leq C\left\{  |K_H(t,r)|  
+ (t-r)^{H+\frac{1}{2}} 
~\indi{\{H<\frac{1}{2}\}} \right\} \varpi_{H}(r,t) \tag{\ref{D_rY}'}
\label{D_rYneg}\\
&\forall r\leq s < t, ~ |\mathbf{D}_r Y^H_t - \mathbf{D}_r Y^H_s| \nonumber\\
&\hspace{1cm}\leq C \left\{ |K_H(t,r) - K_H(s,r)|
+ (t-s) \left(|K_H(s,r)|+(s-r)^{H+\frac{1}{2}} 
~\indi{\{H<\frac{1}{2}\}}\right)\right\} \varpi_{H}(r,t). \tag{\ref{ineq:diff-D_rY}'}
\label{ineq:diff-D_rYneg}
\end{align}
\end{remark}

\subsection{An Itô-Skorokhod formula for fractional Brownian motions
with drift}

In this subsection, we prove an It\^o-Skorokhod formula for processes 
of the form
\begin{align}\label{eq:defXcal}
\Yc_t = y_0 + B^H_t + \int_0^t \beta_s ~ds, \quad t\in[0,T],
\end{align}
where $(\beta_s,~s\in[0,T])$ is a smooth enough 
stochastic process. In our next section we will check that
the formula~\eqref{eq:ItoSkorokhodFormula0} below applies to the
solution~$Y^H$ of the SDE~\eqref{eq:fSDE_Lamperti}. 
It involves the Trace term~\eqref{eq:formula-Tr-BM-drift} which is 
related to the conversion formula from Stratonovich integrals 
w.r.t.~$B^H$ to Skorokhod integrals (see e.g.~\cite{AlosLeonNualart}).
However, as explained in Section~\ref{subsec:SDE_Ito}, 
Stratonovich integrals are useless to our purpose.

\begin{proposition}\label{prop:ItoSkorokhod}
Let $(\Yc_t)_{t\in[0,T]}$ be a process of the form 
\eqref{eq:defXcal}. Assume that $\beta$ is progressively measurable 
w.r.t. the Brownian 
filtration generated by $\mathbf{B}$. Suppose also that for every 
$0\leq s\leq T$ the random variable $\beta_s$ belongs 
to~$\mathbb{D}^{1,2}$ and the process $(\int_{0}^t \beta_s \, 
ds,~t\in[0,T])$ belongs to~$\mathbb{D}^{1,2}(|\HH|)$.
Assume also that for any $0\leq t,r \leq T$,
\begin{equation}\label{eq:AssumptionDrBeta}
\begin{cases}
\mbox{ if } H>\frac{1}{2},
&\left(\int_0^t \EE|\mathbf{D}_{r} \beta_s|^2 ~ds\right)^\frac{1}{2} 
\leq C~|t-r|^\alpha~~~\text{for some}~~\alpha>\frac{1}{2}-H, \\
\mbox{ if } H<\frac{1}{2}, 
&\left(\int_0^t \EE|\mathbf{D}_{r} \beta_s|^2 ~ds\right)^\frac{1}{2}
\leq C.
\end{cases}
\end{equation}
Then, for every 
$H\in(\tfrac{1}{4},1)$, the process
$\mathcal{Y}^H$ belongs to~$\textrm{dom}(\delta_H)$
and for all $G\in\mathcal{C}^{1,2}_b([0,T]\times \R)$ and 
$0\leq t\leq T$ one has
\begin{equation} \label{eq:ItoSkorokhodFormula0}
\begin{split}
G(t,\Yc_t) = & ~G(0,y_0) 
+ \int_0^t \left(\partial_s G(s,\Yc_s) 
+ \partial_y G(s,\Yc_s)~\beta_s\right) d s \\
&+ \delta_H^{(T)}\left(\indi{[0,t]}~(\cdot)~\partial_y
G(\cdot,\Yc_\cdot)~\right) 
+ \Tr\Big[ D^H \partial_y G(\cdot,\Yc_\cdot)\Big]_t,
\end{split}
\end{equation} 
where
\begin{equation} \label{eq:formula-Tr-BM-drift}
\Tr\left[D^H \partial_y G(\cdot,\Yc_\cdot)\right]_t 
:= \int_0^t \partial^2_yG(s,\Yc_s)
\Big( Hs^{2H-1} + \int_0^s 
\partial_s K_H(s,r)\int_0^s \mathbf{D}_r \beta_v~dv~dr\Big) ds.
\end{equation}
\end{proposition}

\begin{proof}
The lengthy proof is divided in several steps. After having checked a 
preliminary result, in Step 1 we derive an It\^o formula pour smooth 
functions of a semi-martingale~$\mathcal{Y}^{H,\epsilon}$ which 
approximates~$\mathcal{Y}^H$. In Steps~2 and 3 we successively prove 
the convergence of each term which arises in the It\^o formula for
$\mathcal{Y}^{H,\epsilon}$.

\vspace{0.3cm} 

\underline{\emph{Preliminary: A stochastic Fubini equality.}}
We start with proving a stochastic Fubini equality. Let the process 
$u(\cdot,\cdot)$ 
be such that
\begin{equation} \label{ineq:D-1-2}
\EE\left[ \int_{0}^T \int_{0}^T 
| u(r,s)|^2\, ds\, dr  \right]
+ \EE\left[ \int_{0}^T \int_{0}^T \int_{0}^T |\mathbf{D}_{\theta} 
u(r,s)|^2\, d\theta\, ds\, dr  \right] < \infty.
\end{equation}
We use the notational convention~\ref{convention-of-writing}
to define the operator~$\bm{\delta}^{(T)}$. Let us check that
\begin{align}\label{eq:StocFubini-denis}
\bm{\delta}^{(T)}\Big(\int_0^T u(r,\cdot)~dr\Big)
= \int_0^T \bm{\delta}^{(T)}(u(r,\cdot))~dr.
\end{align} 
Indeed, for any $F\in \mathbb{D}^{1,2}$ one has
\begin{align*} 
\EE \left[ F~\bm{\delta}^{(T)}\Big(\int_0^T u(r,\cdot)~dr\Big) \right] 
&= \EE\left[ \int_0^T  \mathbf{D}_s F~\int_0^T u(r,s)~
dr ~ ds  \right] \nonumber \\
&= \int_0^T \EE\left[ \int_0^T  \mathbf{D}_s F~  u(r,s)~ds  
\right]~dr \nonumber \\
&=\int_0^T \EE\left[ F~\bm{\delta}^{(T)}~(u(r,\cdot))\right]~dr 
\nonumber \\
&= \EE\left[ F~\int_0^T\bm{\delta}^{(T)}(u(r,\cdot))~dr\right].
\end{align*}
In the preceding calculation, we used the classical Fubini Theorem in 
the second and fourth line, and the duality between the Skorokhod 
integral and the derivative operator in the third one.

We now proceed to the proof of~\eqref{eq:ItoSkorokhodFormula0}.

\vspace{0.3cm}
\underline{\emph{Step $1$: An It\^o formula for an approximation of 
$\mathcal{Y}^H$.}}

Recall the representation~\eqref{BH-B}. By smoothing the kernel $K_H$,
for any $\epsilon>0$ we define the smoothened fBM $B^{H,\varepsilon}$ by
\begin{equation} \label{BH-B-eps}
\forall t\geq0,~~B^{H,\varepsilon}_t 
:= \int_0^t K_H(t+\varepsilon,s)~d\mathbf{B}_s.
\end{equation}
Consider the following process:
$$ \mathcal{Y}^{H,\varepsilon}_t := y_0 + B^{H,\varepsilon}_t + 
\int_0^t \beta_s~ds. $$
The process $B^{H,\varepsilon}$ is not a martingale. As we plan to 
apply the standard It\^o formula for continuous semimartingales we 
use~\eqref{eq:StocFubini-denis} with $u(s,r)\equiv \partial_s 
K_H(s+\varepsilon,r)~\indi{r\leq s}$ to 
rewrite~$\mathcal{Y}^{H,\varepsilon}_t$ as
$$ \mathcal{Y}^{H,\varepsilon}_{t} = y_0 + \int_0^t 
K_H(s+\varepsilon,s)~d\mathbf{B}_s + \int_0^t \left(\int_0^s 
\partial_s K_H(s+\varepsilon,r)~d\mathbf{B}_{r} 
\right) ds + \int_0^t \beta_s~ds. $$
We thus are in a position to apply the Itô formula. It comes:
\begin{equation}\label{eq:ItoSkor1-denis}
\begin{split}
G(t,\mathcal{Y}^{H,\varepsilon}_{t} ) &= 
G(0,y_0) + \int_0^t \left\{ \partial_s 
G(s,\mathcal{Y}^{H,\varepsilon}_s ) + \partial_y 
G(s,\mathcal{Y}^{H,\varepsilon}_s ) \beta_s \right\} ds \\
&~~~+  \int_0^t \partial_y G(r,\mathcal{Y}^{H,\varepsilon}_{r} ) 
\int_0^r \partial_r K_H(r+\varepsilon,s)
~d\mathbf{B}_s~dr \\
&~~~+ \int_0^t \partial_y 
G(s,\mathcal{Y}^{H,\varepsilon}_s ) K_H(s+\varepsilon,s) ~
d\mathbf{B}_s + \tfrac{1}{2} \int_0^t \partial_y^2 
G(s,\mathcal{Y}^{H,\varepsilon}_s ) K_H(s+\varepsilon,s)^2 ~ds.
\end{split}
\end{equation}
In the preceding right-hand side, when $\frac{1}{4}<H<\frac{1}{2}$ the second and third integrals 
diverge when~$\epsilon$ tends to~0. We therefore are going to transform 
their sum. First, we use a standard property of Skorokhod integrals to 
get
\begin{equation}\label{eq:IBP-denis}
\begin{split}
\bm{\delta}^{(T)} \Big(\indi{(0,r)}(\cdot)~\partial_y 
G(r,\mathcal{Y}^{H,\varepsilon}_r)  
\partial_r K_H(r+\varepsilon,\cdot) \Big)  
&= \partial_y G(r,\mathcal{Y}^{H,\varepsilon}_{r} ) \int_0^r 
\partial_r K_H(r+\varepsilon,s) ~d\mathbf{B}_s \\
&\quad - \int_0^r \mathbf{D}_s \left(\partial_y 
G(r,\mathcal{Y}^{H,\varepsilon}_{r} ) \right) 
\partial_r K_H(r+\varepsilon,s)~ds.
\end{split}
\end{equation}
Second, we apply the Fubini formula~\eqref{eq:StocFubini-denis} 
to the Skorokhod integral. This is allowed because
$$ u(r,s) \equiv \indi{[0,r]}(s) 
~\partial_{y} G(r,\mathcal{Y}^{H,\varepsilon}_{r} ) 
~\partial_r K_H(r+\varepsilon,s) $$ 
satisfies~\eqref{ineq:D-1-2} for the following reason.
Since $\partial_{y} G$ and $s\mapsto 
\indi{[0,r]}(s)~\partial_r K_H(r+\varepsilon,s)$ are 
bounded, we have that $u(\cdot,\cdot) \in L^2(\Omega; L^2[0,T]^2)$.
In addition, the assumption~\eqref{eq:AssumptionDrBeta} implies that
\begin{align*}
\EE\left[ \int_{0}^T \int_{0}^T \int_{0}^T |\mathbf{D}_{\theta} 
u(r,s)|^2\, d\theta\, ds\, dr  \right] &\leq  C \, \EE\left[ \int_{0}^T 
\int_{0}^r \int_{0}^T |K_{H}(r+\varepsilon,\theta) + \int_{0}^r 
\mathbf{D}_{\theta} \beta_{v}\, dv|^2\, d\theta\, ds\, dr  \right] \\
&\leq C \left(1 + \int_{0}^T \int_{0}^r \int_{0}^T \int_{0}^r \EE 
|\mathbf{D}_{\theta} \beta_{v}|^2\, dv \, d\theta\, ds\, dr \right)  < 
\infty.
\end{align*}
Third, we observe that
the assumption $\int_{0}^\cdot \beta_{r}\, dr \in 
\mathbb{D}^{1,2}(|\HH|)$ implies that 
$$ \mathbf{D}_s \left(\partial_y 
G(r,\mathcal{Y}^{H,\varepsilon}_{r} ) \right)
= \partial_y^2 G(r,\mathcal{Y}^{H,\varepsilon}_r) 
\left(K_H(r+\varepsilon,s) + \int_0^r \mathbf{D}_s \beta_v~dv\right). $$
and we plug this equality into~\eqref{eq:IBP-denis}. 
Finally, we permute the variables~$r$ and~$s$ in the Lebesgue integral. 
It comes:
\begin{equation*}
\begin{split}
\int_0^t \partial_y G(r,\mathcal{Y}^{H,\varepsilon}_{r} ) 
&\int_0^r \partial_r K_H(r+\varepsilon,s) ~
d\mathbf{B}_s~dr \\
&= \int_0^T \indi{(0,t)}(r)~\bm{\delta}^{(T)} \Big(\indi{(0,r)}(\cdot)
~\partial_y G(r,\mathcal{Y}^{H,\varepsilon}_r)  
\partial_r K_H(r+\varepsilon,\cdot) \Big)~dr \\
&~~~+ \int_0^t \partial_y^2 G(r,\mathcal{Y}^{H,\varepsilon}_r) \int_0^r 
\partial_r K_H(r+\varepsilon,s)
~\Big( K_H(r+\varepsilon,s) + \int_0^r \mathbf{D}_s \beta_v~dv \Big)
~ds~dr \\
&= \bm{\delta}^{(T)} 
\Big( \indi{(0,t)}(\cdot)~\int_\cdot^t~\partial_y 
G(r,\mathcal{Y}^{H,\varepsilon}_r)  
\partial_r K_H(r+\varepsilon,\cdot)~dr \Big) \\
&~~~+ \int_0^t \partial_y^2 G(s,\mathcal{Y}^{H,\varepsilon}_s) 
\int_0^s \partial_s K_H(s+\varepsilon,r)~
\Big( K_H(s+\varepsilon,r) + \int_0^s \mathbf{D}_r \beta_v~dv \Big)
~dr~ds.
\end{split}
\end{equation*}
Finally, we combine the preceding equality 
with~\eqref{eq:ItoSkor1-denis}. 
It results:
\begin{align}
G(t,\mathcal{Y}^{H,\varepsilon}_{t} ) &= 
G(0,\mathcal{Y}^{H,\varepsilon}_0 ) + \int_0^t \left\{ \partial_s 
G(s,\mathcal{Y}^{H,\varepsilon}_s ) + \partial_y 
G(s,\mathcal{Y}^{H,\varepsilon}_s ) \beta_s \right\} ds \label{eq:regularisedIto1}\\
&~~~+ \bm{\delta}^{(T)} \Big(\indi{(0,t)}(\cdot)~
\partial_y G(\cdot,\mathcal{Y}^{H,\varepsilon}_\cdot ) 
K_H(\cdot+\varepsilon,\cdot) \nonumber\\
&~~~~~~~~~~~~~~~~~~~~~~~~~~~+ \indi{(0,t)}(\cdot)
\int_{\cdot}^t\partial_y G(r,\mathcal{Y}^{H,\varepsilon}_r)  
\partial_r K_H(r+\varepsilon,\cdot)~dr~\Big)
\label{eq:regularisedIto2}\\
&~~~+ 
\int_0^t \partial_y^2 G(s,\mathcal{Y}^{H,\varepsilon}_s) 
\Big(\tfrac{1}{2}  K_H(s+\varepsilon,s)^2 + \int_0^s 
\partial_s K_H(s+\varepsilon,r)
~K_H(s+\varepsilon,r)~dr \nonumber\\
&~~~~~~~~~~~~~~~~~~~~~~~~~~~+
\int_0^s \partial_s K_H(s+\varepsilon,r)
\int_0^s \mathbf{D}_r \beta_v ~dv~dr \Big)~ds. 
\label{eq:regularisedIto3}
\end{align}

\vspace{0.3cm}

\underline{\emph{Step $2$: Convergence of the terms in 
\eqref{eq:regularisedIto1} and \eqref{eq:regularisedIto3}.}}
 
By using~\eqref{eq-delta-KH-carre} an easy calculation shows that 
$\sup_{s\in [0,T]} \EE 
\big|\mathcal{Y}^{H,\varepsilon}_{s} - \mathcal{Y}^{H}_{s}\big|^2$ 
converges to $0$ as $\varepsilon\to 0$. The convergence 
in probability of the terms in~\eqref{eq:regularisedIto1} follows.

We now prove the convergence of the trace 
term~\eqref{eq:regularisedIto3}. 

Lemma~\ref{lemma:conv-trace-term} proven below shows that there exists
a positive function $\Psi$  with finite integral on $(0,t)$ such that
\begin{align*}
\sup_{0<\varepsilon<1} 
\Big( \tfrac{1}{2}  K_H(s+\varepsilon,s)^2 + \int_0^s 
\partial_s K_H(s+\varepsilon,r)
~K_H(s+\varepsilon,r)~dr \Big)
\leq \Psi(s)
\end{align*}
and 
\begin{align*}
\forall s>0,
~\tfrac{1}{2} K_H(s+\varepsilon,s)^2 + \int_0^s 
\partial_s K_H(s+\varepsilon,r)~K_H(s+\varepsilon,r)~dr
\underset{\varepsilon\to 0}{\longrightarrow} 
Hs^{2H-1}.
\end{align*}
Therefore, Lebesgue's Dominated Convergence theorem implies that
the following a.s. convergence holds true:
$$ \int_0^t \partial_y^2 G(s,\mathcal{Y}^{H,\varepsilon}_s ) 
\Big( \tfrac{1}{2} K_H(s+\varepsilon,s)^2 + \int_0^s 
\partial_s K_H(s+\varepsilon,r)~K_H(s+\varepsilon,r)~dr
\Big)~ds \\
\underset{\varepsilon\to 0}{\longrightarrow} 
\int_0^t \partial_y^2 G(s,\mathcal{Y}^H_s )~H s^{2H-1}~ds . $$

We now turn to the last term in~\eqref{eq:regularisedIto3}. We have:
\begin{align*}
\EE&\left| \int_0^t \partial_y^2 G(s,\mathcal{Y}^H_s ) \int_0^s  
\left(\int_0^s \mathbf{D}_{r} \beta_v ~dv\right)  
\left(\partial_s K_H(s+\varepsilon,r) - 
\partial_s K_H(s,r)\right) ~dr  ~ds \right|\\
&\leq C~\int_0^t \int_0^s \left(\int_0^s \EE|\mathbf{D}_{r}\beta_v|^2 
~dv\right)^{\frac{1}{2}}  \left|\partial_{s} K_H(s+\varepsilon,r) - \partial_s K_H(s,r)\right| 
~dr  ~ds.
\end{align*}
Therefore, in view of Assumption \eqref{eq:AssumptionDrBeta}, Equality~ 
\eqref{eq:deriv-K_H} and Lebesgue's Dominated Convergence theorem, 
the following a.s. convergence holds:
\begin{align*}
&\int_0^t \partial_y^2 G(s,\mathcal{Y}^{H,\varepsilon}_s ) 
\int_0^s  \partial_s K_H(s+\varepsilon,r)
\Big(\int_0^s \mathbf{D}_{r} \beta_v ~dv\Big)~dr~ds \\
&\underset{\varepsilon\to 0}{\longrightarrow} 
\int_0^t \partial_y^2 G(s,\mathcal{Y}^H_s)  
\int_0^s \partial_s K_H(s,r)
\Big(\int_0^s \mathbf{D}_{r} \beta_v~dv\Big)~dr~ds.
\end{align*}

\vspace{0.3cm}

\underline{\emph{Step $3$: Convergence of the Skorokhod integral 
\eqref{eq:regularisedIto2}.}} 

Before proving the convergence of \eqref{eq:regularisedIto2}, we check 
that its potential limit is well defined. To this end, 
we notice that $\mathcal{Y}^H\in \mathbb{D}^{1,2}(|\HH|)$ since $B^H 
\in 
\mathbb{D}^{1,2}(|\HH|)$ and $\int_{0}^\cdot \beta_{s}\, ds \in 
\mathbb{D}^{1,2}(|\HH|)$ by assumption. 
This implies that $K^{*}_{H,t}[\indi{[0,t]} \partial_{y} 
G(\cdot,\mathcal{Y}^{H}_{\cdot} ) ] \in \text{dom}(\bm{\delta})$. 

To prove the convergence, we consider an arbitrary random variable 
$F\in \mathbb{D}^{1,2}$. We have:
\begin{align*}
&\bigg|\EE\bigg[F  \bigg(\bm{\delta} \Big(\indi{(0,t)}(\cdot) \Big\{
\partial_y G(\cdot,\mathcal{Y}^{H,\varepsilon}_\cdot ) 
K_H(\cdot+\varepsilon,\cdot) +\int_{\cdot}^t\partial_y 
G(r,\mathcal{Y}^{H,\varepsilon}_r)  
\partial_r K_H(r+\varepsilon,\cdot) \, dr \Big\}
~\Big) \\
&\hspace{1cm}
 -   \bm{\delta} \left( K^{*}_{H,t}[\indi{[0,t]}(\cdot)
  \partial_{y} G(\cdot,\mathcal{Y}^{H}_{\cdot} ) ] \right)
 \bigg)  \bigg] \bigg| \\
&= \bigg|\EE\bigg[F  \bigg(\bm{\delta} \Big(\indi{(0,t)}(\cdot) \Big\{
\partial_y G(\cdot,\mathcal{Y}^{H,\varepsilon}_\cdot ) 
K_H(t+\varepsilon,\cdot)
+
\int_{\cdot}^t \left(\partial_y 
G(r,\mathcal{Y}^{H,\varepsilon}_r) - \partial_y 
G(\cdot,\mathcal{Y}^{H,\varepsilon}_\cdot) \right)
\partial_r K_H(r+\varepsilon,\cdot) \, dr \Big\}
~\Big) \\
&\hspace{1cm}
 -   \bm{\delta} \left( K^{*}_{H,t}[\indi{[0,t]}(\cdot) \partial_{y} 
 G(\cdot,\mathcal{Y}^{H}_{\cdot} ) ] \right)
 \bigg)  \bigg] \bigg| \\
 &\leq \left|\EE\left[F \bm{\delta} \Big(\indi{(0,t)}(\cdot) \Big\{
\partial_y G(\cdot,\mathcal{Y}^{H,\varepsilon}_\cdot ) 
K_H(t+\varepsilon,\cdot) - 
\partial_y G(\cdot,\mathcal{Y}^{H}_\cdot ) 
K_H(t,\cdot) \Big\} \Big)\right] \right| \\
&\quad + \Big|\EE\Big[F \bm{\delta} \Big(\indi{(0,t)}(\cdot) \Big\{
\int_{\cdot}^t \left(\partial_y 
G(r,\mathcal{Y}^{H,\varepsilon}_r) - \partial_y 
G(\cdot,\mathcal{Y}^{H,\varepsilon}_\cdot) \right)
\partial_r K_H(r+\varepsilon,\cdot) \, dr\\
&\hspace{2cm} - 
\int_{\cdot}^t \left(\partial_y 
G(r,\mathcal{Y}^{H}_r) - \partial_y 
G(\cdot,\mathcal{Y}^{H}_\cdot) \right)
\partial_r K_H(r,\cdot) \, dr \Big\} \Big)\Big] \Big| \\
&=: A_{1} + A_{2}.
\end{align*}

Set
$$ \|\mathbf{D} F\|_{L^2}^2 
:= \EE \Big[ \int_0^t (\mathbf{D}_s F)^2~ds \Big]. $$
From the duality formula and the Cauchy-Schwarz inequality it results 
that
\begin{align*}
A_{1}^2 \leq \|\mathbf{D} F\|_{L^2}^2 \, \EE\left[\int_{0}^t \Big\{
\partial_y G(s,\mathcal{Y}^{H,\varepsilon}_s ) 
K_H(t+\varepsilon,s) - 
\partial_y G(s,\mathcal{Y}^{H}_s) 
K_H(t,s) \Big\}^2 ds \right]
\end{align*}
and
\begin{align*}
A_{2}^2 &\leq  \|\mathbf{D} F\|_{L^2}^2 \, \EE\Big[\int_{0}^t \Big\{
\int_{s}^t \left(\partial_y 
G(r,\mathcal{Y}^{H,\varepsilon}_r) - \partial_y 
G(s,\mathcal{Y}^{H,\varepsilon}_s) \right)
\partial_r K_H(r+\varepsilon,s) \, dr\\
&\hspace{2cm} - 
\int_{s}^t \left(\partial_y 
G(r,\mathcal{Y}^{H}_r) - \partial_y 
G(s,\mathcal{Y}^{H}_s) \right)
\partial_r K_H(r,s) \, dr \Big\}^2 ds \Big].
\end{align*}

\vspace{0.3cm}
As for $A_{1}$, introduce the term $\partial_y 
G(s,\mathcal{Y}^{H,\varepsilon}_s) K_H(t,s)$ to get
\begin{align*}
A_{1}^2 &\leq 2 \|\mathbf{D} F\|_{L^2}^2 \, \EE\left[\int_{0}^t \Big\{
\partial_y G(s,\mathcal{Y}^{H,\varepsilon}_s ) (
K_H(t+\varepsilon,s)-K_H(t,s)) \Big\}^2 ds \right] \\
&\quad + 2 \|\mathbf{D} F\|_{L^2}^2 \, \EE\left[\int_{0}^t \Big\{
(\partial_y G(s,\mathcal{Y}^{H,\varepsilon}_s ) 
- \partial_y G(s,\mathcal{Y}^{H}_s) )
K_H(t,s) \Big\}^2 ds \right] .
\end{align*}
For the first term in the right-hand side, we use the boundedness of 
$\partial_{y}G$ and observe that $\int_{0}^t 
\left(K_{H}(t+\varepsilon,r) - K_{H}(t,r)\right)^2 \, dr \leq 
\int_{0}^{t+\varepsilon} \left(K_{H}(t+\varepsilon,r) - 
K_{H}(t,r)\right)^2 \, dr = \epsilon^{2H}$. The Dominated Convergence 
theorem implies that the second term tends to~0.
Therefore, $A_{1}$ tends to $0$.

\vspace{0.3cm}
As for $A_{2}$, introduce the term $\int_{s}^t \big(\partial_y 
G(r,\mathcal{Y}^{H,\varepsilon}_r) - \partial_y 
G(s,\mathcal{Y}^{H,\varepsilon}_s) \big)
\partial_r K_H(r,s) \, dr$ to get
\begin{align*}
A_{2}^2 &\leq 2  \|\mathbf{D} F\|_{L^2}^2 \, \EE\Big[\int_{0}^t \Big\{
\int_{s}^t \left(\partial_y 
G(r,\mathcal{Y}^{H,\varepsilon}_r) - \partial_y 
G(s,\mathcal{Y}^{H,\varepsilon}_s) \right)
\Big(\partial_r K_H(r+\varepsilon,s)-\partial_r K_H(r,s)\Big) \, 
dr\Big\}^2 ds \Big]\\
& +2  \|\mathbf{D} F\|_{L^2}^2 \, \EE\Big[\int_{0}^t \Big\{
\int_{s}^t \left(\partial_y 
G(r,\mathcal{Y}^{H}_r) - \partial_y 
G(s,\mathcal{Y}^{H}_s) -\partial_y 
G(r,\mathcal{Y}^{H,\varepsilon}_r) + \partial_y 
G(s,\mathcal{Y}^{H,\varepsilon}_s) \right)
\partial_r K_H(r,s) \, dr \Big\}^2 ds \Big].
\end{align*}
For any bounded measurable process~$(\mathcal{G}_r)$ and any
deterministic positive integrable function~$\mathcal{K}$
Minkowski's integral inequality implies that
\begin{equation} \label{ineq:Minkowski}
 \forall 0<s<t,~~~
\EE\Big(\int_s^t \mathcal{G}_r~\mathcal{K}(r)~dr\Big)^2
\leq \Big(\int_s^t \sqrt{\EE [ 
(\mathcal{G}_r~\mathcal{K}(r))^2 ]}~dr\Big)^2
= \Big(\int_s^t \sqrt{\EE[(\mathcal{G}_r)^2]}~\mathcal{K}(r)~dr\Big)^2. 
\end{equation}
It follows that
\begin{align*}
A_{2}^2 &\leq 2  \|\mathbf{D} F\|_{L^2}^2 
~\int_{0}^t \Big\{ \int_s^t \sqrt{\EE|\partial_y 
G(r,\mathcal{Y}^{H,\varepsilon}_r) - \partial_y 
G(s,\mathcal{Y}^{H,\varepsilon}_s)|^2} 
~\Big|\partial_r K_H(r+\varepsilon,s)
-\partial_r K_H(r,s)\Big| \, dr\Big\}^2 ds \\
& +2  \|\mathbf{D} F\|_{L^2}^2
~\int_0^t \Big\{ \int_s^t \sqrt{ \EE|\partial_y 
G(r,\mathcal{Y}^{H}_r) - \partial_y G(s,\mathcal{Y}^{H}_s) 
-\partial_y G(r,\mathcal{Y}^{H,\varepsilon}_r) 
+ \partial_y G(s,\mathcal{Y}^{H,\varepsilon}_s)|^2 }
~\Big|\partial_r K_H(r,s) \Big| \, dr \Big\}^2 ds \\
&=: A_{2,1}^2 + A_{2,2}^2.
\end{align*}
The Lipschitz property of $\partial_{y}G$ implies that 
$(\EE|\partial_y G(r,\mathcal{Y}^{H,\varepsilon}_r) - \partial_y 
G(s,\mathcal{Y}^{H,\varepsilon}_s)|^2 )^{\frac{1}{2}} \leq C 
{(\EE|B^{H,\varepsilon}_r - B^{H,\varepsilon}_s|^2 )^{\frac{1}{2}}}$. 
In addition, the definition of $B^{H,\varepsilon}$ yields that 
\begin{equation*}
\begin{split}
\EE|B^{H,\varepsilon}_r - B^{H,\varepsilon}_s|^2 
&= \int_{0}^r \left(K_{H}(r+\varepsilon,u) 
- K_{H}(s+\varepsilon,u)\right)^2 du \\
&\leq \int_{0}^{r+\varepsilon} \left(K_{H}(r+\varepsilon,u) - 
K_{H}(s+\varepsilon,u)\right)^2 du  = (r-s)^{2H}.
\end{split}
\end{equation*}
 Therefore,
\begin{align*}
A_{2,1}^2 \leq C\, \int_{0}^t \Big\{
\int_{s}^t (r-s)^H 
\Big|\partial_r K_H(r+\varepsilon,s)-\partial_r K_H(r,s)\Big| \, 
dr\Big\}^2 ds .
\end{align*}
By using the inequalities~\eqref{eq:deriv-K_H-simplifiee-H-petit}, 
\eqref{eq:deriv-K_H-simplifiee-H-grand} in Appendix and the Dominated 
Convergence theorem we conclude that $A_{2,1}$ converges to $0$.

\vspace{0.3cm}
To prove the convergence of $A_{2,2}$, we first observe that
\begin{align*}
&\EE|\partial_y G(r,\mathcal{Y}^{H}_r) 
- \partial_y G(s,\mathcal{Y}^{H}_s) 
-\partial_y G(r,\mathcal{Y}^{H,\varepsilon}_r) 
+ \partial_y G(s,\mathcal{Y}^{H,\varepsilon}_s)|^2 \\
&\hspace{1cm}\leq C~\EE|B^{H}_r - B^{H,\varepsilon}_r |^2 
+C~\EE|B^{H}_s - B^{H,\varepsilon}_s|^2,
\end{align*}
which obviously converges to $0$ for any $r$ and $s$. 
Second, we notice that
\begin{align*}
&\EE|\partial_y 
G(r,\mathcal{Y}^{H}_r) - \partial_y G(s,\mathcal{Y}^{H}_s) -\partial_y 
G(r,\mathcal{Y}^{H,\varepsilon}_r) + \partial_y 
G(s,\mathcal{Y}^{H,\varepsilon}_s)|^2 \\
&\hspace{1cm} \leq C~\EE|B^{H}_r - B^{H}_s |^2  
+C~\EE|B^{H,\varepsilon}_r - B^{H,\varepsilon}_s|^2 \\
&\hspace{1cm} \leq C (r-s)^{2H}.
\end{align*}
By using~\eqref{eq:deriv-K_H-simplifiee-H-petit}, 
\eqref{eq:deriv-K_H-simplifiee-H-grand} and the Dominated Convergence 
theorem we conclude that $A_{2,2}$ converges to $0$. 

We have thus obtained that $A_{1}+A_{2}$ converges to $0$.
Therefore, the term~\eqref{eq:regularisedIto2} weakly converges 
to $ \bm{\delta}( K^{*}_{H,t}[\indi{[0,t]} \partial_{y} 
G(\cdot,\mathcal{Y}^{H}_{\cdot} ) ])$.
\end{proof}

In the second step of the preceding proof we have used the following 
result.

\begin{lemma} \label{lemma:conv-trace-term}
It holds that 
\begin{multline} \label{ineq:trace-term-varepsilon}
\sup_{0<\varepsilon<1}
\Big( \tfrac{1}{2}  K_H(s+\varepsilon,s)^2 + \int_0^s 
\partial_s K_H(s+\varepsilon,r)
~K_H(s+\varepsilon,r)~dr \Big) \\
\leq 
\begin{cases}
C~(1+s^{1-2H}+s^{\frac{1}{2}-H})~(s+1)^{3H-\frac{3}{2}}
~~~\text{when}~H>\tfrac{1}{2}, \\
C~(1+s^{H-1}+s^{2H-1})~~~\text{when}~H<\tfrac{1}{2},
\end{cases}
\end{multline}
and 
\begin{equation} \label{limit:trace-term-varepsilon}
\forall s>0,
~\tfrac{1}{2}  K_H(s+\varepsilon,s)^2 + \int_0^s 
\partial_s K_H(s+\varepsilon,r)~K_H(s+\varepsilon,r)~dr
\underset{\varepsilon\to 0}{\longrightarrow}
Hs^{2H-1}.
\end{equation}
\end{lemma}

\begin{proof}
We have that
\begin{align*}
&\tfrac{1}{2}~K_{H}(s+\varepsilon,s)^2 + \int_{0}^s \partial_s
K_{H}(s+\varepsilon,r)\, K_{H}(s+\varepsilon,r)\, dr \\
&~~~= \tfrac{1}{2}~\chi_{H}^2 
\left(\frac{s+\varepsilon}{s}\right)^{2H-1} \varepsilon^{2H-1} 
+ \tfrac{1}{2}~(H-\tfrac{1}{2})^2~\chi_{H}^2~s^{1-2H} 
\left(\int_{s}^{s+\varepsilon} \theta^{H-\frac{3}{2}} 
(\theta-s)^{H-\frac{1}{2}}\, d\theta\right)^2  \\
&~~~\quad- \chi_{H}^2~(H-\tfrac{1}{2})
~\left(\frac{s+\varepsilon}{s}\right)^{H-\frac{1}{2}} 
~\varepsilon^{H-\frac{1}{2}}~s^{\frac{1}{2}-H}
~\int_{s}^{s+\varepsilon} \theta^{H-\frac{3}{2}}
(\theta-s)^{H-\frac{1}{2}}~d\theta \\
&~~~\quad + (H-\tfrac{1}{2}) 
~\chi_{H}^2~\int_{0}^s \left(\frac{s+\varepsilon}{r}\right)^{2H-1} 
(s+\varepsilon-r)^{2H-2}~dr \\
&~~~\quad - (H-\tfrac{1}{2})^2~\chi_{H}^2\, 
(s+\varepsilon)^{H-\frac{1}{2}} \int_{0}^s r^{1-2H} \, 
(s+\varepsilon-r)^{H-\frac{3}{2}} \int_{r}^{s+\varepsilon} 
\theta^{H-\frac{3}{2}} (\theta-r)^{H-\frac{1}{2}} \, d\theta\, dr \\
&~~~=: A_{1}(\varepsilon) + A_{2}(\varepsilon) 
- A_{3}(\varepsilon) + A_{4}(\varepsilon) - A_5(\varepsilon).
\end{align*}

\vspace{0.2cm}

\underline{Proof of~\eqref{ineq:trace-term-varepsilon}: The case 
$H>\tfrac{1}{2}$.} 

In that case, for any $0<\varepsilon<1$ we obviously have
$$ A_1(\varepsilon) + A_2(\varepsilon)
\leq C~s^{1-2H}~(s+1)^{2H-1} + C~s^{1-2H}~\varepsilon^{2H-1}
~\left(\int_s^{s+\varepsilon} \theta^{H-\frac{3}{2}} 
~d\theta\right)^2 \leq C~s^{1-2H}~(s+1)^{2H-1}. $$
Similarly,
$$ A_3(\varepsilon)
\leq C~\left(\frac{s+1}{s}\right)^{H-\frac{1}{2}}
\varepsilon^{2H-1}~s^{\frac{1}{2}-H}
~\int_{s}^{s+\varepsilon} 
\theta^{H-\frac{3}{2}}~d\theta \leq C~s^{1-2H}~(s+1)^{2H-1}. $$
As for $A_4(\varepsilon)$, we have
\begin{equation*} 
\begin{split}
A_4(\varepsilon) 
&\leq C~(s+1)^{2H-1}~\int_0^s r^{1-2H}~(s-r)^{2H-2}~dr \\
&= C~(s+1)^{2H-1}~s~\int_0^1 (s\theta)^{1-2H}~(s-s\theta)^{2H-2} 
~d\theta \\
&= C~(s+1)^{2H-1}. 
\end{split}
\end{equation*}
As for $A_5(\varepsilon)$ we have that
\begin{equation*}
\begin{split}
A_5(\varepsilon) &\leq C~(s+1)^{H-\frac{1}{2}} \int_0^s r^{1-2H} 
~(s-r)^{H-\frac{3}{2}} \int_r^{s+\varepsilon} 
\theta^{H-\frac{3}{2}}~\theta^{H-\frac{1}{2}}~d\theta~dr \\
&\leq C~(s+1)^{3H-\frac{3}{2}}~\int_0^s r^{1-2H}
~(s-r)^{H-\frac{3}{2}}~dr \\
&= C~(s+1)^{3H-\frac{3}{2}}~s
~\int_0^1 (s\theta)^{1-2H}~(s-s\theta)^{H-\frac{3}{2}}~d\theta \\
&= C~s^{\frac{1}{2}-H}~(s+1)^{3H-\frac{3}{2}}.
\end{split}
\end{equation*}
To summarize the preceding calculations, when $H>\tfrac{1}{2}$ one has
$$ \sum_{i=1}^5 A_i(\varepsilon) 
\leq C~(1+s^{1-2H}+s^{\frac{1}{2}-H})~(s+1)^{3H-\frac{3}{2}}. $$

\vspace{0.2cm}

\underline{Proof of~\eqref{ineq:trace-term-varepsilon}: The case 
$H<\tfrac{1}{2}$.} 

In that case, we estimate the sum
$A_1(\varepsilon)+A_4(\varepsilon)$. Notice that both $A_1(\varepsilon)$
and $A_4(\varepsilon)$ are unbounded when $\varepsilon$ tends to~0
and that $A_4(\varepsilon)$ is negative when $H<\tfrac{1}{2}$.

From the equality
\begin{align*}
\tfrac{1}{2} \varepsilon^{2H-1} = \tfrac{1}{2} (s+\varepsilon)^{2H-1} 
- (H-\tfrac{1}{2}) \int_{0}^s (s+\varepsilon-r)^{2H-2}~dr
\end{align*}
we get that
\begin{equation} \label{eq:A3-varepsilon}
\begin{split}
A_1(\varepsilon)+A_4(\varepsilon) 
&= \chi_{H}^2\, (s+\varepsilon)^{2H-1} 
\left( \tfrac{1}{2} \Big(\frac{s+\varepsilon}{s}\Big)^{2H-1} \right. \\
&\quad \left. + (H-\tfrac{1}{2}) 
\int_0^s (s+\varepsilon-r)^{2H-2} (r^{1-2H} - s^{1-2H})~dr \right).
\end{split}
\end{equation}
For any $0<\varepsilon<1$ one thus has
\begin{equation*}
\begin{split}
|A_1(\varepsilon)+A_4(\varepsilon)|  &\leq C~s^{2H-1}
\left( 1 + \int_0^s (s-r)^{2H-2}~(s^{1-2H} - r^{1-2H})~dr \right) \\
&\leq C~s^{2H-1}~(1+ s~\int_0^1 (s-s\theta)^{2H-2} 
 (s^{1-2H}-(s\theta)^{1-2H})~d\theta\Big) \\
&= C~s^{2H-1}~\Big( 
1 + \int_0^1 (1-\theta)^{2H-2}~(1-\theta^{1-2H})~d\theta \Big) \\
&\leq C~s^{2H-1},
\end{split}
\end{equation*}
since $0< 1-\theta^{1-2H} < 1 - \theta$ for all $0<\theta<1$ and 
$0<H<\frac{1}{2}$.

As for $A_2(\varepsilon)$, we use the inequality
$$ \forall 0<H<\tfrac{1}{2},~\forall 0<s<\theta,
~~~\theta^{H-\frac{3}{2}} 
\leq \theta^{H-\frac{3}{2}+\frac{1}{2}+\frac{H}{2}} 
~(\theta-s)^{-\frac{1}{2}-\frac{H}{2}} \leq 
s^{\frac{3H}{2}-1}~(\theta-s)^{-\frac{1}{2}-\frac{H}{2}} $$
to get
$$ A_2(\varepsilon) \leq C~s^{H-1}~\varepsilon^H < C~s^{H-1}. $$

As for $A_3(\varepsilon)$ we rather use the inequality
$$ \forall 0<H<\tfrac{1}{2},~\forall 0<s<\theta,
~~~\theta^{H-\frac{3}{2}} 
\leq \theta^{H-\frac{3}{2}+H} 
~(\theta-s)^{-H} \leq 
s^{2H-\frac{3}{2}}~(\theta-s)^{-H} $$
to get
$$ A_3(\varepsilon)
\leq C~\varepsilon^{2H-1}~s^{\frac{1}{2}-H}
~\int_{s}^{s+\varepsilon} 
\theta^{H-\frac{3}{2}}~d\theta \leq C~s^{H-1}~\varepsilon^H
\leq C~s^{H-1}. $$

As for $A_5(\varepsilon)$ we have that
\begin{equation*}
\begin{split}
A_5(\varepsilon) &\leq C~s^{H-\frac{1}{2}} \int_0^s r^{1-2H} 
~(s+\varepsilon-r)^{H-\frac{3}{2}} \int_r^{s+\varepsilon} 
r^{H-\frac{3}{2}}~(\theta-r)^{H-\frac{1}{2}}~d\theta~dr \\
&= C~s^{H-\frac{1}{2}}~\int_0^s r^{-H-\frac{1}{2}}
~(s+\varepsilon-r)^{2H-1}~dr \\
&\leq C~s^{H-\frac{1}{2}}~\int_0^s r^{-H-\frac{1}{2}}
~(s-r)^{2H-1}~dr \\
&= C~s^{2H-1}
~\int_0^1 \theta^{-\frac{1}{2}-H}~(1-\theta)^{2H-1}~d\theta \\
&= C~s^{2H-1}.
\end{split}
\end{equation*}

To summarize the preceding calculations, when $H<\tfrac{1}{2}$ one has
$$ \sum_{i=1}^5 A_i(\varepsilon) 
\leq C~(1+s^{H-1}+s^{2H-1}). $$

\vspace{0.2cm}

\underline{Proof of~\eqref{limit:trace-term-varepsilon}: The case 
$H>\tfrac{1}{2}$.} 

In that case, $A_i(\varepsilon)$ obviously tends to~0 with 
$\varepsilon$ for $i=1,2,3$.

In addition, notice that $A_4(\varepsilon)$ tends to
\begin{equation*}
(H-\tfrac{1}{2})~\chi_H^2~\int_0^s s^{2H-1}~r^{1-2H}~(s-r)^{2H-2}~dr
= (H-\tfrac{1}{2})~\chi_H^2~s^{2H-1}
\int_0^1 \gamma^{1-2H}~(1-\gamma)^{2H-2}~d\gamma.
\end{equation*}

Now, observe that $A_5(\varepsilon)$ tends to
$$ (H-\tfrac{1}{2})^2~\chi_{H}^2 
~s^{H-\frac{1}{2}} \int_{0}^s r^{1-2H}
~(s-r)^{H-\frac{3}{2}} \int_{r}^s 
\theta^{H-\frac{3}{2}}~(\theta-r)^{H-\frac{1}{2}}~d\theta~dr. $$
Use the change of variables $\theta=\tfrac{r}{\alpha}$. The above 
expression becomes
$$ (H-\tfrac{1}{2})^2~\chi_{H}^2 
~s^{H-\frac{1}{2}}~\int_0^s (s-r)^{H-\frac{3}{2}}
~\int_{\frac{r}{s}}^1 \alpha^{-2H}~(1-\alpha)^{H-\frac{1}{2}}
~d\alpha~dr. 
$$
Now, use the change of variables $r=s\gamma$. It comes:
$$ (H-\tfrac{1}{2})^2~\chi_{H}^2 
~s^{2H-1} \int_0^1 (1-\gamma)^{H-\frac{3}{2}}
\int_\gamma^1 \alpha^{-2H}~(1-\alpha)^{H-\frac{1}{2}}
~d\alpha~d\gamma. $$
By integrating by parts the inner integral we finally get that  
$A_5(\varepsilon)$ tends to
$$ \tfrac{1}{2}(H-\tfrac{1}{2})~\chi_{H}^2~s^{2H-1}
~\Big( \int_0^1 \gamma^{1-2H}~(1-\gamma)^{2H-2}~d\gamma
- (H-\tfrac{1}{2})\int_0^1 (1-\gamma)^{H-\frac{3}{2}}
\int_\gamma^1 \alpha^{1-2H}~(1-\alpha)^{H-\frac{3}{2}}
~d\alpha~d\gamma \Big). $$
From
\begin{multline*}
\int_0^1 (1-\gamma)^{H-\frac{3}{2}}
\int_\gamma^1 \alpha^{1-2H}~(1-\alpha)^{H-\frac{3}{2}}
~d\alpha~d\gamma \\
= \tfrac{1}{H-\tfrac{1}{2}}~\Big( 
\int_0^1 \alpha^{1-2H}~(1-\alpha)^{H-\frac{3}{2}}~d\alpha
-\int_0^1 \gamma^{1-2H}~(1-\gamma)^{2H-2}~d\gamma \Big)
\end{multline*}
it results that $A_4(\varepsilon)-A_5(\varepsilon)$ tends to
$$ \tfrac{1}{2}~(H-\tfrac{1}{2})~\chi_{H}^2~s^{2H-1}
~\int_0^1 \alpha^{1-2H}~(1-\alpha)^{H-\frac{3}{2}}~d\alpha. $$
By using~\eqref{def:chi-H} and standard properties of Beta and Gamma 
functions
we finally get that the preceding limit is equal to
$$ \tfrac{1}{2}~(H-\tfrac{1}{2})~\chi_{H}^2~s^{2H-1} 
\frac{\Gamma(2-2H)~\Gamma(H-\tfrac{1}{2})}{\Gamma(\tfrac{3}{2}-H)}
=H~s^{2H-1}~\frac{(H-\tfrac{1}{2})
~\Gamma(H-\tfrac{1}{2})}{\Gamma(H+\tfrac{1}{2})}
= H~s^{2H-1}. $$

\vspace{0.2cm}

\underline{Proof of~\eqref{limit:trace-term-varepsilon}: The case 
$H<\tfrac{1}{2}$.}

In that case, $A_i(\varepsilon)$ and obviously tends to~0 with 
$\varepsilon$ for $i=2,3$.

In addition, notice that
\begin{equation} 
A_1(\varepsilon) + A_4(\epsilon) \to \chi_{H}^2\, s^{2H-1} 
\Big( \tfrac{1}{2}  + (H-\tfrac{1}{2}) \int_{0}^s (s-r)^{2H-2} 
\Big(r^{1-2H} - s^{1-2H}\Big) dr \Big).
\end{equation}
By using the change of variables $r=s\theta$ and then integrating by 
parts we transform the right-hand side into
\begin{align} \label{limit-A1+A4}
\chi_{H}^2\, s^{2H-1} & \Big( \tfrac{1}{2}  + (H-\tfrac{1}{2}) 
\int_{0}^1 (1-u)^{2H-2} (u^{1-2H} - 1) du \Big) \nonumber \\
&= \chi_{H}^2\, s^{2H-1} \Big( \tfrac{1}{2}  + (H-\tfrac{1}{2}) 
\Big(\tfrac{1}{1-2H} - \int_{0}^1 (1-u)^{2H-1} \, u^{-2H} \, du\Big)  
\Big) \nonumber \\
&=  \chi_{H}^2\, s^{2H-1}\, (\tfrac{1}{2}-H) \int_{0}^1 (1-u)^{2H-1} \, 
u^{-2H} \, du .
\end{align}

We now observe that
\begin{align*}
-A_5(\varepsilon) \to  - (H-\tfrac{1}{2})^2\, \chi_{H}^2\, 
s^{H-\frac{1}{2}} \int_{0}^s r^{1-2H} \, (s-r)^{H-\frac{3}{2}} 
\int_{r}^{s} \theta^{H-\frac{3}{2}} (\theta-r)^{H-\frac{1}{2}} \, 
d\theta\, dr.
\end{align*}
By using the change of variables $(r,\theta) = (su,s\gamma)$ we 
transform the right-hand side into
\begin{align*}
- (H-\tfrac{1}{2})^2\, \chi_{H}^2\, s^{2H-1} \int_{0}^1 u^{1-2H} 
(1-u)^{H-\frac{3}{2}} \int_{u}^1 \gamma^{H-\frac{3}{2}} 
(\gamma-u)^{H-\frac{1}{2}} \, d\gamma\, du .
\end{align*}
Now, use the change of variables $v= \frac{u}{\gamma}$ and then 
integrate by parts to obtain the new value 
\begin{align*}
&- (H-\tfrac{1}{2})^2\, \chi_{H}^2\, s^{2H-1} \int_{0}^1  
(1-u)^{H-\frac{3}{2}} \int_{u}^1 v^{-2H} (1-v)^{H-\frac{1}{2}} \, dv\, 
du \\
&= - (H-\tfrac{1}{2})\, \chi_{H}^2\, s^{2H-1} \left( \int_{0}^1 v^{-2H} 
\, (1-v)^{H-\frac{1}{2}} dv - \int_{0}^1 u^{-2H} \, (1-u)^{2H-1} du 
\right).
\end{align*}

Combining the preceding equality with~\eqref{limit-A1+A4},
\eqref{def:chi-H} and standard properties of the beta function leads to
\begin{align*}
A_1(\varepsilon) + A_4(\varepsilon) - A_5(\varepsilon) &\to 
\chi_{H}^2~s^{2H-1}~(\tfrac{1}{2}-H) 
\int_{0}^1 v^{-2H} \, (1-v)^{H-\frac{1}{2}} dv \\
&= s^{2H-1} \, \frac{2H\, 
\Gamma(\tfrac{3}{2}-H)}{\Gamma(H+\tfrac{1}{2})~\Gamma(2-2H)}
~(\tfrac{1}{2}-H)~\frac{\Gamma(1-2H)~\Gamma(H+\tfrac{1}{2})}
{\Gamma(\tfrac{3}{2}-H)}.
\end{align*}
It remains to use that $ (1-2H) \Gamma(1-2H) = \Gamma(2-2H)$ to 
conclude that
$$ A_1(\varepsilon) + A_4(\varepsilon) - A_5(\varepsilon) 
\to  H s^{2H-1}. $$
\end{proof}

\subsection{The Itô--Skorokhod formula for the Lamperti 
process $Y^H$} \label{sec:Ito-formula}

Let $Y^H$ be as above. In order to be in a position to apply
the It\^o-Skorokhod formula~\eqref {eq:ItoSkorokhodFormula0} we need to 
check that~$Y^H$ belongs to $\mathbb{D}^{1,2}(|\HH|)$. This property 
seems to be well-known for the fractional Brownian motion of Hurst 
parameter $H>\frac{1}{4}$, but we could not find a proof.

\begin{proposition}\label{lem:YHinD12}
For any $H\in (\tfrac{1}{4},1)\setminus\{\frac{1}{2}\}$, $Y^H\in 
\mathbb{D}^{1,2}(|\HH|)$. 
\end{proposition}

\begin{proof}

When $H>\tfrac{1}{2}$ one obviously has
$\int_0^T\int_0^T |s-t|^{2H-2}ds~dt < \infty$.
Therefore, in view of~\eqref{eq:defAbsH} and Proposition 
\ref{lem:HolderFSDE}, 
the inequality~\eqref{eq:defD12|H|+} holds true, which means that 
$Y^H\in \mathbb{D}^{1,2}(|\HH|)$.

We now treat the case $H\in (\tfrac{1}{4},\tfrac{1}{2})$. 
In view of~\eqref{eq:defAbsH<} and~(\ref{eq:defD12|H|}) we need to 
check that
\begin{equation}\label{eq:integral_D12|H|}
\EE \int_0^T \int_0^T 
\left(\int_t^T |\mathbf{D}_r Y^H_t - \mathbf{D}_r Y^H_s| 
(s-t)^{H-\frac{3}{2}}~ds\right)^2 dt ~ dr<\infty.
\end{equation} 
It suffices to prove that $A$ and $B$ are finite, where
$$ A := \int_0^T \int_0^r \left(\int_t^T |\mathbf{D}_r Y^H_t - 
\mathbf{D}_r Y^H_s| 
(s-t)^{H-\frac{3}{2}}~ds\right)^2 dt ~ dr $$
and
$$ B := \int_0^T \int_r^T \left(\int_t^T |\mathbf{D}_r Y^H_t - 
\mathbf{D}_r Y^H_s| 
(s-t)^{H-\frac{3}{2}} ds\right)^2 dt ~ dr. $$

Use~\eqref{D_rY} to get
\begin{equation*}
\begin{split}
A &= \int_0^T \int_0^r \left(\int_r^T |\mathbf{D}_r Y^H_s| 
(s-t)^{H-\frac{3}{2}} ds\right)^2 dt ~ dr \\
&\leq C\int_0^T \int_0^r \left(\int_r^T (s-t)^{H-\frac{3}{2}} 
\left( |K_H(s,r)| + (s-r)^{H+\frac{1}{2}} \right)
ds\right)^2 dt ~ dr.
\end{split}
\end{equation*}
Observe that, for any $t<r<T$,
$$ 
\int_0^r \left(\int_r^T(s-t)^{H-\frac{3}{2}}(s-r)^{2H+1}~ds\right)^2dt
\leq C\int_0^r \Big(\int_r^T(s-t)^{H-\frac{3}{2}}ds\Big)^2dt
\leq C\int_0^r(T-t)^{2H-1}~dt \leq C. $$
Therefore,
$$ A \leq C + C \int_0^T \int_0^r \Big(\int_r^T 
(s-t)^{H-\frac{3}{2}} ~ |K_H(s,r)|~ds\Big)^2 dt ~ dr. $$
As $(\frac{s}{r})^{H-\frac{1}{2}}<1$ for $s>r$ and $H<\frac{1}{2}$, the 
change of variables $\gamma = \frac{\theta}{r}$ in~\eqref{def:KH}
leads to
\begin{equation}\label{eq:boundKH<1/2} 
\begin{split} 
|K_H(s,r)| &\leq
C\Big\{(s-r)^{H-\frac{1}{2}} + r^{H-\frac{1}{2}} 
\int_1^{\infty}
\gamma^{H-\frac{3}{2}} (\gamma-1)^{H-\frac{1}{2}}~d\gamma\Big\} \\ 
&\leq C~(s-r)^{H-\frac{1}{2}} + C_H~r^{H-\frac{1}{2}}.
\end{split} 
\end{equation} 
It follows that 
$$ A \leq C_H +C \int_0^T \int_0^r \Big(\int_r^T
(s-t)^{H-\frac{3}{2}} ~ (s-r)^{H-\frac{1}{2}}~ds\Big)^2 dt ~
dr . $$ 
Applying Hölder's inequality with $p=\tfrac{4}{3}$ and
$q=4$ one gets 
\begin{align*} A&\leq C_H +C \int_0^T \int_0^r
\Big(\int_r^T 
(s-t)^{\frac{4}{3}(H-\frac{3}{2})} ds\Big)^{\frac{3}{2}} 
\Big(\int_r^T (s-r)^{4H-2}~ds\Big)^{\frac{1}{2}} dt~dr. 
\end{align*} 
The right-hand side is finite for $H>\tfrac{1}{4}$.

\vspace{0.4cm}
We now aim to prove that $B<\infty$.

By using~\eqref{ineq:diff-D_rY} we get
\begin{equation*} 
\begin{split}
B &\leq C \int_0^T \int_r^T \left(\int_t^T
(s-t)^{H-\frac{1}{2}} \Big(|K_H(s,r)| +
(s-r)^{H+\frac{1}{2}}\Big)~ds \right)^2 dt ~ dr \\
&\quad + C\int_0^T \int_r^T \Big( \int_t^T 
(s-t)^{H-\frac{3}{2}}|K_H(s,r) - K_H(t,r)|~ds\Big)^2 dt ~ dr \\
&:= B_1 + B_2.
\end{split}
\end{equation*} 

The inequality~\eqref{eq:boundKH<1/2} implies that 
\begin{align*}
B_1 &\leq C + \int_0^T \int_r^T \Big( \int_t^T 
(s-t)^{H-\frac{1}{2}}(s-r)^{H-\frac{1}{2}}ds \Big)^2 dt~dr  \\
&\leq C + \int_0^T \int_r^T \Big(\int_t^T(s-t)^{2H-1}~ds \Big)^2
~dt~dr < \infty.
\end{align*}

Now, as 
$(\frac{\theta}{r})^{H-\frac{1}{2}}<1$ for $r<t\leq \theta\leq s$ and 
$H<\frac{1}{2}$, in view of~\eqref{eq:deriv-K_H} we have
$$ |K_H(s,r) - K_H(t,r)| 
\leq C \int_t^s(\theta-r)^{H-\frac{3}{2}}~d\theta
= C \int_t^s(\theta-r)^{H-\frac{3}{4}}~(\theta-r)^{-\frac{3}{4}}~d\theta
\leq (t-r)^{H-\frac{3}{4}}~\int_t^s (\theta-r)^{-\frac{3}{4}}~d\theta. 
$$
By using the H\"{o}lder continuity of the function $x^{\frac{1}{4}}$ 
we get
$$ \int_t^T (s-t)^{H-\frac{3}{2}}|K_H(s,r) - K_H(t,r)|~ds 
\leq C~(t-r)^{H-\frac{3}{4}}~\int_t^T (s-t)^{H-\frac{5}{4}}~ds. $$
As $H>\frac{1}{4}$ we deduce that $B_2$ is finite and thus $B$ also is 
finite.
\end{proof}

To conclude this section, we combine the 
propositions~\ref{prop:ItoSkorokhod}
and~\ref{lem:YHinD12} to get the It\^o--Skorokhod formula for functions 
of~$Y^H_t$.

\begin{theorem}
For all $H\in(\tfrac{1}{4},1)\setminus\{\frac{1}{2}\}$, $0\leq t\leq T$ 
and $G\in \mathcal{C}^{1,2}_b([0,T]\times \R)$ one has
\begin{equation} \label{eq:ItoSkorokhodFormula}
\begin{split}
G(t,Y^H_t) &= G(0,Y^H_0) + \int_0^t \left(\partial_s G(s,Y^H_s) + 
\widetilde{b}(Y^H_s)~\partial_y G(s,Y^H_s) \right) ds \\
&\quad\quad + \delta_H^{(T)}\left(\indi{[0,t]}(\sbullet) 
\partial_y 
G(\sbullet,Y^H_\sbullet) \right)
+ \Tr\left[ D^H \partial_y G(\sbullet,Y^H_\sbullet)\right]_t,
\end{split}
\end{equation}
where
\begin{equation} \label{eq:formula-Tr}
\Tr\left[D^H \partial_y G(\cdot,Y^H_\cdot)\right]_t = 
\int_0^t \partial^2_yG(s,Y^H_s)\left( 
Hs^{2H-1} + \int_0^s \partial_s K_H(s,r) 
\int_{0}^s\mathbf{D}_{r}\widetilde{b}(Y^H_{v})\, dv ~dr \right) ds.
\end{equation}
\end{theorem}

\begin{proof}
This is a direct consequence of Proposition \ref{prop:ItoSkorokhod} 
applied to $\beta_{\cdot} = b(Y^H_{\cdot})$, provided that $(\int_{0}^t 
b(Y^H_{s})\, ds,~t\in [0,T])$ is in $\mathbb{D}^{1,2}(|\HH|)$ and that 
\eqref{eq:AssumptionDrBeta} is satisfied. By Proposition 
\ref{lem:YHinD12}, the processes $Y^H$ and $B^H$ belong 
to~$\mathbb{D}^{1,2}(|\HH|)$. Hence, so does $(\int_{0}^t b(Y^H_{s})\, 
ds,~t\in [0,T])$. Finally, one easily deduce the 
inequality~\eqref{eq:AssumptionDrBeta} from~\eqref{D_rY}.
\end{proof}

\section{The smooth functional case: Sensitivity of time marginal 
distributions}\label{sec:fracDiff}

The aim of this section is to prove the following proposition which 
precises the weak convergence result of \cite{JolisViles} for the 
process $X^H$ when $H\rightarrow \tfrac{1}{2}$ by giving a convergence 
rate.

\begin{proposition} \label{th:gapLawXt^H}
Let $X^H$ and $\mathbf{X}$ be the solutions to (\ref{eq:fSDE_0}) and 
(\ref{eq:SDE_0}) respectively. Suppose that $b$ and $\sigma$ satisfy
the hypotheses~\ref{hyp:h1'} and \ref{hyp:h2'}, and $\varphi \in \Cb^{2+\beta}$ for some $\beta>0$. Then, for any 
$T>0$, there exists $C>0$ such that
\begin{equation*}
\forall H\in (\tfrac{1}{4},1),~~\sup_{t\in [0,T]} 
|\EE \varphi(X_t^H) 
- \EE \varphi(\mathbf{X}_t)| \leq C\ |H-\tfrac{1}{2}|.
\end{equation*}
\end{proposition}

\begin{remark}
The convergence rate in Proposition~\ref{th:gapLawXt^H} is optimal.
Indeed, a Taylor expansion of the function $H\mapsto t^{2H-1}$ shows 
that there exists $C>0$ satisfying
$$ \forall t>0,~\forall H\in(\tfrac{1}{4},1),
~~|\EE((B_t^H)^2) - \EE((\mathbf{B}_t)^2) 
- 2(H-\tfrac{1}{2})~t~\log(t)| 
\leq C~(H-\tfrac{1}{2})^2~(1+t^2). $$
Therefore, by means of a suitable truncation of the function~$x^2$ one 
can easily construct a bounded smooth function~$\varphi$ such that
\begin{equation*}
\forall H\in (\tfrac{1}{4},1),~~\sup_{t\in [0,T]} \left|\EE 
\varphi(X_t^H) 
- \EE \varphi(\mathbf{X}_t) \right| = C~|H-\tfrac{1}{2}|
+\ensuremath{\mathop{}\mathopen{}
{\scriptstyle\mathcal{O}}\mathopen{}(H-\tfrac{1}{2})}.
\end{equation*}
\end{remark}

\begin{remark} \label{remark:extension-H-H'}
In view of Proposition~\ref{th:gapLawXt^H} one has the following
estimate for all $H\leq\frac{1}{2}\leq H'$:
\begin{align*}
\sup_{t\in [0,T]} |\EE\varphi(X_{t}^H) - \EE\varphi(X_{t}^{H'})| 
&\leq \sup_{t\in [0,T]} |\EE\varphi(X_{t}^H) - \EE\varphi(\mathbf{X}_{t})| + \sup_{t\in [0,T]} |\EE\varphi(X_{t}^{H'}) - \EE\varphi(\mathbf{X}_{t})|\\
&\leq C(\tfrac{1}{2}-H) + C (H'-\tfrac{1}{2}) = C(H'-H).
\end{align*}
\end{remark}

One can prove an extension of Proposition~\ref{th:gapLawXt^H} 
to any pair $(H,H')$ in $(\frac{1}{4},1)$ as follows. One applies the 
Lamperti transform and Gronwall's lemma to get 
\begin{align*}
\sup_{t\in [0,T]} |\EE\varphi(X_{t}^H) - \EE\varphi(X_{t}^{H'})| 
&\leq \|\varphi'\|_{\infty} \sup_{t\in [0,T]} \EE|X_{t}^H - X_{t}^{H'}| \\
&\leq \|\varphi'\|_{\infty} e^{ \|\tilde{b}'\|_{\infty}T} \, \EE[\sup_{t\in [0,T]} |B_{t}^H-B_{t}^{H'}|].
\end{align*}
It then remains to use that $\EE[\sup_{t\in [0,T]} |B_{t}^H-B_{t}^{H'}|]\lesssim |H-H'|^{1-\varepsilon}$ (see e.g. ~\citet{Decreusefond} or~\citet[p.1404]{R.}).

However, for pedagogical reasons we follow another way in subsections 
~\ref{subsec:HolderEst} and~\ref{subsec:proof_gapLawXt^H}. 
We thus softly introduce our methodology to study the 
sensitivity of 
Laplace transforms of hitting times. Laplace transforms of hitting 
times involve irregular functionals of the paths of~$X^H$ and thus 
arguments based upon Gronwall's lemma cannot work. 

Our strategy is based upon the following observation: when 
$H=\frac{1}{2}$ the process $\mathbf{X}$ is Markovian. Thus, integrals 
w.r.t. its time marginal probability distributions can be expressed in 
terms of elliptic or parabolic PDEs. Whenever the coefficients of the 
generator of $\mathbf{X}$ are smooth enough to allow it, the key 
argument consists in applying Itô's formula to the solution of the 
suitable PDE and then using that the resulting It\^o integral is a 
martingale and thus has zero expectation.

We thereby apply Itô-Skorokhod's formula to the solution of the 
suitable PDE and the fractional diffusion. We thus transform
our sensitivity problem to the comparison between stochastic integrals 
driven by standard Brownian motions and, respectively, by fractional 
Brownian motions. The resulting estimates reflect that the 
larger is $\left|H-\tfrac{1}{2}\right|$, the 
bigger is the loss of the Markov property.

As explained in Section~\ref{subsec:SDE_Ito}, to be in a position to 
apply Itô-Skorokhod's formula we use the Lamperti process $Y^H$ rather 
than~$X^H$.

 \subsection{Preliminary results}\label{subsec:HolderEst}

Let $\mathbf{Y}:= Y^{\frac{1}{2}}$ be the solution 
to~\eqref{eq:fSDE_Lamperti} in the pure Brownian case. Let $F$ be the Lamperti transform of Proposition \ref{prop:Lamperti}.
We need to prove that
\begin{equation*}
\forall H\in (\tfrac{1}{4},1),~~\sup_{t\in [0,T]} \left|\EE 
\varphi\circ F^{-1}(Y_t^H) 
- \EE \varphi\circ F^{-1}(\mathbf{Y}_t) \right| \leq C\ 
|H-\tfrac{1}{2}|.
\end{equation*}

Arbitrarily fix a time $t\in(0,T]$ and consider the following 
parabolic PDE with terminal condition  
$\varphi\circ F^{-1}$ at time~$t$:
\begin{equation}\label{eq:PDE}
\begin{cases}
\frac{\partial }{\partial s}u(s,x) + \widetilde{b}(x)
\frac{\partial}{\partial x} u(s,x) + \tfrac{1}{2} 
\frac{\partial^2}{\partial x^2} u(s,x) = 0,
~~(s,x)\in [0,t)\times \R, \\
u(t,x) = \varphi\circ F^{-1}(x),~~x\in \R.
\end{cases}
\end{equation} 

We prove Proposition \ref{th:gapLawXt^H} in the case $H>\tfrac{1}{2}$ 
only. The necessary additional arguments to handle the case 
$H<\tfrac{1}{2}$ can be found in the technically demanding 
Section~\ref{sec:majgap_0}. We will need the following integrability 
result.

\begin{lemma}\label{lem:bddPDE}
Let $\varphi\in \Cb^{2+\beta}(\R)$ for some $0<\beta<1$. Suppose that 
$b,\sigma$ satisfy the hypotheses~\ref{hyp:h1'}-\ref{hyp:h2'}. 
There exists a unique solution $u(s,x)$ in 
$\Cb^{1,2+\beta}([0,t]\times\R)$ to the PDE~\eqref{eq:PDE}. For any 
$x\in \R$, the functions $\partial_s 
u(\cdot,x)$ and $\partial_x u(\cdot,x)$ are bounded. In addition, for 
any $H>\tfrac{1}{2}$ one has
\begin{equation} \label{inrq:DH-derivee-u}
\int_0^t \int_0^t |r-s|^{2H-2} \left|D^H_r(\partial_x u(s,Y^H_s) 
)\right|~dr~ds <\infty~\text{a.s.}
\end{equation}
\end{lemma}

\begin{proof}
Notice that $\varphi\circ F^{-1}\in \Cb^{2+\gamma}(\R)$
and $\widetilde{b}\in\Cb^2(\R)$. The existence and uniqueness of 
$u(s,x)$ 
in 
$\Cb^{1,2+\beta}([0,t]\times\R)$ is a classical result: See e.g.
~\citet[p.189]{Lunardi}. 
From Feynman-Kac's formula, there exists a locally 
bounded positive function $C(t)$ such that
\begin{equation*}
|u|_{\Cb^{1,2+\beta}([0,t]\times\R)} \leq C(t)~
|\varphi \circ F^{-1}|_{\Cb^{2+\beta}(\R)}.
\end{equation*}
As~$D^H_r\left(\partial_x u(s,Y^H_s)\right) = D^H_r 
Y^H_s~\partial^2_{xx} u(s,Y^H_s)$, 
the desired inequality~\eqref{inrq:DH-derivee-u} follows 
from~Proposition~\ref{lem:HolderFSDE} when $H>\frac{1}{2}$.
\end{proof}

We now are in a position to prove Proposition~\ref{th:gapLawXt^H}.
As already said, we limit ourselves to the case~$H>\tfrac{1}{2}$.

\subsection{Proof of Proposition \ref{th:gapLawXt^H} ($H>\tfrac{1}{2}$)}\label{subsec:proof_gapLawXt^H}

Let $0<t\leq T$ be arbitrarily fixed. We start with representing
$\EE \varphi(Y_t^H) - \EE \varphi(\mathbf{Y}_t)$ in an integral form by 
using the solution $u$ of the PDE~(\ref{eq:PDE}). 
The properties of the function~$u$ recalled in Lemma~\ref{lem:bddPDE}
imply that the process $(u(s,\mathbf{Y}_s),0\leq s\leq t)$ is a 
martingale.

By using the It\^o--Skorokhod formula~\eqref{eq:ItoSkorokhodFormula} 
we get
\begin{align*}
u(t,Y_t^H) &= u(0, F(x_0)) + \int_0^t \left(\partial_s 
u(s,Y^H_s) + \partial_x u(s,Y^H_s)~\widetilde{b}(Y^H_s)\right)\ ds + 
\delta_H\left(\indi{[0,t]} \partial_x u(\cdot,Y_\cdot^H) 
\right) \\
& \quad\quad + \Tr\left[ D^H \partial_y G(\sbullet,Y^H_\sbullet)\right]_t.
\end{align*}
Let us explicit the Trace term. Since we consider the case 
$H>\frac{1}{2}$, by using~\eqref{eq:formula-Tr}, 
\eqref{eq:deriveeMalliavin} and~\eqref{eq:fSDE_Lamperti} we get
\begin{align*}
\Tr\left[ D^H \partial_y G(\sbullet,Y^H_\sbullet)\right]_t &= 
\int_{0}^t \partial^2_{y}G(s,Y^H_{s}) \Big( Hs^{2H-1} + \int_{0}^s 
\partial_{s}K_{H}(s,r)~ K_{H}^*[D^H_{\cdot} Y^H_{s} - 
\indi{[0,s]}(\cdot)](r)~dr\Big) ds.
\end{align*}
As $H>\frac{1}{2}$ we also have $K_{H}(s,s)=0$. Therefore,
\begin{align*}
\int_{0}^s \partial_{s}K_{H}(s,r)~
K_{H}^*[\indi{[0,s]}](\cdot)](r)~dr
&= \int_{0}^s \partial_{s}K_{H}(s,r)~K_{H}(s,r)~dr \\
&= \tfrac{1}{2} 
\int_{0}^s \partial_{s}(K_{H}(s,r))^2~dr = \tfrac{1}{2} \partial_{s} 
\Big(\int_{0}^s (K_{H}(s,r))^2~dr\Big)\\
&= Hs^{2H-1}.
\end{align*}
Combine this equality with~\eqref{eq:defKHgeq12},
\eqref{eq:deriv-Malliavin-Y^H} and $K_{H}(\theta,\theta)=0$ for any 
$0<\theta$ to obtain
\begin{align*}
\Tr\left[ D^H \partial_y G(\sbullet,Y^H_\sbullet)\right]_t &= \int_{0}^t \partial^2_{y}G(s,Y^H_{s}) \int_{0}^s \partial_{s}K_{H}(s,r)~ K_{H}^*[D^H_{\cdot} Y^H_{s}](r)~dr ~ds\\
&=  \int_{0}^t \partial^2_{y}G(s,Y^H_{s}) \int_{0}^s 
\partial_{s}K_{H}(s,r)~\int_{r}^s \partial_{\theta}K_{H}(\theta,r) \, 
D^H_{\theta} Y^H_{s}~d\theta~dr~ds\\
&=  \int_{0}^t \partial^2_{y}G(s,Y^H_{s}) 
\int_{0}^s D^H_{\theta} Y^H_{s}
\Big(\int_{0}^\theta \partial_{s}K_{H}(s,r) \, 
\partial_{\theta}K_{H}(\theta,r)~dr\Big)~d\theta~ds\\
&=  \int_{0}^t \partial^2_{y}G(s,Y^H_{s}) 
\int_{0}^s D^H_{\theta} Y^H_{s}
~\partial_{s}\partial_{\theta} \Big(\int_{0}^\theta
K_{H}(s,r)~K_{H}(\theta,r)~dr\Big)~d\theta~ds.
\end{align*}
It remains to use~\eqref{eq:decompCov} and~\eqref{def:alpha_H} to get 
the following explicit formula for the Trace term:
\begin{align*}
\Tr\left[ D^H \partial_y G(\sbullet,Y^H_\sbullet)\right]_t = \alpha_{H} 
\int_{0}^t \partial^2_{y}G(s,Y^H_{s}) 
\int_{0}^s D^H_{\theta} Y^H_s~(s-\theta)^{2H-2}~d\theta~ds.
\end{align*}
Now, use the definition of $u$ and the fact that the Skorokhod integral 
has zero expectation to get
\begin{equation*}
\begin{split}
\EE \varphi\circ F^{-1}(Y_t^H) 
- \EE\varphi\circ F^{-1}(\mathbf{Y}_t) 
&= \EE u(t,Y_t^H) - u(0, F(x_0)) \\
&= -\tfrac{1}{2}\EE\int_0^t \partial^2_y u(s,Y^H_s)~ds \\
&\quad +\alpha_H \EE\int_0^t \partial^2_y u(s,Y^H_s)
\int_0^s D^H_{\theta} Y^H_s~(s-\theta)^{2H-2}~d\theta~ds \\
&= \EE\int_0^t \partial^2_y u(s,Y^H_s) \Big( Hs^{2H-1} 
-\tfrac{1}{2}\Big) ds \\
&\quad +\alpha_H \EE\int_0^t ~\partial^2_y u(s,Y^H_s) 
\int_0^s (D^H_{\theta} Y^H_s - 1)~(s-\theta)^{2H-2}~d\theta~ds  \\
&=: \Delta^1_H + \Delta^2_H.
\end{split}
\end{equation*}

For any $y>0$ one has $e^y-1 \leq y~e^y$ and 
$1-e^{-y} \leq y$, from which
\begin{equation} \label{rem:power_ineq}
\forall x>0, \forall \alpha\in (-\tfrac{1}{4}, \tfrac{1}{2}),
~|x^\alpha - 1| \leq |\alpha \log(x)|~(1 \vee x^\alpha)
\leq |\alpha \log(x)|~(1 + x^\alpha).
\end{equation}
By using the preceding inequality with $x=s^2$ and 
$\alpha=H-\frac{1}{2}$
we get
$$ |\Delta^1_H|
\leq C~(H-\tfrac{1}{2})~|\partial^2_y u|_\infty \int_0^t |\log(s)|
~(1+s^{2H-1})~ds \leq C~(H-\tfrac{1}{2}). $$

In view of~\eqref{ineq:DrYHt-DrYHs} we also have
\begin{equation*}
\begin{split}
|\Delta^2_H|&\leq C~\alpha_H \int_0^t |\partial^2_y u(s,Y^H_s)| 
\int_0^s (s-\theta)^{2H-2}~(s-\theta) ~d\theta~ds \\
&\leq C~|\partial^2_y u|_\infty~\alpha_H 
\int_0^t \int_0^s (s-\theta)^{2H-1}~d\theta~ds \\
&\leq C~(H-\tfrac{1}{2}). 
\end{split}
\end{equation*}
That ends the proof for $H>\frac{1}{2}$. \qed

\begin{remark} \label{rk:commentaire-H'-smooth-functionals}
We come back to the discussion initiated in the Introduction to justify 
the choice of the Markov model as the proxy model. If the proxy model 
were driven by a fractional Brownian motion with Hurst 
index~$H'\neq\frac{1}{2}$, by applying the It\^o--Skorokhod 
formula~\eqref{eq:ItoSkorokhodFormula} we would be led to estimate
$$ \EE\Tr\Big[ D^H \phi'(Y^H_\cdot)\Big]_t
- \EE\Tr\Big[ D^{H'} \phi'(Y^{H'}_\cdot)\Big]_t. $$
Therefore, in the case of smooth test functionals such as
$\varphi(X_t^H)$ it seems possible to get an accurate sensitivity 
estimate in terms of~$|H-H'|$. We do not develop here the 
calculations and prefer to concentrate on the case of irregular 
functionals.
\end{remark}

\section{The irregular functional case: Sensitivity of Laplace 
transforms of first passage times}\label{sec:majgap_0}

The aim of this section is to
estimate the sensitivity w.r.t. the Hurst parameter~$H$
of the Laplace transform of $\tau^H_X$ defined as in~\eqref{eq:defTauH}.

Our sensitivity analysis on Laplace transforms is based on the PDE 
representation of first hitting time Laplace transforms in the 
pure Markov case $H=\tfrac{1}{2}$. Our strategy starts as in 
Section~\ref{subsec:proof_gapLawXt^H}: Apply It\^o's 
formula to the solution of the suitable PDE in order to transform
the sensitivity problem into a comparison between stochastic integrals 
w.r.t. standard Brownian motions and, respectively, frac{t}ional 
Brownian 
motions. As explained in Section~\ref{subsec:SDE_Ito} we need to 
consider the Lamperti process $Y^H$.

Observe that the first hitting time $\tau^H_X$ of $1$ by $X^H$ 
started from $x_0<1$ is equal to the first hitting time $\tau^{H}_Y$ of 
$F(1)$ by $Y^H$ started from $F(x_0)$.

Before stating our main result in this section we recall that
the notation~$C_H$ has been defined at the end of 
Section~\ref{sec:Intro} and we introduce the following new notation.

\begin{convention} \label{def-Theta-y0}
In all the sequel we set
$$ \mathbb{Y}:=F(1)~~~\text{and}~~~y_0:=F(x_0) < \mathbb{Y}. $$
\end{convention}

\begin{theorem} \label{prop:majgap_0}
Let $X^H$ and $\mathbf{X}$ be the solutions to (\ref{eq:fSDE_0}) and 
(\ref{eq:SDE_0}) respectively. 
Assume that $b$ and $\sigma$ satisfy \ref{hyp:h1'}-\ref{hyp:h2'}. 
Let the function~$\widetilde{b}$ be defined as in 
Proposition~\ref{prop:Lamperti}.

For any $p\geq1$ and $\lambda>|\widetilde{b}'|_\infty$ set
\begin{equation} \label{def:rate}
\rate_p(\mathbb{Y}-y_0,\lambda) := \sup_{s\in \R_+} 
\left( e^{-\frac{1}{2} (\lambda-|\widetilde{b}'|_\infty) p s}
~\EE~e^{-|\mathbb{Y}-Y^H_s|p\mathcal{R}(\lambda)} \right),
\end{equation} 
where
\begin{equation}\label{eq:def_muRST}
\mathcal{R}(\lambda):=\sqrt{2\lambda +\mu^2} - \mu
~~\text{with}~~\mu:=|\widetilde{b}|_\infty.
\end{equation}

Suppose $x_0<1$ and $\lambda >|\widetilde{b}'|_\infty$. 
Set~$\widetilde{\lambda}:= \lambda - |\widetilde{b}'|_\infty$. 
For any $H\in(\tfrac{1}{4},1)$ we have

\begin{multline} \label{ineq:Laplace-transform-main-theorem}
\Big|\EE\left(e^{-\lambda \tau^H_X}\right) 
- \EE\left(e^{-\lambda \tau_{\mathbf{X}}}\right)\Big| \\
\leq C_H~|H-\tfrac{1}{2}|
~\frac{(1+\lambda)^2}{1\wedge\widetilde{\lambda}^3}
~\Big( \rate_{1}(\mathbb{Y}-y_0,{\lambda}) + 
\left(\rate_{2}(\mathbb{Y}-y_0,{\lambda}) \right)
^{\frac{H\wedge\frac{1}{2}}{6}}
+ \left(\rate_{4}(\mathbb{Y}-y_0,{\lambda})\right)
^{\frac{H\wedge\frac{1}{2}}{12}} \Big). 
\end{multline}
\end{theorem}

The following proposition precises the convergence rate
in~\eqref{ineq:Laplace-transform-main-theorem}. It is proven in 
Appendix~\ref{App:moments}.

\begin{proposition}\label{prop:moments}
Let $\lambda >|\widetilde{b}'|_\infty$. Let $m:=\mathbb{Y}-y_0$,
$\mu:=|\widetilde{b}|_\infty$, $q:=p\mathcal{R}(\lambda)$
and 
$\widetilde{\lambda}:= \lambda - |\widetilde{b}'|_\infty$. 

One has
\begin{equation} \label{ineq:moment-Wlambda-YH-integ}
\rate_p(\mathbb{Y}-y_0,\lambda)
\leq C~\Big( e^{-\frac{q}{2}m}
+ e^{-\frac{\widetilde{\lambda}}{2} \Psi^H_q(m)} 
+\exp\big(-2^{-\frac{8}{3}}~m^{\frac{2}{1+2H}}~
\widetilde{\lambda}^{\frac{2H}{1+2H}}\big)
+ \exp\big(-\widetilde{\lambda} \tfrac{m}{2\mu} \big) \Big),
\end{equation}
where
\begin{equation} \label{def-App:Psi-H}
\Psi^H_q(m) := \frac{m}{\mu+q}
~\indi{\left[\left(\frac{m}{\mu+q}\right)^{2H-1} < 1\right]}
+ \left(\frac{m}{\mu+q}\right)^{\frac{1}{2H}}
~\indi{\left[\left(\frac{m}{\mu+q}\right)^{2H-1} \geq 1\right]}.
\end{equation}
\end{proposition}

\begin{remark}\label{rk:optimal-CV-rate}
In Theorem~\ref{prop:majgap_0} the convergence rates w.r.t. 
$|H-\frac{1}{2}|$ and $\lambda$ are optimal. See 
remarks~\ref{rk:estimations-optimales} 
and~\ref{rk:estimation-I1-optimale} below.
\end{remark}

\begin{remark}\label{rem:extThm}
Theorem \ref{prop:majgap_0} provides a sensitivity estimate with a 
constant which explodes when~$\widetilde{\lambda}$ tends to~0. 
In~\cite{RTconf} we emphasize that if the joint probability 
distribution of $B^H$ and its running maximum were explicitly known, then one 
should be able to show that the constants are uniform w.r.t. 
$0<\widetilde{\lambda}=\lambda<1$ when 
$X^H$ is reduced to be the fractional Brownian motion~$B^H$.
In Section~\ref{sec:examples} we succeed to get this result when the drift
coefficient~$\widetilde{b}$  is bounded from below by a positive constant and
$\frac{1}{4}<H<\frac{1}{2}$.
\end{remark}

As explained in the Introduction, the proof of~Theorem 
\ref{prop:majgap_0} is technically demanding because we desire
a bound from above which tends to~0 as fast as possible 
when $H$ tends to~$\tfrac{1}{2}$ and
decays at the same exponential rate when $\lambda$ or $|1-x_0|$ tends 
to infinity as in the exact formula~\eqref{Laplace:brownien}.
This proof is split into Subsections \ref{subsec:W-lambda} to 
\ref{subsec:LpEstimates}.

\begin{itemize}
\item In Subsection~\ref{subsec:W-lambda} we remind the differential
equation solved by the function~$\wl(y) :=\EE\left(e^{-\lambda 
\tau_{\mathbf{Y}}}~\Big|~\mathbf{Y}_0=y\right)$, where
$\mathbf{Y}$ the Lamperti process solution 
to the SDE~\eqref{eq:fSDE_Lamperti} driven by a standard Brownian 
motion~$\mathbf{B}$:
\begin{equation}\label{eq:fSDE_Lamperti-BM}
\forall t\geq 0,\quad \mathbf{Y}_t = y_0 + \mathbf{B}_t + \int_0^t 
\widetilde{b}(\mathbf{Y}_s)\ ds.
\end{equation}
Let $\tau_{\mathbf{Y}}$ be
the first hitting time of $\mathbb{Y}$ by $\mathbf{Y}$.
We suitably define an
extension~$\wl$ to the whole real line of that function and
get estimates on the derivatives $\wl^{(i)}(Y^H_s)$ for $i=1,2$.
\item In Subsection~\ref{subsec:err_dec_fBm} we adopt the same strategy 
as in Subsection~\ref{subsec:proof_gapLawXt^H}. The 
difference \newline $\left|\EE\left(e^{-\lambda 
\tau^H_Y}\right) - \EE\left(e^{-\lambda \tau_{\mathbf{Y}}}
\right)\right|$ is split into the sum of a stopped Lebesgue integral 
and a stopped Skorokhod integral, the integrands being expressed in 
terms of the function~$\wl$.
\item In Subsection~\ref{subsec:I1} we get an accurate estimate on the 
stopped Lebesgue integral.
\item In Subsection~\ref{subsec:proof2II} we get an accurate estimate
on the stopped Skorokhod integral.
\item In Subsections 
\ref{subsec:elementaryProp}, \ref{subsec:lemma-L2-estimates}, \ref{subsec:keyLemma} and \ref{subsec:LpEstimates} we prove
technical intermediate results.

\end{itemize}

\subsection{An `optimal' extension of the Laplace 
transform for $H=\frac{1}{2}$ and related 
estimates}\label{subsec:W-lambda}

For any $\lambda>0$ the function $\wl(y) := \EE\left(e^{-\lambda 
\tau_{\mathbf{Y}}}~\big| \mathbf{Y}_0=y\right)$ defined on the 
interval $(-\infty,\mathbb{Y}]$ solves the following ODE:
\begin{equation}\label{eq:ODE_Y}
\begin{cases}
\widetilde{b}(y) \wl'(y) + \tfrac{1}{2}\wl''(y) &= \lambda 
\wl(y),~y<\mathbb{Y},\\
\wl(\mathbb{Y})&=1,\\
\displaystyle{\lim_{y\rightarrow -\infty}} \wl(y) &= 0.
\end{cases}
\end{equation}

In the sequel we will need to consider `stopped' Skorokhod integrals of 
the type
$$ \left.\delta_H^{(N)}\big(\indi{[0,t]}(\cdot) e^{-\lambda \cdot} 
\wl'(Y_\cdot^H) \big) \right|_{t=\tau^H_Y\wedge N}. $$
These `stopped' integrals can only be defined by considering
the parametered family 
$$ \delta_H^{(N)}\big(\indi{[0,t]}(\cdot) e^{-\lambda \cdot} 
\wl'(Y_\cdot^H) \big) $$ 
which cannot be defined without extending
the domain of the function~$\wl$ to the whole real line. 
Of course, we have to choose an extension which allows us to 
get sharp estimates: We discuss this important issue in the 
remarks~\ref{rk:estimations-optimales} 
and~\ref{rk:estimation-I1-optimale} below.

By abuse of notation we denote our extension below  by $\wl$.
For any $\lambda>0$, $\wl$ is the non-negative $\Cb^2(\R)$ function 
defined as follows:
\begin{equation}\label{eq:ODE_Y_ext}
\begin{cases}
\forall y\leq\mathbb{Y},~~\wl(y) := \EE\big(e^{-\lambda 
\tau_{\mathbf{Y}}}~\big|\mathbf{Y}_0=y\big), \\
\forall y\geq\mathbb{Y},~~ 
\wl(y) := \phi(y)~\wl(2\mathbb{Y}-y),
\end{cases}
\end{equation}
where $\phi(z)$ is a non-negative function in $\Cb^3(\R)$ with 
$\phi(0)=1$, uniformly bounded w.r.t.~$\lambda$ and such 
that the first and second derivatives at $\mathbb{Y}$ of the map 
$\phi(y)~\wl(2\mathbb{Y}-y)$ respectively coincide with 
the left derivatives $W_{\lambda}'(\mathbb{Y}-)$ and
$W_{\lambda}''(\mathbb{Y}-)
=2\lambda-\widetilde{b}(\mathbb{Y})~W_{\lambda}'(\mathbb{Y}-)$. For 
example, 
one can choose
$$ \phi(y) = \Psi\left(2(\wl'(\mathbb{Y}))^2(y-\mathbb{Y})^2
~+~2\wl'(\mathbb{Y})(y-\mathbb{Y})~+~1\right), $$
where $\Psi$ is any non-negative function in $\Cb^3(\R)$ such that 
$\Psi(x)=x$ on $[\frac{1}{2},2]$ and $\Psi(x)=0$ on 
$(-\infty,0)\cup(3,+\infty)$.

In the Brownian motion case, the Laplace transform of the
first hitting time at the threshhold~1 is explicitly given 
by~\eqref{Laplace:brownien}. One easily deduce that
the derivatives w.r.t.~$x_0$ of this Laplace transform tend 
exponentially fast to~0 when $\lambda$ or $(1-x_0)$ tends to infinity.
The following proposition shows that the two first 
derivatives of the function~$\wl$ defined as in 
Section~\ref{subsec:W-lambda} satisfy similar exponential convergence 
rates. We postpone to Appendix~\ref{app:boundW} its easy proof.

\begin{proposition}\label{prop:boundW}
For any $\lambda>0$, let $\wl(y)$ be defined as 
in~\eqref{eq:ODE_Y_ext}. Under the assumptions \ref{hyp:h1'} and 
\ref{hyp:h2'} on 
$b$ 
and $\sigma$ one has
\begin{equation} \label{maj:wl}
\forall y\in\R,~0\leq \wl(y) 
\leq e^{-|\mathbb{Y}-y|~\mathcal{R}(\lambda)},
\end{equation}
where $\mathcal{R}(\lambda)$ is defined as in~\eqref{eq:def_muRST}:
$\mathcal{R}(\lambda):=\sqrt{2\lambda +\mu^2} - \mu$.

In addition, the two first derivatives of $\wl$ satisfy the following 
estimates: There exists $C>0$ depending on $\mu$ only 
such that, for all real numbers $y$ and $\tilde{y}$,
\begin{gather}
|\wl'(y)| \leq C(1 + \lambda)~e^{-|\mathbb{Y}-y|~\mathcal{R}(\lambda)}, 
\label{maj:wl-d1} \\
|\wl''(y)| \leq C(1+\lambda) 
~e^{-|\mathbb{Y}-y|~\mathcal{R}(\lambda)} , \label{maj:wl-d2} \\
|\wl''(y) - \wl''(\tilde{y})| \leq C~(1+\lambda)^2~|y-\tilde{y}|~
\big( e^{-|\mathbb{Y}-y|~\mathcal{R}(\lambda)} + 
e^{-|\mathbb{Y}-\tilde{y}|~\mathcal{R}(\lambda)} \big).
\label{maj:diff-wl-d2}
\end{gather}
\end{proposition}

\subsection{An error decomposition}\label{subsec:err_dec_fBm}

\begin{proposition}\label{prop:gaps}
Set
\begin{align}\label{eq:Delta}
\Delta(s,H) := Hs^{2H-1}-\tfrac{1}{2} 
+ \int_0^s \partial_s 
K_H(s,r)~\int_{0}^s\mathbf{D}_r\widetilde{b}(Y^H_{v})~dv~dr. 
\end{align}
For any $\lambda> |\widetilde{b}'|_\infty$ it holds that
\begin{align} \label{eq:diff-transf-Laplace}
\EE\Big(e^{-\lambda \tau^H_Y}\Big) - \EE\Big(e^{-\lambda 
\tau_{\mathbf{Y}}}\Big) &= \EE\Big[\int_0^{\tau^H_Y} \Delta(s,H) 
e^{-\lambda s} \wl''(Y_s^H)\ ds \Big] \nonumber \\
&\quad\quad +\lim_{N\rightarrow +\infty} \EE\Big[ 
\left.\delta_H^{(N)}\Big(\indi{[0,t]} e^{-\lambda \cdot} 
\wl'(Y_\cdot^H) \Big)\right|_{t=\tau^H_Y\wedge N}\Big] \nonumber \\
&=: I_1(\lambda) + I_2(\lambda).
\end{align}
\end{proposition}

\begin{proof}
Let $N>0$. All the stochastic integrals below are well-defined and 
integrable in view of the bounds on $|\wl'|$ and $|\wl''|$ in 
Proposition~\ref{prop:boundW}.

Apply the It\^o--Skorokhod 
formula~(\ref{eq:ItoSkorokhodFormula}) to 
$e^{-\lambda t}\wl(Y^H_t)$ and use the convention of 
writing~\ref{convention-of-writing}: For any $0<t\leq N$,
\begin{align*}
e^{-\lambda t} \wl(Y_t^H) - \wl(y_0) 
&= \int_0^t e^{-\lambda s} \Big(\widetilde{b}(Y_s^H)~\wl'(Y_s^H) - 
\lambda \wl(Y_s^H)\Big) \ ds  \\
& \quad \quad + \delta_H^{(N)}\Big(
e^{-\lambda \sbullet}\indi{[0,t]}(\sbullet) \wl'(Y^H_\sbullet) 
\Big) + \Tr\Big[D^H 
e^{-\lambda \sbullet}~\wl'(Y^H_\sbullet)\Big]_t.
\end{align*}
Using the ODE (\ref{eq:ODE_Y}) satisfied by $\wl$ 
we get: 
\begin{align*} 
e^{-\lambda(N\wedge \tau^H_Y)} \wl(Y_{N\wedge \tau^H_Y}^H) - \wl(y_0) &= 
-\tfrac{1}{2} \int_0^{N\wedge \tau^H_Y}  e^{-\lambda s} \wl''(Y_s^H)~ds 
+ \delta_H^{(N)}\left. \Big(\indi{[0,t]}\ 
\wl'(Y^H_\cdot)\ e^{-\lambda 
\cdot}\Big)\right|_{t=N\wedge\tau^H_Y} \\
&\quad\quad + \Tr\Big[D^H 
e^{-\lambda \sbullet}~\wl'(Y^H_\sbullet)\Big]_{t=N\wedge\tau^H_Y}.
\end{align*}
We now use the equality~\eqref{eq:formula-Tr} and get: 
\begin{align*}
\EE\Big(e^{-\lambda (N\wedge 
\tau^H_Y)}~\wl(Y^H_{N\wedge \tau^H_Y})\Big) - \wl(y_0) &= \EE \Big[ 
\int_0^{N\wedge\tau^H_Y} \Delta(s,H) \wl''(Y^H_s)\ 
e^{-\lambda s}\ ds \Big] \\
&\quad +\EE\Big[\left.\delta_H^{(N)}\Big(\indi{[0,t]}\ 
\wl'(Y^H_\cdot)\ e^{-\lambda 
\cdot}\Big)\right|_{t=N\wedge\tau^H_Y}\Big].
\end{align*}
The dominated convergence theorem and the inequality~\eqref{maj:wl} 
imply that the left-hand side converges when $N$ tends to infinity. We 
claim that the first integral in the right-hand side also converges in 
the same limit. Actually, we
combine the dominated convergence theorem with the inequality~\eqref{maj:wl-d2} and the estimate~\eqref{ineq:Delta} which will be proven below.
\end{proof}

In the preceding we have used the following technical lemma 
which will also be needed in the proof of Proposition~\ref{prop:I1'}.
\begin{lemma}\label{lem:tecDelta}
One has
\begin{align*}
\forall H\in(\tfrac{1}{4},1), 
~~\Big|\int_0^s \partial_s K_H(s,r)~
\int_{0}^s\mathbf{D}_r\widetilde{b}(Y^H_{v})~dv~dr \Big|
\leq C_H~|H-\tfrac{1}{2}|~e^{s|\widetilde{b}'|_\infty}
(1 + s^{2})~ \text{a.s.}
\end{align*}
Therefore,
\begin{equation} \label{ineq:Delta}
|\Delta(s,H)| \leq \left|Hs^{2H-1}-\tfrac{1}{2}\right|
+C_H~\left|H-\tfrac{1}{2}\right|~e^{s|\widetilde{b}'|_\infty}
(1 + s^{2})~ \text{a.s.}
\end{equation}
\end{lemma}
           
\begin{proof}
In view of Proposition~\ref{lem:DYH} we have
$$ \int_{0}^s |\mathbf{D}_r\widetilde{b}(Y^H_{v})|~dv
\leq C~e^{s|\widetilde{b}'|_\infty} \int_r^s 
\left\{  |K_H(v,r)| + (v-r)^{H+\frac{1}{2}} 
~\indi{\{H<\frac{1}{2}\}} \right\}~dv. $$
Then use~\eqref{def:KH} to get
\begin{equation} \label{ineq:integ-Drb}
\begin{split}
\int_0^s |\mathbf{D}_r\widetilde{b}(Y^H_{v})|~dv
&\leq C~e^{s|\widetilde{b}'|_\infty} \int_r^s 
\Big\{ \left(\tfrac{v}{r}\right)^{H-\frac{1}{2}}~(v-r)^{H-\frac{1}{2}}
+ |H-\tfrac{1}{2}|~r^{\frac{1}{2}-H} 
\int_r^v \theta^{H-\frac{3}{2}} (\theta-r)^{H-\frac{1}{2}}~d\theta \\
&~~~~~~~~~~~~~~~~~~~~~~~~~
+(v-r)^{H+\frac{1}{2}} ~\indi{\{H<\frac{1}{2}\}} \Big\}~dv.
\end{split}
\end{equation}

We now distinguish two cases.

\vspace{0.2cm}

\underline{The case $H<\tfrac{1}{2}$.}

In this case, the change of variables $\gamma=\frac{\theta}{r}$ 
leads to
$$ \int_r^v \theta^{H-\frac{3}{2}} (\theta-r)^{H-\frac{1}{2}}~d\theta
\leq r^{2H-1}\int_1^\infty
\gamma^{H-\frac{3}{2}}~(\gamma-1)^{H-\frac{1}{2}}~d\gamma 
\leq r^{2H-1}~(C+\int_2^\infty \gamma^{H-\frac{3}{2}}~d\gamma)
\leq \frac{C~r^{2H-1}}{|H-\tfrac{1}{2}|}. $$
In view of~\eqref{ineq:integ-Drb} and 
$\left(\tfrac{v}{r}\right)^{H-\frac{1}{2}}<1$ for $0<r<v$
and $H<\tfrac{1}{2}$ we deduce that 
\begin{equation*}
\int_0^s |\mathbf{D}_r\widetilde{b}(Y^H_{v})|~dv
\leq C~e^{s|\widetilde{b}'|_\infty}
( (s-r)^{H+\frac{1}{2}} + r^{H-\frac{1}{2}}~(s-r) 
+(s-r)^{H+\frac{3}{2}}).
\end{equation*}
Recall~\eqref{eq:deriv-K_H}. It comes:
\begin{multline*}
\left|\int_0^s \partial_s K_H(s,r)~
\int_{0}^s\mathbf{D}_r\widetilde{b}(Y^H_{v})~dv~dr \right| \\
\leq C~|H-\tfrac{1}{2}|~e^{s|\widetilde{b}'|_\infty}
\int_0^s 
\left(\tfrac{s}{r}\right)^{H-\frac{1}{2}}(s-r)^{H-\frac{3}{2}} 
~( 
(s-r)^{H+\frac{1}{2}}
+ r^{H-\frac{1}{2}}~(s-r) 
+(s-r)^{H+\frac{3}{2}} )~dr.
\end{multline*}
It now remains to observe that
$$ \int_0^s \left(\tfrac{s}{r}\right)^{H-\frac{1}{2}}
(s-r)^{2H-1}~dr \leq \int_0^s (s-r)^{2H-1}~dr \leq C~s^{2H}, $$
 $$ \int_0^s \left(\tfrac{s}{r}\right)^{H-\frac{1}{2}}
~(s-r)^{H-\frac{1}{2}}~r^{H-\frac{1}{2}}~dr = C~s^{2H}, $$
and
$$ \int_0^s \left(\tfrac{s}{r}\right)^{H-\frac{1}{2}}
~(s-r)^{2H}~dr \leq \int_0^s (s-r)^{2H}~dr \leq C~s^{2H+1} 
\leq C~(1+s^2). $$
\vspace{0.2cm}

\underline{The case $H>\tfrac{1}{2}$.} 

In this case,
$$ \int_r^v \theta^{H-\frac{3}{2}} (\theta-r)^{H-\frac{1}{2}}~d\theta 
\leq \int_r^v (\theta-r)^{2H-2}~d\theta = \frac{(v-r)^{2H-1}}{2H-1}. $$
In view of~\eqref{ineq:integ-Drb} we deduce that
\begin{equation*}
\begin{split}
\int_0^s |\mathbf{D}_r\widetilde{b}(Y^H_{v})|~dv
&\leq C~e^{s|\widetilde{b}'|_\infty} \int_r^s 
\Big\{ \left(\tfrac{v}{r}\right)^{H-\frac{1}{2}}~(v-r)^{H-\frac{1}{2}}
+ r^{\frac{1}{2}-H}~(v-r)^{2H-1} \Big\}~dv \\
&\leq C~e^{s|\widetilde{b}'|_\infty} 
~s^{H-\frac{1}{2}}~r^{\frac{1}{2}-H}~(s-r)^{H+\frac{1}{2}}
+ r^{\frac{1}{2}-H}~(s-r)^{2H}.
\end{split}
\end{equation*}
It comes:
\begin{multline*}
\left|\int_0^s \partial_s K_H(s,r)~
\int_{0}^s\mathbf{D}_r\widetilde{b}(Y^H_{v})~dv~dr \right| \\
\leq C~|H-\tfrac{1}{2}|~e^{s|\widetilde{b}'|_\infty}
\int_0^s 
\left(\tfrac{s}{r}\right)^{H-\frac{1}{2}}(s-r)^{H-\frac{3}{2}} 
~( s^{H-\frac{1}{2}}~r^{\frac{1}{2}-H}~(s-r)^{H+\frac{1}{2}}
+ r^{\frac{1}{2}-H}~(s-r)^{2H})~dr.
\end{multline*}
By using the change of variable $r=s\theta$ we get
$$ \int_0^s s^{2H-1}~(s-r)^{2H-1}~r^{1-2H}~dr
= s^{2H} \int_0^1 (1-\theta)^{2H-1}~\theta^{1-2H}~d\theta
= C_H~s^{2H} $$
and
$$ \int_0^s s^{H-\frac{1}{2}}~(s-r)^{3H-\frac{3}{2}}~r^{1-2H}~dr
= s^{2H} \int_0^1 
(1-\theta)^{3H-\frac{3}{2}}~\theta^{1-2H}~d\theta = C_H~s^{2H}. $$
As $H<1$ we can bound $C_H~s^{2H}$ from above by $C_H~(1+s^2)$. 
That ends the proof.
\end{proof}

We now proceed to the proof of Theorem~\ref{prop:majgap_0}. We aim to 
prove that both $|I_1(\lambda)|$ and $|I_2(\lambda)|$ are bounded from 
above by the right-hand side 
of~\eqref{ineq:Laplace-transform-main-theorem}.

\begin{remark} \label{rk:estimations-optimales}
The remark~\ref{rk:estimation-I1-optimale} below shows that
 in~\eqref{ineq:Laplace-transform-main-theorem} 
the convergence rates w.r.t. $|H-\frac{1}{2}|$ and  $\rate_{1}(\mathbb{Y}-y_{0},{\lambda})$ cannot be improved.
Notice that $\mathcal{M}_{2}$ and $\mathcal{M}_{4}$ decay at the same rate as $\mathcal{M}_{1}$ when~$\lambda\to \infty$.

We emphasize that  $I_1(\lambda)$ does not depend on 
the way the original function~$\wl$ is extended since it
depends on the path of $Y^H$ up to time~$\tau^H_Y$.
In contrast, $I_2(\lambda)$ depends on the chosen extension.
Our  choice allows us to obtain estimates in terms of~$\rate_{1}(\mathbb{Y}-y_{0},{\lambda})$
and does not prevent us to obtain the desired optimal rates.
\end{remark}

\begin{remark} \label{rk:commentaire-H'}
We again come back to the discussion initiated in the Introduction to 
justify 
the choice of the Markov model as the proxy model. If the proxy model 
were driven by a fractional Brownian motion with Hurst 
index~$H'\neq\frac{1}{2}$, in view 
of~\eqref{eq:diff-transf-Laplace}, the 
equality~\eqref{eq:intro-pivot-H'} would lead to estimate
$$ \EE\Big[ 
\left.\delta_H^{(N)}\Big(\indi{[0,t]}(\cdot) e^{-\lambda \cdot} 
\wl'(Y_\cdot^H) \Big)\right|_{t=\tau^H_Y\wedge N}\Big]
- \EE\Big[ 
\left.\delta_{H'}^{(N)}\Big(\indi{[0,t]}(\cdot) e^{-\lambda \cdot} 
\wl'(Y_\cdot^{H'}) \Big)\right|_{t=\tau_Y^{H'}\wedge N}\Big], $$
in terms of~$|H-H'|$. We do not see how to solve this issue.

However, as in Remark~\ref{remark:extension-H-H'}
we can compare the rough and the non-rough models as follows: for all $H\leq\frac{1}{2}\leq H'$, for some constant $C_{H,H'}(\tilde{\lambda})$
one has
\begin{align*}
|\EE e^{-\lambda \tau^H_{Y}} - \EE e^{-\lambda \tau^{H'}_{Y}}| \leq C_{H,H'}(\tilde{\lambda})\, (H'-H)\, (\mathcal{M}_{p}(\lambda,H)+\mathcal{M}_{p}(\lambda,H')).
\end{align*}
\end{remark}

\subsection{Estimate on $I_1(\lambda)$ defined as in 
\eqref{eq:diff-transf-Laplace}}\label{subsec:I1}

Applying Fubini's theorem we get
\begin{equation} \label{eq:I_1-apres-Fubini}
I_1(\lambda) = \int_0^\infty e^{-\lambda s}~\EE\left(\Delta(s,H) 
~\indi{\{\tau^H_Y\geq s\}}~\wl''(Y_s^H)\right)~ds,
\end{equation}
where $\Delta(s,H)$ is defined by~(\ref{eq:Delta}).

\begin{proposition}\label{prop:I1'}
As in Proposition~\ref{prop:moments} set
$ \widetilde{\lambda} :=  \lambda-|\widetilde{b}'|_\infty$.
Suppose $\widetilde{\lambda}>0$. One has
\begin{equation*}
|I_1(\lambda)| \leq 
C_H \frac{1+\lambda}{1\wedge\widetilde{\lambda}^3}
~|H-\tfrac{1}{2}|~\rate_{1}(\mathbb{Y}-y_{0},{\lambda}).
\end{equation*}
\end{proposition}

\begin{proof}
In view of Inequalities~\eqref{ineq:Delta} and~\eqref{maj:wl-d2} 
one has
\begin{equation*}
\begin{split}
|I_1(\lambda)| &= \left|\int_0^\infty e^{-\lambda s} 
~\EE\left(\Delta(s,H) \, \indi{\{\tau^H_Y \geq s\}} \wl''(Y_s^H)\right)ds\right| \\
&\leq C~(1+\lambda)~\int_0^\infty e^{-\lambda s} 
~\EE~e^{-|\mathbb{Y}-Y^H_s|\mathcal{R}(\lambda)}
~\big|Hs^{2H-1}-\tfrac{1}{2}\big|~ds  \\
&\quad + C_H~|H-\tfrac{1}{2}|~(1+\lambda)
\int_0^\infty e^{-(\lambda-|\widetilde{b}'|_\infty)s} 
~\EE~e^{-|\mathbb{Y}-Y^H_s|\mathcal{R}(\lambda)}
(1+s^{2})~ds\\
& \leq C~(1+\lambda)~\rate_1(\mathbb{Y}-y_0,{\lambda})
~\int_0^\infty e^{-\frac{1}{2}\widetilde{\lambda}s}~
\big|Hs^{2H-1}-\tfrac{1}{2}\big|~ds \\
&~~~~~+ C_H~|H-\tfrac{1}{2}|~(1+\lambda)
~\rate_1(\mathbb{Y}-y_0,{\lambda})
~\int_0^\infty e^{-\frac{\widetilde{\lambda}}{2}s}~(1+s^{2})~ds.
\end{split}
\end{equation*} 
Split the integral
$$ \int_0^\infty e^{-\frac{1}{2}\widetilde{\lambda}s}~
\big|Hs^{2H-1}-\tfrac{1}{2}\big|~ds $$
into integrals from $0$ to $\alpha := (\tfrac{1}{2H})^{\frac{1}{2H-1}}$
and from $\alpha$ to $+\infty$. This leads one to consider
\begin{equation*}
I_{11} := \text{sign}(H-\tfrac{1}{2}) \int_{0}^\alpha 
e^{-\frac{1}{2}\widetilde{\lambda}s}~ 
\left(\tfrac{1}{2}-Hs^{2H-1}\right)~ds
~\text{ and }~ I_{12} := \text{sign}(H-\tfrac{1}{2}) 
\int_{\alpha}^{+\infty} e^{-\frac{1}{2}\widetilde{\lambda}s}~ 
\left(Hs^{2H-1}- \tfrac{1}{2}\right)~ds.
\end{equation*}
As for $I_{11}$, integrate by parts and use that $1-\alpha^{2H-1} = 
\frac{1}{H}(H-\frac{1}{2})$. It comes:
\begin{align*}
I_{11} = \tfrac{1}{2}~\text{sign}(H-\tfrac{1}{2})~\alpha~\big( 
\tfrac{1}{H}(H-\tfrac{1}{2})~e^{-\frac{1}{2}\widetilde{\lambda}\alpha} 
+ \tfrac{\widetilde{\lambda}}{2} \int_{0}^\alpha 
e^{-\frac{1}{2}\widetilde{\lambda}s}~(s-s^{2H})~ds \big).
\end{align*}
Observe that $\alpha$ is a bounded function of~$H\in(\tfrac{1}{4},1)$. 
In addition, for any~$s\in[0,\alpha]$ apply the Mean Value theorem to 
the map $H\in(\tfrac{1}{4},1)\mapsto s-s^{2H} 
= s-s^{1+2(H-\frac{1}{2})}$ around the 
point~$H=\tfrac{1}{2}$. It comes:
\begin{equation}\label{eq:I1-2}
\begin{split}
I_{11} &\leq C~|H-\tfrac{1}{2}|
+ |H-\tfrac{1}{2}|~\sup_{s\in [0,\alpha]} \sup_{\gamma\in 
(-\frac{1}{4},\frac{1}{2})} (|\log(s)|~s^{1+2\gamma})~ 
\widetilde{\lambda}\int_{0}^\alpha e^{-\frac{1}{2}\widetilde{\lambda}s} 
~ds \\
&\leq C~|H-\tfrac{1}{2}|.
\end{split}
\end{equation}
As for $I_{12}$, we integrate by parts and apply the Mean Value
theorem to the map $H\mapsto s-s^{2H}$. In addition, we use that
$$ \exists C>0,~\forall s>0,~\forall 
\gamma\in(-\tfrac{1}{4},\tfrac{1}{2}),
~~~|\log(s)|~(s^{1+2\gamma}) \leq C~(1+s^2). $$
We get:
\begin{equation}\label{eq:I1-3}
\begin{split}
I_{12}&\leq C~|H-\tfrac{1}{2}|
+ C~|H-\tfrac{1}{2}|~\widetilde{\lambda}~\int_\alpha^\infty
~(1+s^2) 
~e^{-\frac{1}{2}\widetilde{\lambda}s}~ds \\
&\leq C~|H-\tfrac{1}{2}|~(1+\tfrac{1}{\widetilde{\lambda}^2}).
\end{split}
\end{equation}
To conclude, it remains to gather the inequalities~\eqref{eq:I1-2} and 
~\eqref{eq:I1-3} with
$$ \forall H\in(\tfrac{1}{4},1),~~~\int_0^\infty 
e^{-\frac{\widetilde{\lambda}}{2}s}~
(1+s^{2})~ds
\leq C~(1+\tfrac{1}{\widetilde{\lambda}^3}). $$
\end{proof}

\begin{remark} \label{rk:estimation-I1-optimale}
When $Y^H$ reduces to $B^H$, that is, when 
$\widetilde{b}\equiv0$, $F(y)=y$ and $y_0\equiv0$, in view 
of~\eqref{eq:I_1-apres-Fubini} and~\eqref{eq:ODE_Y} one has
$$ I_1(\lambda) = 
2\lambda\int_0^\infty e^{-\lambda s}~(H~s^{2H-1}-\tfrac{1}{2})
~\EE\left(\indi{\{\tau^H_B\geq s\}}~\wl(B_s^H)\right)~ds. $$
One cannot compute the exact value of the right-hand side since the 
joint law of~$(\tau^H_{B},B^H_s)$ is unknown. The preceding proof consists
in replacing the function~$\wl(y)$ with a continuous extension on the 
whole real line which decays fast to~0 when~$y$ tends to $+\infty$. In 
view of~\eqref{Laplace:brownien}, a natural choice
is~$e^{-|1-y|\sqrt{2\lambda}}$. It leads to estimate
$$ 2\lambda\int_0^\infty e^{-\lambda s}~
(H~s^{2H-1}-\tfrac{1}{2})~\EE\left(\indi{\{\tau^H_Y\geq s\}}
~e^{-|1-B^H_s|\sqrt{2\lambda}}\right)~ds
\simeq 2\lambda\int_0^\infty e^{-\lambda s}~(H~s^{2H-1}-\tfrac{1}{2})
~\EE e^{-|1-B^H_s|\sqrt{2\lambda}}~ds. $$
The calculation done above to estimate~$I_{11}$ and $I_{12}$ shows that 
the preceding quantity is of the order 
~$|H-\tfrac{1}{2}|~\frac{1}{1\wedge\widetilde{\lambda}^2}
~\rate_{1}(\mathbb{Y},{\lambda})$.
\end{remark}

\subsection{Estimate on $I_2(\lambda)$ defined as in 
\eqref{eq:diff-transf-Laplace}}\label{subsec:proof2II}

Recall that
$$ I_2(\lambda) := \lim_{N\rightarrow +\infty} \EE\Big[ 
\left.\delta_H^{(N)}\left(\indi{[0,t]}(\cdot) e^{-\lambda \cdot} 
\wl'(Y_\cdot^H) \right)\right|_{t=\tau^H_Y\wedge N}\Big]. $$

The aim of this section is to prove the following proposition.
\begin{proposition} \label{prop:I2'}
Suppose $\widetilde{\lambda} := \lambda-|\widetilde{b}'|_\infty >0$. Then
\begin{align*} \label{ineq:resultI2}
|I_2(\lambda)| 
&\leq 
C_H~|H-\tfrac{1}{2}|
~\frac{(1+\lambda)^2}{1\wedge\widetilde{\lambda}^3}
\Big( \left(\rate_{2}(\mathbb{Y}-y_0,{\lambda}) \right)
^{\frac{H\wedge\frac{1}{2}}{6}}
+ \left(\rate_{4}(\mathbb{Y}-y_0,{\lambda})\right)
^{\frac{H\wedge\frac{1}{2}}{12}} \Big).
\end{align*}
\end{proposition}

We emphasize that the optional stopping theorem does not hold true for 
the Skorokhod integrals $\delta_H^{(N)}$ when $H\neq\frac{1}{2}$. 
However, applying this theorem to standard It\^o integrals provides
$$\forall N>0,\quad 
{\EE\Big(\left.\bm{\delta^{(N)}}
(\indi{[0,t]}(\sbullet) e^{-\lambda\sbullet} 
\wl'(Y^H_\sbullet))\right|_{t=N\wedge \tau^H_Y} \Big) = 0}.$$ 

We thus are led to introduce the centering term 
$\bm{\delta^{(N)}}(\indi{[0,t]}(\sbullet) e^{-\lambda\sbullet} 
\wl'(Y^H_\sbullet))\Big|_{t=N\wedge \tau^H_Y}$, which is crucial
to get an estimate on~$I_2(\lambda)$ of the order $|H-\frac{1}{2}|$:
\begin{multline*}
\left|\EE\Big[\left. \delta_H^{(N)} \big(
\indi{[0,t]}(\sbullet) e^{-\lambda\sbullet} \wl'(Y^H_\sbullet)
\big)\right|_{t=\tau^H_Y\wedge N} 
- \left. \bm{\delta^{(N)}}
\big( \indi{[0,t]}(\sbullet) e^{-\lambda\sbullet}  
\wl'(Y^H_\sbullet)\big)\right|_{t=\tau^H_Y\wedge N}\Big]\right| \\
= \left|\EE\Big[\left. \bm{\delta^{(N)}}
\big( \{ K_{H,N}^*-\text{Id}\}
(\indi{[0,t\wedge N]}(\sbullet) e^{-\lambda\sbullet} 
\wl'(Y^H_\sbullet))\big)\right|_{t=\tau^H_Y\wedge N}\Big] \right|.
\end{multline*}
Define the field $\{U^{(N)}_t(v),v\geq0,t>0\}$ 
and the process $\{\Upsilon^{(N)}_t, t>0\}$ by
\begin{equation} \label{def:U-Nt}
U^{(N)}_t(v) := 
\{K_{H,N}^* - \textrm{Id}\}\big( \indi{[0,t\wedge N]}(\sbullet) 
e^{-\lambda\sbullet} \wl'(Y^H_\sbullet) \big)(v)
\end{equation}
and
\begin{equation} \label{def:Upsilon-Nt}
\Upsilon^{(N)}_t := \bm{\delta^{(N)}}(U^{(N)}_t(\sbullet)).
\end{equation}
Let $[t]$ denote the integer part of $t$. As $\Upsilon^{(N)}_0=0$ 
for any $t>0$ we have
$$ \Upsilon^{(N)}_t = \Upsilon^{(N)}_t - \Upsilon^{(N)}_{[t]}
+ \sum_{n=1}^{[t]}(\Upsilon^{(N)}_n - 
\Upsilon^{(N)}_{n-1})~\indi{t\geq1}. $$
Therefore,
\begin{equation}\label{eq:1stBoundI2}
|I_2(\lambda)| 
\leq  \lim_{N\rightarrow \infty}~ \sum_{n=0}^{N-1}
\EE\sup_{t\in [n,n+1]}\left[| \Upsilon^{(N)}_t - 
\Upsilon^{(N)}_n|\right].
\end{equation}
In order to estimate the right-hand side of the preceding inequality
we now apply the following corollary of Garsia-Rodemich-Rumsey's lemma:

\begin{lemma}[Garsia-Rodemich-Rumsey]\label{lem:GRR}
Let $\{X_t, t\in[a,b]\}$ be an $\R$-valued continuous stochastic 
process. Then, for $p\geq 1$ and $q>0$ such that $pq>2$, 
\begin{align*}
\EE\Big( \sup_{t\in[a,b]} |X_t-X_a| \Big) &\leq C \frac{pq}{pq-2} 
(b-a)^{q-\frac{2}{p}}~\EE\Big[ \Big( \int_a^b \int_a^b 
\frac{|X_s-X_t|^p}{|t-s|^{pq}}\ ds\ dt\Big)^{\frac{1}{p}} \Big] \\
&\leq  C \frac{pq}{pq-2} (b-a)^{q-\frac{2}{p}}~\Big(\int_a^b \int_a^b 
\frac{\EE\Big(|X_t-X_s|^p\Big)}{|t-s|^{pq}}~ds~dt\Big)^{\frac{1}{p}},
\end{align*}
provided the right-hand side in each line is finite.
\end{lemma}

\begin{proof}
With the notations of \cite[p.353-354]{Nualart}, apply the general 
Garsia-Rodemich-Rumsey lemma with $\psi(x) = x^p$ and $p(x)=x^q$ to 
obtain the first line. The second line results from H\"older's 
inequality.
\end{proof}

We thus obtain:
\begin{equation}\label{eq:1stBoundI2-bis}
|I_2(\lambda)| 
\leq  \lim_{N\rightarrow \infty}~ \sum_{n=0}^{N-1}
C \frac{pq}{pq-2} \Big(\int_n^{n+1} \int_n^{n+1} 
\frac{\EE\big(|\Upsilon^{(N)}_t-\Upsilon^{(N)}_s|^p\big)}
{|t-s|^{pq}}\ ds~dt\Big)^{\frac{1}{p}},
\end{equation}
for any $p\geq 1$ and $q>0$ such that $pq>2$.

We now need to estimate moments of 
$|\Upsilon^{(N)}_t-\Upsilon^{(N)}_s|$ with two different constraints. 
On 
the one hand, to get finiteness of the right-hand side 
of~\eqref{eq:1stBoundI2-bis} it is natural to choose the value of~$pq$ 
close to~2 and then to choose $p$ large to allow the $p$-th 
moment of $|\Upsilon^{(N)}_t-\Upsilon^{(N)}_s|$ to be of order 
$(t-s)^{\gamma(p)}$ 
with a large enough power~$\gamma(p)$. On the other hand, to get a 
convergence rate of $|I_2(\lambda)|$ in terms of $|H-\frac{1}{2}|$,
$\lambda$ and~$|\mathbb{Y}-y_0|$ it 
is convenient to consider the second moment of 
$|\Upsilon^{(N)}_t-\Upsilon^{(N)}_s|$ whose convergence rate to~0
can be obtained by using the
explicit value of $(K_H(t,v)-1)^2$ (see the term $J_1$ in 
the proof of Lemma~\ref{lem:gapKH} and Lemma~\ref{lem:EstDeriv}),
whereas the estimation of other moments 
of~$|\Upsilon^{(N)}_t-\Upsilon^{(N)}_s|$ would necessarily involve the 
hardly tractable terms~$(K_H(t,v)-1)^\gamma$ with $\gamma>2$.

The preceding leads us to use the obvious inequality
\begin{equation*}
\forall p\geq 2,
~~\EE\left(|\Upsilon^{(N)}_t-\Upsilon^{(N)}_s|^p\right) 
\leq \left(\EE|\Upsilon^{(N)}_t-\Upsilon^{(N)}_s|^{2(p-1)}
\right)^{\frac{1}{2}} \times 
\left\|\Upsilon^{(N)}_t-\Upsilon^{(N)}_s\right\|_2.
\end{equation*}
In Subsections~\ref{subsec:lemma-L2-estimates} 
and~\ref{subsec:LpEstimates} we respectively prove that 
for any $0<s<t<N$ with $0<t-s<1$ we have
\begin{equation*}
\begin{split}
\left\|\Upsilon^{(N)}_t-\Upsilon^{(N)}_s\right\|_2
\leq C_H~|H-\tfrac{1}{2}|&~(t-s)^{H\wedge\frac{1}{2}} 
~(1+|\log(t-s)|)~(1+\lambda)^2
~(1+t^2)~e^{-\frac{1}{2}\widetilde{\lambda}s}
\\
&\times\Big(
\left(\rate_{2}(\mathbb{Y}-y_0,{\lambda}) \right)^\frac{1}{2}
+\left(\rate_{4}(\mathbb{Y}-y_0,{\lambda})\right)^\frac{1}{4}
\Big)
\end{split}
\end{equation*}
and, for any $p\geq 2$,
\begin{multline*}
\left(\EE|\Upsilon^{(N)}_t-\Upsilon^{(N)}_s|^{2(p-1)}
\right)^{\frac{1}{2}}  \\
\leq C_H~|H-\tfrac{1}{2}|^{p-1}~(t-s)^{(p-1)(H\wedge\frac{1}{2})}
~(1+|\log(t-s)|)^{p-1}~(1+\lambda)^{2(p-1)}
~(1+t^2)^{p-1}~e^{-(p-1)\widetilde{\lambda}s} .
\end{multline*}
Coming back to~\eqref{eq:1stBoundI2-bis} and for instance choosing
$p=\frac{3}{H\wedge\frac{1}{2}}$ and $q=H\wedge\frac{1}{2}$ we get
\begin{align*}
|I_2(\lambda)| 
&\leq C_H~|H-\tfrac{1}{2}|
~(1+\lambda)^2
~\left((\rate_{2}(\mathbb{Y}-y_0,{\lambda}))
^{\frac{H\wedge\frac{1}{2}}{6}}
+ (\rate_{4}(\mathbb{Y}-y_0,{\lambda}))
^{\frac{H\wedge\frac{1}{2}}{12}}\right) \\
&\quad\quad \times 
\lim_{N\rightarrow \infty}~\sum_{n=0}^{N-1} 
~(1+(n+1)^2)~e^{-\widetilde{\lambda}n}
\left(\int_{n}^{n+1}\int_{n}^{n+1}
\big(1+ |\log(t-s)|\big)^{\frac{3}{H\wedge \frac{1}{2}}}~ds~ 
dt\right)^{\frac{H\wedge\frac{1}{2}}{3}} 
\\
 &\leq  C_H~|H-\tfrac{1}{2}|
~\frac{(1+\lambda)^2}{1\wedge \widetilde{\lambda}^3}
~\left( 
(\rate_{2}(\mathbb{Y}-y_0,{\lambda}))^{\frac{H\wedge\frac{1}{2}}{6}}
+ (\rate_{4}(\mathbb{Y}-y_0,{\lambda}))
^{\frac{H\wedge\frac{1}{2}}{12}}\right).
\end{align*}

\subsection{An elementary proposition}\label{subsec:elementaryProp}

In this subsection we prove the following elementary result which will 
be often used in the sequel.

\begin{proposition} \label{prop:diverses-integrales}
Let $0\leq S<T$. Let $(\xi_\theta)$ be a square integrable 
process on~$[S,T]$. 

\begin{enumerate}[label=(\roman*)]
\item Let $f$ be an integrable function on~$[S,T]$. 
One has
\begin{equation} \label{integrale-moment}
\EE\Big( \int_S^T \xi_\theta~f(\theta)~d\theta \Big)^2
\leq \sup_{S\leq \theta\leq T} \EE((\xi_\theta)^2)
~\Big(\int_S^T |f(\theta)|~d\theta\Big)^2.
\end{equation}

\item 
Suppose in addition that $0<T-S<1$ 
and
$\sup_{S\leq \theta\leq T}\EE((\xi_\theta)^4)<\infty$.
Let $\frac{1}{4}<H<1$.
Let $\phi(\cdot,S)$ be a Lebesgue measurable function such that 
$(\cdot-S)^H~|\phi(\cdot,S)|$ is integrable on~$[S,T]$. Then
\begin{equation} \label{integrale-moment-increment-f}
\EE\left( \int_S^T \xi_\theta
~(Y_\theta^H - Y_S^H)~\phi(\theta,S)~d\theta \right)^2
\leq C \sup_{S\leq \theta\leq T} \sqrt{\EE((\xi_\theta)^4)}
~\Big( \int_S^T (\theta - S)^H~|\phi(\theta,S)|~d\theta \Big)^2.
\end{equation}
\end{enumerate}
\end{proposition}

\begin{proof}
By Cauchy-Schwarz formula,
$$ \Big( \int_S^T \xi_\theta~f(\theta)~d\theta \Big)^2
\leq \int_S^T |f(\theta)|~d\theta~~\int_S^T (\xi_\theta)^2~|f(\theta)|
~d\theta. $$
This provides~\eqref{integrale-moment}. 

Similarly,
\begin{multline*}
\left( \int_S^T 
~\frac{\xi_\theta~(Y_\theta^H - Y_S^H)}{(\theta-S)^{H}}
~(\theta-S)^H~\phi(\theta,S)~d\theta \right)^2 \\
\leq \int_S^T 
~\frac{(\xi_\theta)^2~(Y_\theta^H - Y_S^H)^2}{(\theta-S)^{2H}}
~(\theta-S)^H~|\phi(\theta,S)|~d\theta
~~\int_S^T (\theta-S)^H~|\phi(\theta,S)|~d\theta.
\end{multline*}
From~\eqref{eq:fSDE_Lamperti} 
and~\eqref{eq:defCov}
 we deduce that
$$ \EE( (\xi_\theta)^2~(Y_\theta^H - Y_S^H)^2)
\leq C~\sqrt{\EE( (\xi_\theta)^4)}~((\theta-S)^{2H}+(\theta-S)^2). $$
To get~\eqref{integrale-moment-increment-f} it remains to use that 
$(\theta-S)^{2-2H} \leq 1$ since $0<T-S<1$ by hypothesis. 
\end{proof}

\subsection{
$L^2$-estimate on $(\Upsilon^{(N)}_t-\Upsilon^{(N)}_s)$}
\label{subsec:lemma-L2-estimates}

\begin{lemma} \label{L2-estimates}
Suppose $\frac{1}{4}<H<1$ and $H\neq\frac{1}{2}$. 

Let $\Upsilon^{(N)}$ be defined as in~\eqref{def:Upsilon-Nt}.
For any $0<s<t<N$ with $0<t-s<1$ we have
\begin{equation} \label{ineq:Upsilon_t-Upsilon_s}
\begin{split}
\left\|\Upsilon^{(N)}_t-\Upsilon^{(N)}_s\right\|_2
\leq C_H~|H-\tfrac{1}{2}|&~(t-s)^{H\wedge\frac{1}{2}} 
~(1+|\log(t-s)|)~(1+\lambda)^2
~(1+t^2)~e^{-\frac{1}{2}\widetilde{\lambda}s} \\
&\times\Big(
\left(\rate_{2}(\mathbb{Y}-y_0,{\lambda}) \right)^\frac{1}{2}
+\left(\rate_{4}(\mathbb{Y}-y_0,{\lambda})\right)^\frac{1}{4} \Big).
\end{split}
\end{equation}
\end{lemma}

\begin{proof}
Recall~\eqref{def:U-Nt} and~\eqref{eq:KHbis}. For any 
$0\leq s\leq t\leq N$ and $v$ in $[0,N]$ one has
\begin{align*}
U^{(N)}_t(v)-U^{(N)}_s(v) &= \indi{(s,t]}(v)~K_H(t,v) 
~\wl'(Y^H_v)~e^{-\lambda v} \\
&+ \int_v^t  \partial_\theta K_H(\theta,v) 
~(\indi{(s,t)}(\theta)~\wl'(Y_\theta^H)~e^{-\lambda \theta} 
- \indi{(s,t)}(v)~\wl'(Y_v^H)~e^{-\lambda v})~d\theta \\
&\quad - \indi{(s,t]}(v)~\wl'(Y_v^H)~e^{-\lambda v}.
\end{align*}
Therefore,
\begin{equation} \label{eq:expansion-Ut-Us}
\begin{split}
U^{(N)}_t(v)-U^{(N)}_s(v) &= \indi{(s,t]}(v)~(K_H(t,v) - 1) 
~\wl'(Y^H_v)~e^{-\lambda v} \\
&\quad+ \indi{(s,t]}(v)~\wl'(Y_v^H) \int_v^t 
\partial_\theta K_H(\theta,v)
~(e^{-\lambda\theta} - e^{-\lambda v})~d\theta \\
&\quad+ \indi{(s,t]}(v)\int_v^t 
\partial_\theta K_H(\theta,v)
~(\wl'(Y_\theta^H) -\wl'(Y_v^H))~e^{-\lambda\theta}~d\theta \\
&\quad+ \indi{(0,s)}(v)\int_s^t  
\partial_\theta K_H(\theta,v)
~\wl'(Y_\theta^H)~e^{-\lambda \theta}~d\theta \\
&=: J^{(1)} + J^{(2)} + J^{(3)} + J^{(4)}.
\end{split}
\end{equation}

In view of Meyer's inequalities (\cite[Prop.3.2.1]{Nualart}) we 
have
$$ \big\|\Upsilon^{(N)}_t-\Upsilon^{(N)}_s\big\|_2
\leq C \, \sum_{i=1}^4 \Big\{\int_0^N \EE( |J^{(i)}|^2)~dv 
\Big\}^{\frac{1}{2}}
+ C \, \sum_{i=1}^4 \Big\{ \EE \int_0^N \int_0^N 
|\mathbf{D}_r 
J^{(i)}|^2~dr~dv \Big\}^{\frac{1}{2}}. $$
The first term in the right-hand side is simpler than the second one
and leads to even better estimates.
We thus only detail the calculations which concern the second term. 
We will use the two following inequalities which result
from~\eqref{D_rY} and~\eqref{maj:wl-d2}:

\begin{align} \label{ineq:DrW'YH}
|\mathbf{D}_r \wl'(Y^H_\theta)|~e^{-\lambda \theta}
&\leq \indi{\{r\leq \theta\}} 
~|\wl''(Y^H_\theta)|~e^{-\lambda\theta+|\widetilde{b}'|_\infty\theta}
~\Big(K_H(\theta,r) + C~(\theta-r)^{H+\frac{1}{2}}
~\indi{\{H<\frac{1}{2}\}}\Big) \nonumber \\
&\leq C~\indi{\{r\leq \theta\}}~(1+\lambda)~
~e^{-\widetilde{\lambda}\theta
-|\mathbb{Y}-Y^H_\theta| \mathcal{R}(\lambda)}
~(K_H(\theta,r) + (\theta-r)^{H+\frac{1}{2}}
~\indi{\{H<\frac{1}{2}\}}),
\end{align}
from which
\begin{multline} \label{ineq:vrW'YH_V-carre}
\EE((\mathbf{D}_r \wl'(Y^H_\theta))^2)~e^{-2\lambda \theta} \\
\leq C~\indi{\{r\leq \theta\}}
~(1+\lambda)^2~\rate_{2}(\mathbb{Y}-y_0,{\lambda}) 
~\left(|K_H(\theta,r)|^2 
+ (\theta-r)^{2H+1}~\indi{\{H<\frac{1}{2}\}}\right)
~e^{-\widetilde{\lambda} \theta}.
\end{multline}

\vspace{0.3cm}
\fbox{\textbf{A bound from above for $\int_0^N\int_0^N
\EE( |\mathbf{D}_r J^{(1)}|^2 )~dr~dv$.}}

\vspace{0.1cm}

We have
$$ \mathbf{D}_r J^{(1)} = \indi{(s,t]}(v)~\indi{r\leq v} 
~\mathbf{D}_r(\wl'(Y^H_v))~(K_H(t,v) - 1)~e^{-\lambda v}, $$
from which
$$ \int_0^N\int_0^N \EE( |\mathbf{D}_r J^{(1)}|^2 )~dr~dv 
= \int_s^t \Big(\int_0^v \EE((\mathbf{D}_r \wl'(Y^H_v))^2)~dr \Big)
~(K_H(t,v) - 1)^2~e^{-2\lambda v}~dv. $$
Now, successively use~\eqref{ineq:vrW'YH_V-carre} 
and~\eqref{eq:norme-L2-K_H} to get
\begin{equation*}
\int_s^t \int_0^v \EE (|\mathbf{D}_r J^{(1)}|^2)~dr~dv
\leq C~(1+\lambda)^2~\rate_{2}(\mathbb{Y}-y_0,{\lambda})
~\int_s^t (v^{2H} + v^{2H+2})~(K_H(t,v) - 1)^2
~e^{-\widetilde{\lambda} v}~dv.
\end{equation*}
Bound $v^{2H+2}$ in the right-hand side by $t^2v^{2H}$ and 
use~\eqref{ineq:gapKH} (see Lemma \ref{lem:gapKH} below) to 
conclude that
\begin{multline}\label{ineq:norme-L2-DrJ1} 
\int_0^N \int_0^N \EE (|\mathbf{D}_r J^{(1)}|^2)~dr~dv \\
\leq  C ~ \big|H-\tfrac{1}{2}\big|^2 ~ (t-s)^{(2H)\wedge 1}
~(1+ (\log(t-s))^2)~(1+\lambda)^2 ~\rate_{2}(\mathbb{Y}-y_0,{\lambda}) 
~(1+t^4)~e^{-\widetilde{\lambda} s}. 
\end{multline}

\vspace{0.5cm}
\fbox{\textbf{A bound from above for 
$\int_0^N\int_0^N\EE(|\mathbf{D}_r J^{(2)}|^2)~dr~dv$.}}

\vspace{0.1cm}
We have
$$ \mathbf{D}_r J^{(2)} = \indi{(s,t]}(v)~\indi{r\leq v}~
\mathbf{D}_r(\wl'(Y_v^H)) 
\int_v^t \partial_\theta K_H(\theta,v)
~(e^{-\lambda\theta} - e^{-\lambda v})~d\theta, $$
from which
$$ \int_0^N\int_0^N \EE( |\mathbf{D}_r J^{(2)}|^2 )~dr~dv 
= \int_s^t \Big( \int_0^v \EE((\mathbf{D}_r \wl'(Y^H_v))^2)~dr \Big)
~\Big( \int_v^t 
\partial_\theta K_H(\theta,v)
~(e^{-\lambda\theta} - e^{-\lambda v})~d\theta \Big)^2~dv. $$
Notice that $|e^{-\lambda\theta} - e^{-\lambda v}|
\leq (\theta-v)~e^{-\lambda v}$ for $0 < v \leq \theta$.
Combine this inequality with~\eqref{eq:deriv-K_H-simplifiee-H-petit} 
and~\eqref{eq:deriv-K_H-simplifiee-H-grand} in the appendix to get
$$ \int_v^t \left|\partial_\theta K_H(\theta,v)\right|
~|e^{-\lambda\theta} - e^{-\lambda v}|~d\theta 
\leq C~|H-\tfrac{1}{2}|~(t-v)^{H+\frac{1}{2}}
~\Big(1 + \frac{t^{H-\frac{1}{2}}}{v^{H-\frac{1}{2}}}
~\indi{\tfrac{1}{2}<H<1}\Big)~e^{-\lambda v}. $$
Bound $(t-v)$ by $(t-s)$ in the right-hand side. 
Then, as above, successively use~\eqref{ineq:vrW'YH_V-carre} 
and~\eqref{eq:norme-L2-K_H} to get
\begin{equation*}
\begin{split}
\int_0^N \int_0^N \EE |\mathbf{D}_r J^{(2)}|^2~dr~dv
&\leq C~|H-\tfrac{1}{2}|^2~(1+\lambda)^2
~\rate_{2}(\mathbb{Y}-y_0,{\lambda})~(t-s)^{2H+1}  \\
&~~~~~~\int_s^t (v^{2H} + v^{2H+2})
~\Big(1 + \frac{t^{2H-1}}{v^{2H-1}}~\indi{\frac{1}{2}<H<1}\Big)
~e^{-\widetilde{\lambda} v}~dv.
\end{split}
\end{equation*}
Notice that
$$ \forall s\leq v\leq t,
~~(v^{2H} + v^{2H+2})~\Big(1 + \frac{t^{2H-1}}{v^{2H-1}} \Big)
= (v+v^3)~(v^{2H-1}+t^{2H-1}) \leq C~(1+t^{2H+2}).$$
We thus have obtained:
\begin{equation} \label{ineq:norme-L2-DrJ2}
\int_0^N \int_0^N \EE |\mathbf{D}_r J^{(2)}|^2~dr~dv
\leq C~|H-\tfrac{1}{2}|^2~(t-s)^{2H+2}
~(1+\lambda)^2~\rate_{2}(\mathbb{Y}-y_0,{\lambda})
~(1+t^4)~e^{-\widetilde{\lambda} s}.
\end{equation}

\vspace{0.5cm}
\fbox{\textbf{A bound from above for 
$\int_0^N\int_v^N\EE(|\mathbf{D}_r J^{(3)}|^2)~dr~dv$.}}

\vspace{0.1cm}

In view of~\eqref{ineq:DrW'YH} we have
\begin{equation*}
\begin{split}
|\indi{v\leq r}~\mathbf{D}_r J^{(3)}| &= \indi{(s,t]}(v)~\indi{v\leq r} 
~\Big| \int_r^t \mathbf{D}_r \wl'(Y^H_\theta)
~\partial_\theta K_H(\theta,v)
~e^{-\lambda\theta}~d\theta \Big| \\
&\leq C~\indi{(s,t]}(v)~\indi{v\leq r}~(1+\lambda) \\
&~~~~~\int_r^t e^{-\widetilde{\lambda}\theta
-|\mathbb{Y}-Y^H_\theta| \mathcal{R}(\lambda)}
~(K_H(\theta,r) + (\theta-r)^{H+\frac{1}{2}}~\indi{\{H<\frac{1}{2}\}})
~|\partial_\theta K_H(\theta,v)|~d\theta.
\end{split}
\end{equation*}
Apply~\eqref{integrale-moment} with
$$ \xi_\theta\equiv e^{-\frac{1}{2}\widetilde{\lambda}\theta}
~e^{-|\mathbb{Y}-Y_\theta^H|~\mathcal{R}(\lambda)} $$
and
$$ f(\theta) \equiv \indi{r\leq\theta}~\left(K_H(\theta,r) 
+ (\theta-r)^{H+\frac{1}{2}}~\indi{\{H<\frac{1}{2}\}} \right)
~\partial_\theta K_H(\theta,v)
~e^{-\frac{1}{2}\widetilde{\lambda}\theta}. $$
It comes:
$$ \indi{v\leq r}~\EE(|\mathbf{D}_r J^{(3)}|^2)
\leq C~\indi{(s,t]}(v)~\indi{v\leq r}
~(1+\lambda)^2~\rate_2(\mathbb{Y}-y_0,{\lambda})
~\left(\int_r^t |f(\theta)|~d\theta\right)^2. $$
Now, in view of~\eqref{eq:deriv-K_H-simplifiee-H-petit} we have
$$ |f(\theta)| \leq 
\Big|\partial_\theta K_H(\theta,v)\Big|
~K_H(\theta,r)~e^{-\frac{1}{2}\widetilde{\lambda}r}
+ C~\indi{\{H<\frac{1}{2}\}}|H-\tfrac{1}{2}|
~(\theta-v)^{H-\frac{3}{2}}~(\theta-r)^{H+\frac{1}{2}}
~e^{-\frac{1}{2}\widetilde{\lambda}r}. $$
By using the definitions~\eqref{def:A(r,v)} 
and~\eqref{def:I-reg(v,r,t)} we get
\begin{equation*}
\begin{split}
\int_s^t\int_v^t \Big(\int_r^t |f(\theta)|~d\theta\Big)^2~dr~dv
&\leq C~e^{-\widetilde{\lambda}s}~\int_s^t \int_v^t 
(\mathcal{A}(v,r,t))^2~dr~dv \\
&~~~~~ + 
C~\indi{\{H<\frac{1}{2}\}}~|H-\tfrac{1}{2}|^2
~e^{-\widetilde{\lambda}s}
~\int_s^t \int_v^t (\mathcal{I}(v,r,t))^2~dr~dv.
\end{split}
\end{equation*}
In view of~\eqref{ineq:L2-maj-A(r,v)-H-grand}, 
\eqref{ineq:L2-maj-A(r,v)} and~\eqref{ineq:L2-maj-I-reg(v,r,t)} the 
right-hand side is bounded from above by
\begin{equation} \label{ineq:f(theta)-L2-J3}
\begin{cases}
C_H~|H-\tfrac{1}{2}|^2~(t-s)^{3-2H}~t^{6H-3}~e^{-\widetilde{\lambda}s}
~~~\text{when}~\tfrac{1}{2}<H<1, \\
C_H~|H-\tfrac{1}{2}|^2~\Big((t-s)^{4H}~((\log(t-s))^2+1) 
+ (t-s)^{4H+2} \Big)~e^{-\widetilde{\lambda}s}
~~~\text{when}~\tfrac{1}{4}<H<\tfrac{1}{2}.
\end{cases}
\end{equation}

As $H>\frac{1}{4}$ and $0<t-s<1$ we have thus obtained
\begin{equation} \label{ineq:norme-L2-DrJ3-partie-1-1}
\int_0^N\int_v^N \EE(|\mathbf{D}_r J^{(3)}|^2)~dr~dv
\leq C_H~|H-\tfrac{1}{2}|^2~(t-s)
~(1+\lambda)^2~\rate_2(\mathbb{Y}-y_0,{\lambda})
(1+t^3)~e^{-\widetilde{\lambda}s}.
\end{equation}

\vspace{0.5cm}
\fbox{\textbf{A bound from above for $\int_0^N\int_0^v 
\EE(|\mathbf{D}_r J^{(3)}|^2)~dr~dv$.}}

\vspace{0.1cm}
We have
$$ \indi{r\leq v}~\mathbf{D}_r J^{(3)} = \indi{r\leq v}~\indi{(s,t]}(v)
\int_v^t \partial_\theta K_H(\theta,v)
~(\wl''(Y_\theta^H)~\mathbf{D}_r Y_\theta^H 
-\wl''(Y_v^H)~\mathbf{D}_r Y_v^H)~e^{-\lambda\theta}~d\theta. $$
Insert and subtract $\wl''(Y_v^H)~\mathbf{D}_rY_\theta^H$ in the 
right-hand side. For $r\leq v$ set
$$ \mathbf{D}_r J^{(3)}_1 := \indi{(s,t]}(v)
\int_v^t \partial_\theta K_H(\theta,v)
~(\wl''(Y_\theta^H)-\wl''(Y_v^H))~\mathbf{D}_r Y_\theta^H 
~e^{-\lambda\theta}~d\theta $$
and
$$ \mathbf{D}_r J^{(3)}_2 := \indi{(s,t]}(v)
\int_v^t \partial_\theta K_H(\theta,v)
~\wl''(Y_v^H)~(\mathbf{D}_r Y_\theta^H - \mathbf{D}_r Y_v^H)
~e^{-\lambda\theta}~d\theta. $$

\vspace{0.5cm}

\textbf{(i) A bound for $\mathbf{D}_r J^{(3)}_1$.} In view 
of~\eqref{maj:diff-wl-d2-old} and~\eqref{D_rY}, for $r\leq v$ we have
\begin{equation*}
\begin{split}
|\mathbf{D}_r J^{(3)}_1|
&\leq C~\indi{(s,t]}(v)~(1+\lambda)^2
\int_v^t \left|\partial_\theta K_H(\theta,v)\right|
~\big( e^{-|\mathbb{Y}-Y_\theta^H|~\mathcal{R}(\lambda)} + 
e^{-|\mathbb{Y}-Y_v^H|~\mathcal{R}(\lambda)} \big)
~|Y_\theta^H - Y_v^H| \\
&~~~~~~~~~~~~~~~~~~~~~~~~~~~~~~~~~~~
\left\{ K_H(\theta,r) + (\theta-r)^{H+\frac{1}{2}} 
~\indi{\{H<\frac{1}{2}\}} \right\}
~e^{-\widetilde{\lambda}\theta}~d\theta.
\end{split}
\end{equation*}
Apply~\eqref{integrale-moment-increment-f} with
$$ \xi_\theta\equiv e^{-\frac{1}{2}\widetilde{\lambda}\theta}
~~\big( e^{-|\mathbb{Y}-Y_\theta^H|~\mathcal{R}(\lambda)} + 
e^{-|\mathbb{Y}-Y_v^H|~\mathcal{R}(\lambda)} \big) $$
and
$$ \phi(\theta,v) \equiv \indi{(s,t]}(v)
~|\partial_\theta K_H(\theta,v)| 
~\big\{  K_H(\theta,r)  + (\theta-r)^{H+\frac{1}{2}} 
~\indi{\{H<\frac{1}{2}\}} 
\big\}~e^{-\frac{1}{2}\widetilde{\lambda}\theta}. $$
It comes:
$$ \indi{r\leq v}~\EE(|\mathbf{D}_r J^{(3)}_1|^2)
\leq C~(1+\lambda)^4~\sup_{v\leq \theta\leq t} 
\sqrt{\EE((\xi_\theta)^4)}
~\left( \int_v^t (\theta - v)^H~\phi(\theta,v)~d\theta \right)^2. $$
Notice that
$$ \sup_{v\leq \theta\leq t} \EE((\xi_\theta)^4) 
\leq C~\rate_{4}(\mathbb{Y}-y_0,{\lambda}). $$
In addition, in view of~\eqref{eq:deriv-K_H-simplifiee-H-petit}, for 
any $\theta>v$ one has
$$ \phi(\theta,v) \leq 
\indi{(s,t]}(v) \Big\{
~|\partial_\theta K_H(\theta,v)| 
~K_H(\theta,r)  + C~|H-\tfrac{1}{2}|~(\theta-r)^{H+\frac{1}{2}} 
~(\theta-v)^{H-\frac{3}{2}}
~\indi{\{H<\frac{1}{2}\}}~\Big\} 
~e^{-\frac{1}{2}\widetilde{\lambda}\theta}. $$
Therefore, in view of~\eqref{def:A(r,v)} 
and~\eqref{def:I-diese(v,r,t)}, we have
\begin{multline*}
\int_s^t\int_0^v
\Big( \int_v^t (\theta - v)^H~\phi(\theta,v)~d\theta \Big)^2~dr~dv \\
\leq C~\int_s^t e^{-\widetilde{\lambda}v} 
\int_0^v (\mathcal{A}^\sharp(v,r,t))^2~dr~dv
+ C~\indi{\{H<\frac{1}{2}\}}~|H-\tfrac{1}{2}|^2
\int_s^t ~e^{-\widetilde{\lambda}v}
\int_0^v (\mathcal{I}^\sharp(v,r,t))^2~dr~dv.
\end{multline*}

We now use~\eqref{ineq:L2-maj-Adiese(r,v)-H-grand},
\eqref{ineq:L2-maj-Adiese(r,v)-H-petit}
and~\eqref{ineq:L2-maj-Idiese(v,r,t)}. As we are in the case $0<t-s<1$
the right-hand side of the preceding inequality is 
bounded from above by
$$ C_H~|H-\tfrac{1}{2}|^2~(t-s)^{4H\wedge(2H+1)}
~(1+t^{2H}+t^{4H-1}+t^{2H+2})~e^{-\widetilde{\lambda}s}. $$

We thus have obtained that
\begin{equation} \label{ineq:norme-L2-DrJ3-partie-1}
\int_0^N\int_0^v \EE(|\mathbf{D}_r J^{(3)}_1|^2)~dr~dv
\leq C_H~|H-\tfrac{1}{2}|^2~(t-s)^{4H\wedge(2H+1)}~(1+\lambda)^4
~\sqrt{\rate_{4}(\mathbb{Y}-y_0,{\lambda})}
~(1+t^4)~e^{-\widetilde{\lambda}s}.
\end{equation}

\vspace{0.5cm}

\textbf{(ii)} We now turn to~$\mathbf{D}_r J^{(3)}_2$.
In view of~\eqref{maj:wl-d2-old} and~\eqref{ineq:diff-D_rY}, 
for $r\leq v$ we have 
\begin{equation*}
\begin{split}
|\mathbf{D}_r J^{(3)}_2|
&\leq C~\indi{(s,t]}(v)~(1+\lambda) 
\int_v^t \left|\partial_\theta K_H(\theta,v)\right|
~e^{-\widetilde{\lambda}\theta} 
~e^{-|\mathbb{Y}-Y_v^H|~\mathcal{R}(\lambda)} \\
&~~~~~~~~~~~~~~~~~~~~~~~~~~~~~~
\left\{ |K_H(\theta,r) - K_H(v,r)|
+ (\theta-v) \left(K_H(v,r)+(v-r)^{H+\frac{1}{2}} 
~\indi{\{H<\frac{1}{2}\}}\right)\right\}~d\theta.
\end{split}
\end{equation*}
In order to be in a position to again use our estimate on
$\mathcal{A}^\sharp(v,r,t)$ and $\mathcal{I}^\sharp(v,r,t)$ we 
replace $K_H(v,r)$ by $K_H(v,r)-K_H(\theta,r)+K_H(\theta,r)$ and we
bound $(v-r)^{H+\frac{1}{2}}$ from above by 
$(\theta-r)^{H+\frac{1}{2}}$. In addition, we 
use~\eqref{eq:deriv-K_H-simplifiee-H-petit} and the obvious
inequalities $\theta-v \leq (\theta-v)^H$ and 
$(\theta-v)^{H-\frac{1}{2}} \leq (\theta-v)^{2H-\frac{3}{2}}$
for any $0<\theta-v<t-s<1$ and $\tfrac{1}{4}<H<1$.
This leads us to apply~\eqref{integrale-moment} with
$$ \xi_\theta\equiv e^{-\frac{1}{2}\widetilde{\lambda}\theta}
~ e^{-|\mathbb{Y}-Y_v^H|~\mathcal{R}(\lambda)}$$
and
\begin{equation*}
\begin{split}
f(\theta) &\equiv  
C~\indi{(s,t]}(v)~\indi{r\leq v} \Big\{
|\partial_\theta K_H(\theta,v)|
~|K_H(\theta,r) - K_H(v,r)| \\
&~~~+ |\partial_\theta K_H(\theta,v)|
~K_H(\theta,r)~(\theta-v)^H +
|H-\tfrac{1}{2}|~(\theta-v)^{2H-\frac{3}{2}}~(\theta-r)^{H+\frac{1}{2}}
~\indi{\{H<\frac{1}{2}\}} \Big\}
~e^{-\frac{1}{2}\widetilde{\lambda}\theta}.
\end{split}
\end{equation*}
It comes:
$$ \EE(|\mathbf{D}_r J^{(3)}_2|^2)
\leq C~(1+\lambda)^2~\sup_{v\leq \theta\leq t} \EE((\xi_\theta)^2)
\Big(\int_v^t f(\theta)~d\theta\Big)^2. $$
Notice that
$$ \sup_{v\leq \theta\leq t}\EE((\xi_\theta)^2) 
\leq \rate_{2}(\mathbb{Y}-y_0,{\lambda}). $$
In view of~\eqref{def:A(r,v)} and~\eqref{def:I-diese(v,r,t)} we also
have
\begin{equation*}
\begin{split}
\int_s^t\int_0^v
\Big( \int_v^t f(\theta)~d\theta \Big)^2~dr~dv
&\leq C~\int_s^t e^{-\widetilde{\lambda}v} 
\int_0^v (\mathcal{A}^\flat(v,r,t))^2~dr~dv
+ C~\int_s^t e^{-\widetilde{\lambda}v} 
\int_0^v (\mathcal{A}^\sharp(v,r,t))^2~dr~dv \\
&~~~+ C~\indi{\{H<\frac{1}{2}\}}~|H-\tfrac{1}{2}|^2
\int_s^t e^{-\widetilde{\lambda}v} 
\int_0^v (\mathcal{I}^\sharp(v,r,t))^2~dr~dv.
\end{split}
\end{equation*}

We now use~\eqref{ineq:L2-maj-Adiese(r,v)-H-grand}, 
\eqref{ineq:L2-maj-Adiese(r,v)-H-petit}, 
\eqref{ineq:L2-maj-mathcal-Abemol(r,v)-H-grand},
\eqref{ineq:L2-maj-mathcal-Abemol(r,v)-H-petit}
and~\eqref{ineq:L2-maj-Idiese(v,r,t)}. As we are in the case $0<t-s<1$
the right-hand side of the preceding inequality is bounded from above by
$$ C_H~|H-\tfrac{1}{2}|^2~(t-s)^{4H\wedge(2H+1)}
~(1+t^{2H+2})~e^{-\widetilde{\lambda}s}. $$

We thus have obtained that
\begin{equation} \label{ineq:norme-L2-DrJ3-partie-2}
\int_0^N\int_0^v \EE(|\mathbf{D}_r J^{(3)}_2|^2)~dr~dv
\leq C_H~|H-\tfrac{1}{2}|^2~(t-s)^{4H\wedge(2H+1)}~(1+\lambda)^2
~\rate_{2}(\mathbb{Y}-y_0,{\lambda})
~(1+t^4)~e^{-\widetilde{\lambda}s}.
\end{equation}

\vspace{0.5cm}
\fbox{\textbf{A bound from above for $\int_0^N\int_0^N 
\EE(|\mathbf{D}_r J^{(4)}|^2)~dr~dv$.}}

In view of~\eqref{ineq:DrW'YH} we have
\begin{equation*}
\begin{split}
\big|\mathbf{D}_r J^{(4)}\big| &= \indi{(0,s)}(v)
~\Big| \int_s^t \partial_\theta K_H(\theta,v)
~\mathbf{D}_r(\wl'(Y_\theta^H))~e^{-\lambda \theta}~d\theta \Big| \\
&\leq C~\indi{(0,s)}(v)~(1+\lambda) \\
&~~~~~~~~~~ \Big|\int_s^t \indi{r\leq\theta}
~\partial_\theta K_H(\theta,v)
~e^{-\widetilde{\lambda}\theta
-|\mathbb{Y}-Y^H_\theta| \mathcal{R}(\lambda)}
~(K_H(\theta,r) + (\theta-r)^{H+\frac{1}{2}}~\indi{\{H<\frac{1}{2}\}})
~d\theta \Big|.
\end{split}
\end{equation*}
Apply~\eqref{integrale-moment} with
$$ \xi_\theta\equiv e^{-\frac{1}{2}\widetilde{\lambda}\theta}
~e^{-|\mathbb{Y}-Y_\theta^H|~\mathcal{R}(\lambda)} $$
and
$$ f(\theta) \equiv \indi{r\leq\theta}~\left(K_H(\theta,r) 
+ (\theta-r)^{H+\frac{1}{2}}~\indi{\{H<\frac{1}{2}\}} \right)
~\partial_\theta K_H(\theta,v)
~e^{-\frac{1}{2}\widetilde{\lambda}\theta}. $$
It comes:
$$ \EE(|\mathbf{D}_r J^{(4)}|^2)
\leq C~\indi{(0,s)}(v)
~(1+\lambda)^2~\rate_2(\mathbb{Y}-y_0,{\lambda})
~\left(\int_s^t |f(\theta)|~d\theta\right)^2
~e^{-\widetilde{\lambda}s}. $$
In view of~\eqref{def:A(r,v)} and~\eqref{def:I-natural(v,r,t)} 
we have
\begin{equation*}
\begin{split}
\int_0^s\int_0^t
\Big( \int_s^t |f(\theta)|~d\theta \Big)^2~dr~dv
&\leq C~e^{-\widetilde{\lambda}s} \int_0^s 
\int_0^t (\mathcal{A}^\natural(v,r,t))^2~dr~dv \\
&~~~+ C~e^{-\widetilde{\lambda}s} 
~\indi{\{H<\frac{1}{2}\}}~|H-\tfrac{1}{2}|^2
\int_0^s \int_0^t (\mathcal{I}^\natural(v,r,t))^2~dr~dv.
\end{split}
\end{equation*}
We now use~\eqref{ineq:L2-maj-mathcal-Abecarre(r,v)}
and~\eqref{ineq:L2-maj-Inatural(v,r,t)} and get
\begin{equation} \label{ineq:norme-L2-DrJ4}
\int_0^N \int_0^N \EE(|\mathbf{D}_r J^{(4)}|^2)~dr~dv 
\leq C~|H-\tfrac{1}{2}|^2~(t-s)^{2H}~
~(1+\lambda)^2~\rate_2(\mathbb{Y}-y_0,{\lambda})
(1+t^{2H+1})~e^{-\widetilde{\lambda}s}.
\end{equation}

\vspace{0.7cm}

To conclude the proof of~\eqref{ineq:Upsilon_t-Upsilon_s}, it remains 
to gather
\eqref{ineq:norme-L2-DrJ1}, \eqref{ineq:norme-L2-DrJ2},
\eqref{ineq:norme-L2-DrJ3-partie-1-1},
\eqref{ineq:norme-L2-DrJ3-partie-1},
\eqref{ineq:norme-L2-DrJ3-partie-2},
\eqref{ineq:norme-L2-DrJ4}.
\end{proof}

\subsection{The key lemma to estimate~$\mathbf{D}_r J^{(1)}$}\label{subsec:keyLemma}

\begin{lemma} \label{lem:gapKH}
For any $0<s<t < N$ with $0<t-s<1$ and $H\in(\frac{1}{4},1)$ it holds 
that
\begin{equation} \label{ineq:gapKH}
\int_s^t v^{2H}~(K_H(t,v) - 1)^2~dv \leq C~(H-\tfrac{1}{2})^2
~(t-s)^{(2H)\wedge 1} (1+ (\log(t-s))^2)~(1+t^2).
\end{equation}
\end{lemma}

\begin{proof}
Notice that $\chi_H$ is a bounded function of $H\in(\frac{1}{4},1)$.
We therefore have
$$ (K_H(t,v) - 1)^2 \leq C~\left( 
\left(\frac{t(t-v)}{v}\right)^{H-\frac{1}{2}}-1\right)^2
+ C~(H-\tfrac{1}{2})^2~v^{1-2H} \Big(\int_v^t
\theta^{H-\frac{3}{2}} (\theta-v)^{H-\frac{1}{2}}~d\theta \Big)^2. $$
We thus are led to consider
$$ R_1 := \int_s^t v~\big( t^{H-\frac{1}{2}}~(t-v)^{H-\frac{1}{2}} 
- v^{H-\frac{1}{2}}\Big)^2~dv $$
and
$$ R_2 := \Big(H-\tfrac{1}{2}\Big)^2\int_s^t v\Big(\int_v^t 
\theta^{H-\frac{3}{2}} (\theta-v)^{H-\frac{1}{2}}~d\theta \Big)^2~dv. $$

As for $R_1$ we have
\begin{equation*}
\begin{split}
R_1 &\leq 2t \int_s^t\Big( t^{H-\frac{1}{2}}~(t-v)^{H-\frac{1}{2}} 
- 1\Big)^2~dv + 2~\int_s^t v~(1-v^{H-\frac{1}{2}})^2~dv \\
&\leq 2~\Big( t^{2H}~\tfrac{1}{2H}~(t-s)^{2H} 
- 2t^{H+\frac{1}{2}}~\tfrac{1}{H+\tfrac{1}{2}}~(t-s)^{H+\frac{1}{2}}
+ t(t-s) \Big) \\
&~~~+ 2~\Big(\tfrac{1}{2}~(t^2-s^2) 
+ \tfrac{1}{2H+1}~(t^{2H+1}-s^{2H+1})
- \tfrac{2}{H+\tfrac{3}{2}}~(t^{H+\frac{3}{2}} - s^{H+\frac{3}{2}})
\Big).
\end{split}
\end{equation*}
We aim to use Taylor expansions of functions of 
$H-\frac{1}{2}$ around $H-\frac{1}{2}=0$. This leads us
to consider the following maps $\Psi_1(z)$ and $\Psi_2(z)$ for 
$z\in(-\frac{1}{4},\frac{1}{2})$: 
\begin{align*}
\Psi_1(z) := \tfrac{1}{1+2z}~(t(t-s))^{1+2z} 
- \tfrac{2}{1+z}~(t(t-s))^{1+z}+t(t-s), \\
\Psi_2(z) := \tfrac{1}{2}~(t^2-s^2) 
+ \tfrac{1}{2+2z}~(t^{2+2z}-s^{2+2z})
- \tfrac{1}{1+\tfrac{z}{2}}~(t^{2+z} - s^{2+z}).
\end{align*}
Observe that $\Psi_i(0)=\Psi'_i(0)=0$ for $i=1,2$. In addition, an easy 
calculation 
shows that $\Psi_1''(z)$ is a sum of terms of the type
$$ \frac{1}{(1+2z)^i}( (\log(t(t-s)))^j (t(t-s))^{1+2z}
~~~\text{or}~~~
\frac{1}{(1+z)^i}( (\log(t(t-s)))^j (t(t-s)^{1+z}) $$
with $i\in\{1,2,3\}$ and $j\in\{0,1,2\}$. 
Consequently, as $0<t-s<1$,
$$ \text{for}~H<\tfrac{1}{2},
~~\sup_{z\in(H-\frac{1}{2},0)}|\Psi_1''(z)| 
\leq C~(1+(\log(t-s))^2)~(t-s)^{2H}~(1 + t^2) $$
and
$$ \text{for}~H>\tfrac{1}{2},
~~\sup_{z\in(0,H-\frac{1}{2})}|\Psi_1''(z)| 
\leq C~(1+(\log(t-s))^2)~(t-s)~(1 + t^2). $$
Similarly,
$\Psi_2''(z)$ is a sum of terms of the type
$$ \frac{1}{(1+z)^i}( (\log(t))^j t^{2+2z} - (\log(s))^j s^{2+2z} )
~~~\text{or}~~~
\frac{1}{(1+\tfrac{z}{2})^i}( (\log(t))^j t^{2+z} - (\log(s))^j 
s^{2+z}) $$
with $i\in\{1,2,3\}$ and $j\in\{0,1,2\}$. Consequently, 
$$ \sup_{z\in(-\frac{1}{4},\frac{1}{2})}|\Psi_2''(z)| 
\leq C~(t-s)~(1 + t^2). $$
It therefore results from Taylor expansions of $\Psi_i$ that
\begin{equation} \label{ineq:R2}
\boxed{R_1 \leq C~(\Psi_1(H-\tfrac{1}{2}) + \Psi_2(H-\tfrac{1}{2}))\leq 
 C~(H-\tfrac{1}{2})^2(1+(\log(t-s))^2)~(t-s)^{(2H)\wedge 1}~(1 + t^2).}
\end{equation}

As for $R_2$ we observe that
$$ \theta^{H-\frac{3}{2}} \leq \frac{v^{H-1}}{\sqrt{\theta}}
\leq \frac{v^{H-1}}{\sqrt{\theta-v}}, $$
from which
$$ R_2 \leq C~(H-\tfrac{1}{2})^2 \int_s^t v^{2H-1}
~\Big(\int_v^t (\theta-v)^{H-1}~d\theta
\Big)^2~dv. $$
We thus get
\begin{equation} \label{ineq:R4}
\boxed{ R_2 \leq C~(H-\tfrac{1}{2})^2~(t-s)^{2H}~t^{2H}.}
\end{equation}

\end{proof}

\subsection{$L^p$-estimate on $(\Upsilon^{(N)}_t-\Upsilon^{(N)}_s)$}
\label{subsec:LpEstimates}

In this section we prove $L^p$-estimates on 
$\delta^{(N)}_H(U^{(N)}_t(\cdot)-U^{(N)}_s(\cdot))$. 
In the calculations below it will suffice to use the 
following estimate which results from~\eqref{ineq:DrW'YH}:
\begin{equation} \label{ineq:Dr-W-seconde}
\forall 0 \leq r,~~|\mathbf{D}_r(\wl'(Y^H_v))| \leq 
C~\indi{r\leq v}~(1+\lambda)~\big( |K_H(v,r)| 
+ (v-r)^{H+\frac{1}{2}} \big)~e^{|b'|_\infty v}.
\end{equation}

\begin{lemma}\label{lem:EstDeriv}
Suppose $\frac{1}{4}<H<1$ and $H\neq\frac{1}{2}$. 

Let $\Upsilon^{(N)}$ be defined as in~\eqref{def:Upsilon-Nt}.
For any $p\geq 2$ and $0<s<t<N$ with $0<t-s<1$ we have
\begin{equation} \label{ineq:Upsilon_t-Upsilon_s-Lp}
\big\|\Upsilon^{(N)}_t-\Upsilon^{(N)}_s\big\|_p
\leq C_H~|H-\tfrac{1}{2}|~(t-s)^{H\wedge\frac{1}{2}}
(1+\lambda)^2
~(1+ |\log(t-s)|)~(1+t^2)~e^{-\widetilde{\lambda}s}.
\end{equation}
\end{lemma}

\begin{proof}
We again consider~\eqref{eq:expansion-Ut-Us}. 
In view of Meyer's inequalities (\cite[Prop.3.2.1]{Nualart}) we have
$$ \big\|\Upsilon^{(N)}_t-\Upsilon^{(N)}_s\big\|_p
\leq C~\sum_{i=1}^4 \Big\{\int_0^N \EE( |J^{(i)}|^2)~dv 
\Big\}^{\frac{1}{2}}
+ C~\sum_{i=1}^4 \Big\{ \EE \Big( \Big( \int_0^N \int_0^N 
\big |\mathbf{D}_r 
J^{(i)} \big|^2~dr~dv\Big)^{\frac{p}{2}} \Big) \Big\}^{\frac{1}{p}}. $$
As in the proof of Lemma~\ref{L2-estimates} we limit ourselves to treat 
the second term. We start with applying Minkowski's inequality~\eqref{ineq:Minkowski} to get 
for $p\geq 2$
$$ \Big\{ \EE \Big( \Big( \int_0^N \int_0^N 
\big |\mathbf{D}_r 
J^{(i)} \big|^2~dr~dv\Big)^{\frac{p}{2}} \Big) \Big\}^{\frac{1}{p}}
\leq \Big(\int_0^N \int_0^N \Big\{ \EE \big(\big 
|\mathbf{D}_r 
J^{(i)} \big|^p \big) \Big\}^{\frac{2}{p}}~dr~dv 
\Big)^{\frac{1}{2}}. $$

In most of the calculations below we use~\eqref{ineq:Dr-W-seconde} and 
exhibit a \textsl{deterministic}
upper bound~$\mathcal{D}_r^{(i)}$ for~$|\mathbf{D}_r J^{(i)} \big|$.
We are thus reduced to use the $L^2$-estimates obtained in 
Section~\ref{subsec:lemma-L2-estimates} to get suitable upper bounds 
for
$$ \int_0^N \int_0^N |\mathcal{D}_r^{(i)}|^2~dr~dv~~~(1\leq i\leq 4). $$

\vspace{0.3cm}
\fbox{\textbf{A bound from above for 
$\int_0^N \int_0^N \Big\{ \EE \big(\big |\mathbf{D}_r 
J^{(1)} \big|^p \big) \Big\}^{\frac{2}{p}}~dr~dv$}}

\vspace{0.1cm}
Recall that
$$ \mathbf{D}_r J^{(1)} = \indi{(s,t]}(v)~\indi{r\leq v} 
~\mathbf{D}_r(\wl'(Y^H_v))~(K_H(t,v) - 1)~e^{-\lambda v} $$
and, by hypothesis, $\widetilde{\lambda}>0$. By 
using~\eqref{ineq:Dr-W-seconde} we thus are in a position to choose
$$ \mathcal{D}_r^{(1)} := C~\indi{(s,t]}(v)~\indi{r\leq v}
~(1+\lambda)~\big( |K_H(v,r)| 
+ (v-r)^{H+\frac{1}{2}} \big)~(K_H(t,v) - 1)
~e^{-\widetilde{\lambda} s}. $$

As in the proof of~\eqref{ineq:norme-L2-DrJ1} we 
use~\eqref{eq:norme-L2-K_H} and get
$$ \int_0^N\int_0^N \Big\{\EE(|\mathbf{D}_r 
J^{(1)}|^p)\Big\}^{\frac{2}{p}}~dr~dv 
\leq C~(1+\lambda)^2\, e^{-2\widetilde{\lambda}s} 
\int_s^t (v^{2H} + v^{2H+2})~(K_H(t,v) - 1)^2~dv. $$
In view of~\eqref{ineq:gapKH} we deduce that
\begin{equation}\label{ineq:norme-Lp-DrJ1} 
\int_0^N\int_0^N \Big\{\EE(|\mathbf{D}_r 
J^{(1)}|^p)\Big\}^{\frac{2}{p}}~dr~dv
\leq C~(H-\tfrac{1}{2})^2~(t-s)^{(2H)\wedge 1}~(1+ (\log(t-s))^2)
~(1+\lambda)^2~(1+t^4)\, e^{-2\widetilde{\lambda}s}.
\end{equation}

\vspace{0.3cm}
\fbox{\textbf{A bound from above for 
$\int_0^N \int_0^N \Big\{ \EE \big(\big |\mathbf{D}_r 
J^{(2)} \big|^p \big) \Big\}^{\frac{2}{p}}~dr~dv$}}

\vspace{0.1cm}
We have
$$ \mathbf{D}_r J^{(2)} = \indi{(s,t]}(v)~\indi{r\leq v}~
\mathbf{D}_r(\wl'(Y_v^H)) \int_v^t \partial_\theta K_H(\theta,v)
~(e^{-\widetilde{\lambda} \theta} 
- e^{-\widetilde{\lambda}  v})~d\theta. $$
By using~\eqref{ineq:Dr-W-seconde} and the hypothesis 
$\widetilde{\lambda}>0$ we are in a position to choose
$$ \mathcal{D}_r^{(2)} := C~\indi{r\leq v}~(1+\lambda)
~\big( |K_H(v,r)| + (v-r)^{H+\frac{1}{2}} \big)
\int_v^t \partial_\theta K_H(\theta,v)
~(e^{-\widetilde{\lambda} \theta} 
- e^{-\widetilde{\lambda}  v})~d\theta. $$
Now proceed as in the proof of~\eqref{ineq:norme-L2-DrJ2}. It comes:
\begin{equation*} 
\begin{split}
\int_0^N \Big\{ \EE \big(\big |\mathbf{D}_r 
J^{(2)} \big|^p \big) \Big\}^{\frac{2}{p}}~dr 
&\leq C~\indi{(s,t]}(v)
~|H-\tfrac{1}{2}|^2~(t-s)^{2H+1}~(1+\lambda)^2\\
&\hspace{1.5cm}\times (v^{2H} + v^{2H+1})
~\Big(1 + \frac{t^{2H-1}}{v^{2H-1}}
~\indi{\tfrac{1}{2}<H<1}\Big) \, e^{-2\widetilde{\lambda} v}\\
&\leq C~\indi{(s,t]}(v)
~|H-\tfrac{1}{2}|^2~(t-s)^{2H+1}~(1+\lambda)^2~(1+t^{2H+2}) \, 
e^{-2\widetilde{\lambda} s}, 
\end{split}
\end{equation*}
from which
\begin{equation}\label{ineq:norme-Lp-DrJ2} 
\int_0^N \int_0^N \Big\{ \EE \big(\big |\mathbf{D}_r 
J^{(2)} \big|^p \big) \Big\}^{\frac{2}{p}}~dr~dv
\leq C~|H-\tfrac{1}{2}|^2~(t-s)^{2H+2}~(1+\lambda)^2~(1+t^4)\, 
e^{-2\widetilde{\lambda} s}.
\end{equation}

\vspace{0.3cm}
\fbox{\textbf{A bound from above for 
$\int_0^N \int_v^N \Big\{ \EE \big(\big |\mathbf{D}_r 
J^{(3)} \big|^p \big) \Big\}^{\frac{2}{p}}~dr~dv$}}

\vspace{0.1cm}

We proceed as above. From
$$ |\indi{v\leq r}~\mathbf{D}_r J^{(3)}| = \indi{(s,t]}(v)~\indi{v\leq 
r} 
~\Big| \int_r^t \mathbf{D}_r \wl'(Y^H_\theta)
~\partial_\theta K_H(\theta,v)
~e^{-\widetilde{\lambda} \theta}~d\theta \Big| $$
we deduce that we can choose
$$\mathcal{D}_r^{(3)} := C~\indi{(s,t]}(v)~\indi{v\leq r}~(1+\lambda)
\, e^{-\widetilde{\lambda}  s} \int_r^t (K_H(\theta,r) 
+ (\theta-r)^{H+\frac{1}{2}}~\indi{\{H<\frac{1}{2}\}})
~\Big| \partial_\theta K_H(\theta,v) \Big|~d\theta. $$
Use~\eqref{ineq:f(theta)-L2-J3}. It comes:
$$ \int_0^N\int_v^N \Big\{ \EE \big(\big |\mathbf{D}_r 
J^{(3)} \big|^p \big) \Big\}^{\frac{2}{p}}~dr~dv
\leq C_H~|H-\tfrac{1}{2}|^2~(t-s)~(1+\lambda)^2~(1+t^3)\, 
e^{-2\widetilde{\lambda} s}. $$

\vspace{0.3cm}
\fbox{\textbf{A bound from above for 
$\int_0^N \int_0^v \Big\{ \EE \big(\big |\mathbf{D}_r 
J^{(3)} \big|^p \big) \Big\}^{\frac{2}{p}}~dr~dv$}}

\vspace{0.1cm}

As in the proof of~\eqref{ineq:norme-L2-DrJ3-partie-1},
for $r\leq v$ we consider
$$ \mathbf{D}_r J^{(3)}_1 := \indi{(s,t]}(v)
\int_v^t \partial_\theta K_H(\theta,v)
~(\wl''(Y_\theta^H)-\wl''(Y_v^H))~\mathbf{D}_r Y_\theta^H 
~e^{-\lambda\theta}~d\theta $$
and
$$ \mathbf{D}_r J^{(3)}_2 := \indi{(s,t]}(v)
\int_v^t \partial_\theta K_H(\theta,v)
~\wl''(Y_v^H)~(\mathbf{D}_r Y_\theta^H - \mathbf{D}_r Y_v^H)
~e^{-\lambda\theta}~d\theta. $$

We have
$$|\mathbf{D}_r J^{(3)}_1| \leq C~\indi{(s,t]}(v)~(1+\lambda)^2 
~e^{-\widetilde{\lambda} s}
\int_v^t \left|\partial_\theta K_H(\theta,v)\right|
~|Y_\theta^H - Y_v^H| 
\left\{ K_H(\theta,r) + (\theta-r)^{H+\frac{1}{2}} 
~\indi{\{H<\frac{1}{2}\}} \right\}~d\theta. $$
From Minskowski's inequality it results that
\begin{equation*}
\begin{split}
&\Big\{ \EE (|\mathbf{D}_r J^{(3)}_1|^p) \Big\}^{\frac{1}{p}} \\
&\leq C~(1+\lambda)^2~e^{-\widetilde{\lambda} s} \int_v^t  
\Big\{ \EE \big(|Y_\theta^H - Y_v^H|^p\big) \Big\}^{\frac{1}{p}}
~\Big|\partial_\theta K_H(\theta,v)\Big|
\left\{ K_H(\theta,r) + (\theta-r)^{H+\frac{1}{2}} 
~\indi{\{H<\frac{1}{2}\}} \right\}~d\theta \\
&\leq C~(1+\lambda)^2~e^{-\widetilde{\lambda} s}
\int_v^t  (\theta-v)^{H}~\Big|\partial_\theta K_H(\theta,v)\Big|
\left\{ K_H(\theta,r) + (\theta-r)^{H+\frac{1}{2}} 
~\indi{\{H<\frac{1}{2}\}} \right\}~d\theta.
\end{split}
\end{equation*}
We now use our estimate on the weighted $L^2$-norm of the function 
$\phi(\theta,v)$ chosen in the proof 
of~\eqref{ineq:norme-L2-DrJ3-partie-1} to get
\begin{equation}\label{ineq:norme-Lp-DrJ3-partie-1} 
\int_0^N \int_0^v \Big\{ \EE \big(\big |\mathbf{D}_r 
J^{(3)}_1 \big|^p \big) \Big\}^{\frac{2}{p}}~dr~dv
\leq C_H~|H-\tfrac{1}{2}|^2~(t-s)^{4H\wedge(2H+1)}~(1+\lambda)^4
~(1+t^4)~e^{-2\widetilde{\lambda} s}.
\end{equation}

\vspace{0.5cm}

Similarly, for $r\leq v$ we have 
\begin{equation*}
\begin{split}
|\mathbf{D}_r J^{(3)}_2|
&\leq C~\indi{(s,t]}(v)~(1+\lambda)\, e^{-\widetilde{\lambda} s} 
\int_v^t \left|\partial_\theta K_H(\theta,v)\right| \\
&~~~~~~~~~~~~~~~~~~~~~~~~~~~~~~
\left\{ |K_H(\theta,r) - K_H(v,r)|
+ (\theta-v) \left(K_H(v,r)+(v-r)^{H+\frac{1}{2}} 
~\indi{\{H<\frac{1}{2}\}}\right)\right\}~d\theta.
\end{split}
\end{equation*}
We now use our estimate on the weighted $L^2$-norm of the function 
$f(\theta)$ chosen in the proof 
of~\eqref{ineq:norme-L2-DrJ3-partie-2} to get
\begin{equation}\label{ineq:norme-Lp-DrJ3-partie-2} 
\int_0^N \int_0^v \Big\{ \EE \big(\big |\mathbf{D}_r 
J^{(3)}_2 \big|^p \big) \Big\}^{\frac{2}{p}}~dr~dv
\leq C_H~|H-\tfrac{1}{2}|^2~(t-s)^{4H\wedge(2H+1)}
~(1+\lambda)^2~(1+t^4)~e^{-2\widetilde{\lambda}  s}.
\end{equation}

\vspace{0.3cm}

\fbox{\textbf{A bound from above for $\int_0^N \int_0^N \Big\{ \EE 
\big(\big |\mathbf{D}_r 
J^{(4)} \big|^p \big) \Big\}^{\frac{2}{p}}~dr~dv$ }}

\vspace{0.1cm}

We obviously can choose
$$ \mathcal{D}_r^{(4)} 
:= C~\indi{(0,s)}(v)~(1+\lambda)\, e^{-\widetilde{\lambda}  s}\Big|
\int_s^t \indi{r\leq\theta}~\partial_\theta K_H(\theta,v)
~(K_H(\theta,r) + (\theta-r)^{H+\frac{1}{2}}~\indi{\{H<\frac{1}{2}\}})
~d\theta \Big|. $$
Proceed as in the proof of~\eqref{ineq:norme-L2-DrJ4} to obtain

\begin{equation} \label{ineq:norme-Lp-DrJ4}
\int_0^N \int_0^N \Big\{ \EE \big(\big |\mathbf{D}_r 
J^{(4)} \big|^p \big) \Big\}^{\frac{2}{p}}~dr~dv
\leq C~|H-\tfrac{1}{2}|^2~(t-s)^{2H}~(1+\lambda)^2
~(1+t^{2H+1})~e^{-2\widetilde{\lambda} s}.
\end{equation}

\vspace{0.7cm}

To conclude the proof of~\eqref{ineq:Upsilon_t-Upsilon_s-Lp}, it 
remains to gather
\eqref{ineq:norme-Lp-DrJ1}, \eqref{ineq:norme-Lp-DrJ2},
\eqref{ineq:norme-Lp-DrJ3-partie-1},
\eqref{ineq:norme-Lp-DrJ3-partie-2},
\eqref{ineq:norme-Lp-DrJ4}.
\end{proof}

\section{Application to weak convergence rates when $H\to \frac{1}{2}$}
\label{sec:examples}

In this section, we apply Theorem~\ref{prop:majgap_0} to estimate the weak convergence rate of $\tau^H_X$ towards $\tau_{\mathbf{X}}$ when~$H$ tends to $\frac{1}{2}$. 

In the subsection~\ref{subsec:WCR} we explain how Theorem~\ref{prop:majgap_0} can be used to prove the vague, respectively the weak convergence, of~$\tau_{X}^H$.  We introduce
the Aletti metric which allows one to quantify the weak convergence in terms of $|H-\frac{1}{2}|$.
Then, in the subsection~\ref{subsec:hypo-DR} we consider the Lamperti transform $Y^H$ of $X^H$. We exhibit sufficient conditions on
the Malliavin derivatives $D^H_{\cdot} Y^H_{t}$ and on the tail probability of $\tau^H_Y$ which allow one to get a weak convergence rate of $\tau^H_X$ towards $\tau_{\mathbf{X}}$ of order~$|H-\frac{1}{2}|$.  Finally, in the subsection~\ref{subsec:cond-drift}
we provide sufficient conditions on the drift coefficient which imply the conditions exhibited 
in the subsection~\ref{subsec:hypo-DR}.

\subsection{On the vague convergence and on the weak convergence rate of
$\tau^H_X$} \label{subsec:WCR}

Let $(\mu_{n})_{n\in \N}$ and $\mu$ be locally finite measures on $\R^d$.  The
sequence~$\mu_n$ is said to converge vaguely to~$\mu$ if $\int f(x) \mu_n(dx)$
converges to $\int f(x) \mu(dx)$ for any continuous and bounded function~$f$
vanishing at infinity (see e.g. \cite[p.66]{BW}).

If for some $\lambda^\sharp\geq 0$ the Laplace transform $\mathcal{L}_{\mu_{n}}$ of
$\mu_{n}$ converges pointwise to some function $\mathcal{L}$ on the interval
$(\lambda^\sharp,+\infty)$, then $(\mu_{n})_{n\in \N}$  converges
vaguely~\cite[Thm.8.5(a)]{BW}. 

Even if the $\mu_{n}$'s are probability measures,  the vague convergence is not equivalent to the weak convergence since the limit
measure~$\mu$ may be \emph{defective} in the sense that its mass is strictly less
than~$1$. However, if $\mu(\R^d)=1$, then the vague convergence to $\mu$ of probability measures $\mu_{n}$ is equivalent to  their weak convergence
(see e.g.~\cite[Lem 5.20]{Kallenberg}).

Theorem~\ref{prop:majgap_0} implies the vague convergence of $\tau_{X}^H$ to
$\tau_{\mathbf{X}}$ as $H\to\frac{1}{2}$ since  it shows that $\EE e^{-\lambda
\tau_{X}^H}$ tends to $\EE e^{-\lambda \tau_{\mathbf{X}}}$ for any
$\lambda>|\widetilde{b}'|_{\infty}$.
Notice that one cannot expect the weak convergence for any drift. For instance,
the process $X^H_{t} = -t-A+ B^H_{t}$ satisfies, for some $A>0$ large enough to be precised below,
$$\PP(\inf\{t\geq 0:~ X^H_{t}=1\} = \infty)>0,~\forall H<1. $$
Indeed, 
\begin{align*}
\PP(\inf\{t\geq 0:~ X^H_{t}=1\} = \infty) &= \PP( \forall t\geq0,~ B_{t}<1+A+t) \\
&\geq \PP \left( \forall T\geq 0, ~\frac{\sup_{t\leq T} |B_{t}|}{1+ T^H|\log T|^2} < \frac{1+A+ T}{1+ T^H|\log T|^2} \right).
\end{align*}
We know e.g. from \cite[Prop.3.1]{Kubilius} that there exists an a.s. finite random variable $\xi(H)$ such that a.s., $\forall T\geq 0, ~\frac{\sup_{t\leq T} |B_{t}|}{1+ T^H|\log T|^2} \leq \xi(H)$. Hence,
$$\PP(\inf\{t\geq 0:~ X^H_{t}=1\} = \infty) \geq \PP(\forall T\geq 0,~ \xi(H)< \frac{1+A+T}{1+ T^H|\log T|^2} ) \geq \PP( \xi(H)< A) ,$$
where the last probability is positive for $A$ large enough.
Therefore, in this example the law of $\tau_{X}^H$ is a subprobability measure,
including in the limit case $H=\frac{1}{2}$. 

As mentioned above, to get the weak convergence of $\tau_{X}^H$ to 
$\tau_{\mathbf{X}}$ as $H\to\frac{1}{2}$ one needs that the probability distributions of 
$\tau_X^H$ and $\tau_{\mathbf{X}}$ are non defective.
 In that case, the Laplace transforms of $\tau_{X}^H$  converge on $[0,+\infty)$.

When the preceding conditions are satisfied, a weak convergence rate of
$\tau^H_X$ to~$\tau_{\mathbf{X}}$ can be obtained by using a  distance
introduced by \citet{Aletti}. For nonnegative random variables $Z_{1}$ and
$Z_{2}$, Aletti's distance between their probability distribution is defined as
\begin{align*}
d_{A}(Z_{1}, Z_{2}) := \inf\left\{\varepsilon>0:~ \sup_{\lambda\in \R_{+}} \max\left( \EE 
e^{-\lambda Z_{1}} - \left(\varepsilon + \EE e^{-\lambda Z_{2}}\right) 
e^{\lambda\varepsilon} ,
\EE e^{-\lambda Z_{2}} - \left(\varepsilon + \EE e^{-\lambda Z_{1}}\right)  e^{\lambda\varepsilon} \right) \leq 0 \right\} .
\end{align*}
This distance metrises the weak convergence topology on the space of probability 
measures supported on $\R_{+}$: See~\cite[Thm. 2]{Aletti}.

In the estimate \eqref{ineq:Laplace-transform-main-theorem} the 
$\frac{1}{1\wedge\widetilde{\lambda}^3}$ term is natural since
without additional hypotheses the laws of $\tau^H_X$ and~$\tau_{\mathbf{X}}$ may be defective and, therefore, their Laplace transforms may be discontinuous at~0. In the following proposition the conclusion of Theorem~\ref{prop:majgap_0} is reinforced in order to be consistent with the non defectiveness of the laws of $\tau^H_X$ and~$\tau_{\mathbf{X}}$ and to allow
one to quantify the weak convergence by means of Aletti's distance.

\begin{proposition} \label{prop:estim-aletti}
Let $b$ and $\sigma$ be as in Theorem~\ref{prop:majgap_0}.
Suppose that  for any $H\in(\tfrac{1}{4},1)$ and any $\lambda\geq0$ one has
\begin{equation}\label{eq:boundwithoutlambda}
\Big|\EE\left(e^{-\lambda \tau^H_X}\right) 
- \EE\left(e^{-\lambda \tau_{\mathbf{X}}}\right)\Big| \\
\leq C_H\, |H-\tfrac{1}{2}| \, \Phi(H,\lambda), 
\end{equation}
where the function $\Phi$ satisfies $\sup_{H\in (1/4,1)} \sup_{\lambda\in \R_{+}} \Phi(H,\lambda) <\infty $ and, as above, $C_H$ denotes a constant which depends on $H$ only and is locally bounded in the open interval~$(\frac{1}{4},1)$.

Suppose in addition that the probability laws of $\tau_{X}^H$ and 
$\tau_{\mathbf{X}}$ are non defective.  
Then, $\tau_{X}^H$ weakly converges to $\tau_{\mathbf{X}}$ as $H$ tends to $\frac{1}{2}$ and
\begin{align*}
d_{A}(\tau_{X}^H, \tau_{\mathbf{X}})  \leq C_{H}\, |H-\tfrac{1}{2}|, \quad \forall H\in(\tfrac{1}{4},1).
\end{align*}
\end{proposition} 

\begin{proof}
Set $\delta_{H} := C_{H}\, |H-\tfrac{1}{2}|$
and $\varepsilon  := \delta_{H}~\sup_{H\in (1/4,1)} \sup_{\lambda\in \R_{+}} \Phi(H,\lambda) <\infty$. 
For any $\lambda\geq 0$ one has
\begin{align*}
 \EE e^{-\lambda\tau_{X}^H} - \left(\varepsilon + \EE e^{-\lambda \tau_{\mathbf{X}}}\right) e^{\lambda\varepsilon} 
 &\leq  \EE e^{-\lambda\tau_{X}^H} - \varepsilon\, e^{\lambda\varepsilon} - \EE e^{-\lambda\tau_{X}^H}\, e^{\lambda\varepsilon} + \delta_{H}~\Phi(H,\lambda) 
 ~e^{\lambda\varepsilon}\\
 &\leq (1- e^{\lambda\varepsilon})\, \EE e^{-\lambda\tau_{X}^H}\\
 &\leq 0.
\end{align*}
Similarly, one proves that
$ \EE e^{-\lambda \tau_{\mathbf{X}}} - \left(\varepsilon + \EE e^{-\lambda \tau_{X}^H}\right) e^{\lambda\varepsilon} \leq 0$.
Hence, $d_{A}(\tau_{X}^H, \tau_{\mathbf{X}})\leq \varepsilon$, which is the desired result.
\end{proof}

We now exhibit a sufficient condition on $D^H_\cdot Y^H_t$ and on the tail probability of $\tau_{Y}^H$ which implies the inequality~\eqref{eq:boundwithoutlambda}.

\subsection{Sufficient conditions on $D^H_\cdot Y^H_t$ and the tail probability of $\tau_{Y}^H$ for~\eqref{eq:boundwithoutlambda}} 
\label{subsec:hypo-DR}

The aim of this subsection is to obtain an improvement of the estimate in Theorem~\ref{prop:majgap_0}
under suitable conditions on the tail of the probability distribution of~$\tau^H_Y $ and on the process $\varpi_{H}(r,t)$ defined in~\eqref{eq:varpiH}  in terms of $D^H_\cdot Y^H_t$. We set
\begin{equation} \label{def:esp-varpi_H}
\Pi_{H}(t) : = \sqrt{\EE |\varpi_H(0,t+1)|^2}.
\end{equation}

\begin{theorem} \label{prop:majgap_0_bis}
In addition to the hypotheses made in Theorem~\ref{prop:majgap_0}, assume 
\begin{equation} \label{cond:tail-tau_H_Y}
\exists \vartheta>2,~~\forall H\in(\tfrac{1}{4},1),~~ \PP(\tau^H_Y \geq t) \leq \frac{C_H}{(1+t)^{2\vartheta}},\quad \forall t\geq 0,
\end{equation}
and
\begin{equation} \label{cond:Pi_H}
\forall H\in(\tfrac{1}{4},1),~~ \int_{\R_{+}} \sqrt{\PP(\tau^H_Y \geq t)}~\Pi_H(t)~(1+t^2)~dt \leq C_{H}. 
 \end{equation}

For any $p\geq1$ and $\lambda>0$ set
\begin{equation*}
\ratet_p(\mathbb{Y}-y_0,\lambda) := \sup_{s\in \R_+} 
\left( e^{- \lambda p s}
~\EE~e^{-|\mathbb{Y}-Y^H_s|p\mathcal{R}(\lambda)} \right),
\end{equation*} 
where
\begin{equation*}
\mathcal{R}(\lambda):=\sqrt{2\lambda +\mu^2} - \mu
~~\text{with}~~\mu:=|\widetilde{b}|_\infty.
\end{equation*}

For any $\forall H\in(\tfrac{1}{4},1)$ and $\lambda>0$ we then have  
\begin{multline} \label{ineq:Laplace-transform-main-theorem-hypo-renforcees}
\Big|\EE\left(e^{-\lambda \tau^H_X}\right) 
- \EE\left(e^{-\lambda \tau_{\mathbf{X}}}\right)\Big| \\
\leq C_H~|H-\tfrac{1}{2}|
~(1+\lambda)^2
~\Big( \ratet_{1}(\mathbb{Y}-y_0,{\lambda}) + 
\left(\ratet_{2}(\mathbb{Y}-y_0,{\lambda}) \right)
^{\frac{H\wedge\frac{1}{2}}{6}}
+ \left(\ratet_{4}(\mathbb{Y}-y_0,{\lambda})\right)
^{\frac{H\wedge\frac{1}{2}}{12}} \Big). 
\end{multline}
\end{theorem}

\begin{proof}
The proof of~\eqref{ineq:Laplace-transform-main-theorem-hypo-renforcees} consists in suitably modifying the proof of Theorem~\ref{prop:majgap_0}. We start with the decomposition stated in Proposition~\ref{prop:gaps}:
\begin{align*} %
\EE\Big(e^{-\lambda \tau^H_Y}\Big) - \EE\Big(e^{-\lambda 
\tau_{\mathbf{Y}}}\Big) = I_1(\lambda) + I_2(\lambda).
\end{align*}
Propositions~\ref{prop:I1'} and \ref{prop:I2'} below provide desired estimates
on each one of these two terms.
\end{proof}

\paragraph{Bound on $I_{1}$.}
The proposition~\ref{prop:I1'} can be improved as follows.

\begin{proposition}\label{prop:I1'-sec6}
Under the hypotheses of Theorem~\ref{prop:majgap_0_bis}  for any $\lambda>0$ 
one has
\begin{equation} \label{ineq:I_1-sec6}
|I_1(\lambda)| \leq C_H~(1+\lambda)
~|H-\tfrac{1}{2}|~\sqrt{\ratet_{2}(\mathbb{Y}-y_{0},{\lambda})}.
\end{equation}
\end{proposition}

\begin{proof}
Recall the definition~\eqref{eq:Delta}:
$$  \Delta(s,H) := Hs^{2H-1}-\tfrac{1}{2} 
+ \int_0^s \partial_s 
K_H(s,r)~\int_{0}^s\mathbf{D}_r\widetilde{b}(Y^H_{v})~dv~dr.  $$
By using Minkowski's inequality~\eqref{ineq:Minkowski} we get
$$ \EE|\Delta(s,H)|^2 \leq 2(Hs^{2H-1}-\tfrac{1}{2} )^2
+ 2 \int_0^s \partial_s 
K_H(s,r)~\sqrt{\EE\Big|\int_{0}^s\mathbf{D}_r\widetilde{b}(Y^H_{v})~dv\Big|^2}~dr.  $$
We now use~\eqref{D_rYneg} and get
\begin{align*} %
\sqrt{\EE |\Delta(s,H)| ^2} \leq \left|Hs^{2H-1}-\tfrac{1}{2}\right|
+C_H~\left|H-\tfrac{1}{2}\right| \Pi_H(s)~(1 + s^{2}).
\end{align*}
Therefore, in view of Inequality~\eqref{maj:wl-d2} one has
\begin{equation} \label{I1-majoration-sec6}
\begin{split}
|I_1(\lambda)| &= \left|\int_0^\infty e^{-\lambda s} 
~\EE\left(\Delta(s,H) \, \indi{\{\tau^H_Y \geq s\}} \wl''(Y_s^H)\right)ds\right| \\
& \leq C~(1+\lambda)
~\sqrt{\ratet_2(\mathbb{Y}-y_0,{\lambda})}
~\int_0^\infty 
\sqrt{\PP\{\tau^H_Y \geq s\}}~
e^{-\frac{1}{2}{\lambda}s}~
\big|Hs^{2H-1}-\tfrac{1}{2}\big|~ds \\
&\quad + C_H~|H-\tfrac{1}{2}|~(1+\lambda)
~\sqrt{\ratet_2(\mathbb{Y}-y_0,{\lambda})}
\int_0^\infty \sqrt{\PP\{\tau^H_Y \geq s\}}
~e^{-\frac{1}{2}\lambda s} \, \Pi_H(s) ~(1+s^{2})~ds.
\end{split}
\end{equation} 
Consider the first integral in the right-hand side of the last inequality. Split the integral
$$ \int_0^\infty \sqrt{\PP\{\tau^H_Y \geq s\}}~
e^{-\frac{1}{2}{\lambda}s}~
\big|Hs^{2H-1}-\tfrac{1}{2}\big|~ds $$
into integrals from $0$ to $\alpha := (\tfrac{1}{2H})^{\frac{1}{2H-1}}$
and from $\alpha$ to $+\infty$. This leads one to consider
\begin{align*}
I_{11} &:= 
-\int_{0}^\alpha 
 \sqrt{\PP\{\tau^H_Y \geq s\}}~
e^{-\frac{1}{2}{\lambda}s}~
\left(\tfrac{1}{2}-Hs^{2H-1}\right)~ds, \\
I_{12} &:= 
-\int_{\alpha}^{+\infty} 
~ \sqrt{\PP\{\tau^H_Y \geq s\}}~
e^{-\frac{1}{2}{\lambda}s}
\left(Hs^{2H-1}- \tfrac{1}{2}\right)~ds.
\end{align*}
As for $I_{11}$, bound the exponential from above by $1$ and use \eqref{cond:tail-tau_H_Y}. Then, integrate by parts and use that $1-\alpha^{2H-1} = 
\frac{1}{H}(H-\frac{1}{2})$. It comes: 
\begin{align*}
I_{11} \leq \frac{C\alpha}{2(1+\alpha)^{\vartheta}}
~
\frac{1}{H}(H-\tfrac{1}{2}) + C~\int_{0}^\alpha 
  \frac{1}{(1+s)^{\vartheta+1}}~(s-s^{2H})~ds.
\end{align*}
Observe that $\alpha$ is a bounded function of~$H\in(\tfrac{1}{4},1)$. 
In addition, for any~$s\in[0,\alpha]$ apply the Mean Value Theorem to 
the map $H\in(\tfrac{1}{4},1)\mapsto s-s^{2H} 
= s-s^{1+2(H-\frac{1}{2})}$ around the 
point~$H=\tfrac{1}{2}$. It comes:
\begin{equation}\label{sec6-eq:I1-2}
\begin{split}
I_{11} &\leq C~|H-\tfrac{1}{2}|
+ C |H-\tfrac{1}{2}|~\sup_{s\in [0,\alpha]} \sup_{\gamma\in 
(-\frac{1}{4},\frac{1}{2})} (|\log(s)|~s^{1+2\gamma})~ 
\int_{0}^\alpha  \frac{1}{(1+s)^{\vartheta+1}}~ds \\
&\leq C~|H-\tfrac{1}{2}|.
\end{split}
\end{equation}
As for $I_{12}$, we again use \eqref{cond:tail-tau_H_Y} and bound the exponential from above by $1$ to obtain
$$ I_{12} \leq -C~\int_{\alpha}^{+\infty} 
\frac{1}{(1+s)^{\vartheta}}\left(Hs^{2H-1}- \tfrac{1}{2}\right)~ds. $$
Then proceed as above. First, integrate by parts. Second, use that
$1-\alpha^{2H-1} = 
\frac{1}{H}(H-\frac{1}{2})$. Last, apply the Mean Value Theorem to 
$H\in(\tfrac{1}{4},1)\mapsto s-s^{2H} 
= s-s^{1+2(H-\frac{1}{2})}$ around~$H=\tfrac{1}{2}$ and use that
$$ \exists C>0,~\forall s>0,~\forall 
\gamma\in(-\tfrac{1}{4},\tfrac{1}{2}),
~~~|\log(s)|~(s^{1+2\gamma}) \leq C~(1+s^2). $$
It comes, using that $\vartheta>2$:
\begin{equation} \label{sec6-eq:I1-3}
\begin{split}
I_{12}&\leq C~|H-\tfrac{1}{2}|
+ C~|H-\tfrac{1}{2}|~\int_\alpha^\infty~ 
 \frac{1+s^2}{(1+s)^{\vartheta+1}}~ds.  \\
&\leq C~|H-\tfrac{1}{2}|.
\end{split}
\end{equation}

Now, consider the second integral in the right-hand side of~\eqref{I1-majoration-sec6}. Bound the exponential from above by~$1$. Then  use \eqref{cond:tail-tau_H_Y} and \eqref{cond:Pi_H}. It comes:
\begin{align*}
C_H~&|H-\tfrac{1}{2}|~(1+\lambda)
~\sqrt{\ratet_2(\mathbb{Y}-y_0,\lambda)}
\int_0^\infty \sqrt{\PP(\tau^H_Y \geq s)}
~\Pi_H(s) ~(1+s^{2})~ds\\
&\leq C_H~|H-\tfrac{1}{2}|~(1+\lambda)
~\sqrt{\ratet_2(\mathbb{Y}-y_0,\lambda)}.
\end{align*}
This observation combined with \eqref{sec6-eq:I1-2} and 
\eqref{sec6-eq:I1-3} provides the desired result~\eqref{ineq:I_1-sec6}.
\end{proof}

\paragraph{Bound on $I_{2}$.} Recall that
$$ I_2(\lambda) := \lim_{N\rightarrow +\infty} \EE\Big[ 
\left.\delta_H^{(N)}\left(\indi{[0,t]}(\cdot) e^{-\lambda \cdot} 
\wl'(Y_\cdot^H) \right)\right|_{t=\tau^H_Y\wedge N}\Big]. $$

Under the preceding hypotheses the proposition~\ref{prop:I2'} can be improved as follows.
\begin{proposition}\label{prop:I2'-sec6}
Under the hypotheses of Theorem~\ref{prop:majgap_0_bis}  for any $\lambda>0$ 
one has
\begin{equation*}
|I_2(\lambda)| 
\leq 
C_H~|H-\tfrac{1}{2}|
~(1+\lambda)^2~
\Big( \left(\ratet_{2}(\mathbb{Y}-y_0,{\lambda}) \right)
^{\frac{H\wedge\frac{1}{2}}{6}}
+ \left(\ratet_{4}(\mathbb{Y}-y_0,{\lambda})\right)
^{\frac{H\wedge\frac{1}{2}}{12}} \Big).
\end{equation*}
\end{proposition}

\begin{proof}
Recall that the field $\{U^{(N)}_t(v),v\geq0,t>0\}$ 
and the process $\{\Upsilon^{(N)}_t, t>0\}$ are defined by \eqref{def:U-Nt} and \eqref{def:Upsilon-Nt}. 
Recall also that $ I_2(\lambda) := \lim_{N\rightarrow \infty} \EE ( \Upsilon^{(N)}_{\tau^H_Y\wedge N} ). $
As $\Upsilon^{(N)}_0=0$, 
for any $t>0$ we have
$$ \Upsilon^{(N)}_t = \Upsilon^{(N)}_t - \Upsilon^{(N)}_{[t]}
+ \sum_{n=1}^{[t]}(\Upsilon^{(N)}_n - 
\Upsilon^{(N)}_{n-1})~\indi{t\geq1}. $$
Therefore, for any $p_0>1$ to be chosen later one has
\begin{equation*}\label{sec6-eq:1stBoundI2}
\begin{split}
|I_2(\lambda)| 
&\leq  \lim_{N\rightarrow \infty}~ \sum_{n=0}^{N-1}
\EE\big( \sup_{t\in [n,n+1]} |\Upsilon^{(N)}_t - 
\Upsilon^{(N)}_n|~\indi{\tau^H_Y\geq n} \big) \\
&\leq  \lim_{N\rightarrow \infty}~ \sum_{n=0}^{N-1}
\Big(\EE\big( \sup_{t\in [n,n+1]} |\Upsilon^{(N)}_t - 
\Upsilon^{(N)}_n|^{p_0}\big)\Big)^{\frac{1}{p_0}}
~\big(\PP(\tau^H_Y\geq n)\big)^{\frac{p_0-1}{p_0}}.
\end{split}
\end{equation*}
In order to estimate the right-hand side of the preceding inequality
we apply the corollary of Garsia-Rodemich-Rumsey's lemma given in Lemma~\ref{lem:GRR}. 
We thus obtain, for $p=p_{0}$ and $p_{0}q>2$,
\begin{equation}\label{sec6-eq:1stBoundI2-bis}
|I_2(\lambda)| 
\leq  C~\lim_{N\rightarrow \infty}~ \sum_{n=0}^{N-1}
\Big(\frac{q}{q-2}\Big)^{\frac{1}{p_0}} 
\left(\int_n^{n+1} \int_n^{n+1} 
\frac{\EE\big(|\Upsilon^{(N)}_t-\Upsilon^{(N)}_s|^{p_0}\big)}
{|t-s|^{p_{0}q}}\ ds~dt\right)^{\frac{1}{p_0}}
~\big(\PP(\tau^H_Y\geq n)\big)^{\frac{p_0-1}{p_0}}.
\end{equation}
We now use the inequality
\begin{equation*}
~~\EE\left(|\Upsilon^{(N)}_t-\Upsilon^{(N)}_s|^{p_0}\right) 
\leq \left(\EE|\Upsilon^{(N)}_t-\Upsilon^{(N)}_s|^{2(p_0-1)}
\right)^{\frac{1}{2}} \times 
\left\|\Upsilon^{(N)}_t-\Upsilon^{(N)}_s\right\|_2.
\end{equation*}
Then, obvious modifications of the subsections~\ref{subsec:lemma-L2-estimates} 
and~\ref{subsec:LpEstimates} and the conditions~\eqref{cond:tail-tau_H_Y},
\eqref{cond:Pi_H} lead to the following: For any $0<s<t<N$ with $0<t-s<1$ it holds that
\begin{equation*}
\begin{split}
\left\|\Upsilon^{(N)}_t-\Upsilon^{(N)}_s\right\|_2
\leq C_H~|H-\tfrac{1}{2}|&~(t-s)^{H\wedge\frac{1}{2}} 
~(1+|\log(t-s)|)~(1+\lambda)^2 
~(1+t^2)~e^{-\frac{1}{2} \lambda s}~\Pi_H(s)
\\
&\times\Big(
\left(\ratet_{2}(\mathbb{Y}-y_0,{\lambda}) \right)^\frac{1}{2}
+\left(\ratet_{4}(\mathbb{Y}-y_0,{\lambda})\right)^\frac{1}{4}
\Big),
\end{split}
\end{equation*}
and for any $p_0\geq 2$,
\begin{equation*}
\begin{split}
\left(\EE|\Upsilon^{(N)}_t-\Upsilon^{(N)}_s|^{2(p_0-1)}
\right)^{\frac{1}{2}}  
&\leq C_H~|H-\tfrac{1}{2}|^{p_0-1}~(t-s)^{(p_0-1)(H\wedge\frac{1}{2})} 
~(1+|\log(t-s)|)^{p_0-1} \\
&~~~~~~~(1+\lambda)^{2(p_0-1)}
~(1+t^2)^{p_0-1}~e^{-(p_0-1)\lambda s} ~\Pi_H(s)^{p_{0}-1}.
\end{split}
\end{equation*}
Coming back to~\eqref{sec6-eq:1stBoundI2-bis} and choosing
$q=H\wedge \frac{1}{2}$ and $p_0$ large enough (such that $p_{0}q>2$), we get
\begin{align*}\label{I2-majoration}
|I_2(\lambda)| 
&\leq C_H~|H-\tfrac{1}{2}|
~(1+\lambda)^2
~\left((\ratet_{2}(\mathbb{Y}-y_0,{\lambda}))
^{\frac{1}{2p_{0}}}
+ (\ratet_{4}(\mathbb{Y}-y_0,{\lambda}))
^{\frac{1}{4p_{0}}}\right) \nonumber \\
&\quad\quad \times 
\lim_{N\rightarrow \infty}~\sum_{n=0}^{N-1} 
~(1+(n+1))^2
~\Pi_H(n)~\left(\int_{n}^{n+1}\int_{n}^{n+1}
\big(1+ |\log(t-s)|\big)^{p_{0}}~ds~ 
dt\right)^{\frac{1}{p_{0}}} \nonumber\\
&~~~~~~~~~~~~~~~~~~~~~~~~~
\times \big(\PP(\tau^H_Y\geq n)\big)^{\frac{p_0-1}{p_0}}
\\
 &\leq  C_H~|H-\tfrac{1}{2}|
~(1+\lambda)^2
~\left( 
(\ratet_{2}(\mathbb{Y}-y_0,{\lambda}))^{\frac{1}{2p_{0}}}
+ (\ratet_{4}(\mathbb{Y}-y_0,{\lambda}))
^{\frac{1}{4p_{0}}}\right), \nonumber
\end{align*}
where we used \eqref{cond:Pi_H} and a series-integral comparison to obtain the last line.
\end{proof}

\subsection{Sufficient conditions on $\widetilde{b}$ for weak convergence rates} 
\label{subsec:cond-drift}

We now discuss conditions on the drift $\widetilde{b}$ which imply the 
conditions~\eqref {cond:tail-tau_H_Y} and~\eqref{cond:Pi_H}  in 
Theorem~\ref{prop:majgap_0_bis} and therefore permit to apply the proposition~\ref{prop:estim-aletti}.

The following result is obvious.
\begin{proposition} \label{prop:b-decroissante}
Let the assumptions of Theorem~\ref{prop:majgap_0} hold true. 
 
 In addition, suppose that 
 \begin{equation}\label{hyp:btildeprimenegative}
\widetilde{b}'(x) \leq 0 ,\quad \forall x\in \R,
 \end{equation}
and that
\begin{equation} \label{cond::tail-tau_H-H-qcq}
\exists \vartheta>3,~~\forall H\in (\tfrac{1}{4},1),~~ \PP(\tau^H_Y \geq t) \leq \frac{C_H}{(1+t)^{2\vartheta}},\quad \forall t\geq 0.
\end{equation}
Then, the laws of $\tau_{X}^H$ and $\tau_{\mathbf{X}}$ are non defective for any $H\in (\tfrac{1}{4},1)$, $\tau_{X}^H$ weakly converges to $\tau_{\mathbf{X}}$ as $H\to \frac{1}{2}$ and
\begin{equation}\label{eq:boundAletti}
d_{A}(\tau_{X}^H, \tau_{\mathbf{X}})  \leq 
C_H~|H-\tfrac{1}{2}|.
\end{equation}
\end{proposition}

\begin{proof}
The hypothesis \eqref{hyp:btildeprimenegative} on $\widetilde{b}'$ implies that 
one can choose $\Pi_H(t)\equiv 1$. One can thus obtain~\eqref{cond:Pi_H} by 
using~\eqref{cond::tail-tau_H-H-qcq}.
Observe also that~\eqref{cond::tail-tau_H-H-qcq} is stronger than~\eqref {cond:tail-tau_H_Y} and obviously implies the non defectiveness of the laws of $\tau_{X}^H$ and $\tau_{\mathbf{X}}$.
Therefore, the conclusion
follows from~Theorem~\ref{prop:majgap_0_bis} and
Proposition~\ref{prop:estim-aletti}.
\end{proof}

The stringent monotonicity condition~\eqref{hyp:btildeprimenegative} allowed us to get \eqref{eq:boundAletti} for any $H\in(\tfrac{1}{4},1)$. 
When $H$ is restricted to the interval $(\tfrac{1}{4},\tfrac{1}{2})$, we are going to replace this condition by a more satisfying one.

Observe that the fractional Brownian motion does not allow one to apply 
Theorem~\ref{prop:majgap_0_bis} since $\PP(\tau_{B^H}\geq 1+t) \geq 
\frac{\log(1+t)^{-\gamma}}{(1+t)^{1-H}}$
for some $\gamma>0$ (see~\cite{Aurzada}). The strict ellipticity of the drift 
$\widetilde{b}$ is a natural condition to obtain~\eqref{cond:tail-tau_H_Y}. The following lemma provides a tail estimate which implies~\eqref{cond:tail-tau_H_Y} and will be used to prove the theorem~\ref{theo:application-aletti-b-elliptique} below.

\begin{lemma} \label{lemma:queue-distribution-tau}
Suppose that the drift~$\widetilde{b}$ satisfies
\begin{equation} \label{hypo-b-cas-non-defectif}
\exists \nu>0,~\mu>0,~\forall x\in\mathbb{\R},
~~0<\nu \leq \widetilde{b}(x) \leq \mu.
\end{equation}
There exists $C>0$ such that for any 
$\frac{1}{4}<H<1$, $\nu>0$, $\mu>0$ and 
$t>\frac{m}{\mu}$ one has
\begin{equation} \label{ineq:tail-estimate}
\PP(\tau_Y^H \geq t) \leq \frac{C}{1\vee t^{1-H}}~\exp\Big(-\tfrac{1}{2}
~\frac{(\nu~t-m)^2}{t^{2H}}\Big).
\end{equation}
Therefore, $\PP(\tau_{Y}^H = \infty) =0$.
\end{lemma}

\begin{proof}
As previously, let $m:=\mathbb{Y}-y_0$. One has
\begin{equation*}
\begin{split}
\PP(\tau_Y^H \geq t) &\leq \PP(\sup_{s\leq t} (B^H_s 
+ \nu~s) < m) \\
&= \PP(\sup_{u\leq 1} (B^H_{tu} + \nu~tu) < m) \\
&= \PP\Big(\sup_{u\leq 1} 
(t^{H-1}~B^H_u + \nu~u) < \tfrac{m}{t}\Big) \\
&\leq \PP\Big(t^{H-1}~B^H_1 + \nu < \tfrac{m}{t}\Big).
\end{split}
\end{equation*}
The desired result follows from standard inequalities on Gaussian distributions.
\end{proof}

In the irregular case~$\tfrac{1}{4}<H<\tfrac{1}{2}$ the condition~\eqref{hypo-b-cas-non-defectif} is sufficient 
to quantify the weak convergence of $\tau_{X}^H$ when $H$ increases
to~$\frac{1}{2}$, as shown by the following theorem.

\begin{theorem} \label{theo:application-aletti-b-elliptique}
Let $\tfrac{1}{4}<H<\tfrac{1}{2}$.

In addition to the assumptions of Theorem~\ref{prop:majgap_0} assume that the function $\widetilde{b}$ satisfies~\eqref{hypo-b-cas-non-defectif}. Then, $\tau_{X}^H$ weakly converges to $\tau_{\mathbf{X}}$ as $H\nearrow \frac{1}{2}$ and
\begin{equation*}
d_{A}(\tau_{X}^H, \tau_{\mathbf{X}})  \leq 
C_H~|H-\tfrac{1}{2}|.
\end{equation*}
\end{theorem}

\begin{proof}
Observe that $\Pi_H(t) \leq C \exp(|\widetilde{b}'|_\infty t)$. Since $H<\tfrac{1}{2}$, Lemma~\ref{lemma:queue-distribution-tau} implies \eqref{cond:Pi_H}. 
The result then follows from Theorem~\ref{prop:majgap_0_bis} and Proposition~\ref{prop:estim-aletti}. 
\end{proof}

\begin{remark}
Unfortunately, when $\tfrac{1}{2}<H<1$ the ellipticity condition~\eqref{hypo-b-cas-non-defectif} 
is far from being sufficient to imply~\eqref{cond:Pi_H}. On the one hand, 
without additional assumptions on~$\widetilde{b}$, the only 
bound from above for
$\Pi_H(t)$ is~$C \exp(|\widetilde{b}'|_\infty t)$. On the other hand, for these values of~$H$, one cannot expect that the law of $\tau_{Y}^H$ has exponential moments when~$H>\frac{1}{2}$ under the sole condition~\eqref{hypo-b-cas-non-defectif}.
Actually, Prakasa Rao has obtained asymptotic tail estimates for the running maxima of drifted fractional Brownian motions~\cite{Prakasa-Rao}.
These estimates suggest that  the rate of the exponential decay of the tails 
of~$\tau_{Y}^H$ is close to the bound from above provided in Lemma~\ref{lemma:queue-distribution-tau}. 

We have not succeeded to relax the monotonicity hypothesis in~Proposition~\ref{prop:b-decroissante} in order to obtain a satisfying generic result when $\tfrac{1}{2}<H<1$. In this situation, we are only able to suggest to check the hypotheses
of Theorem~\ref{prop:majgap_0_bis} on a case-by-case basis.
\end{remark}

The following table summarises the results obtained in this section for $H\in (\tfrac{1}{4},1)$.

{\renewcommand{\arraystretch}{1.2}
\begin{center}
\begin{tabular}{|l||l|l|}
\hline
    & $\widetilde{b}'\leq 0$ & No condition on $\widetilde{b}'$ \\
\hline\hline
				  & - The law of $\tau_{X}^H$  is non defective  & - The law of $\tau_{X}^H$  is non defective  \\
			    &  - \eqref{ineq:Laplace-transform-main-theorem-hypo-renforcees} holds true 		    
			   & - If $H<\frac{1}{2}$, 
 \eqref{ineq:Laplace-transform-main-theorem-hypo-renforcees} holds true \\
$\inf \widetilde{b}>0$ 	    
			    & - Weak convergence & - Weak convergence for $H\nearrow \frac{1}{2}$  \\
			    & - Proposition~\ref{prop:b-decroissante}   &  - If $H<\frac{1}{2}$, Theorem~\ref{theo:application-aletti-b-elliptique}\\
			    &	& - If $H>\frac{1}{2}$, no improvement of Theorem~\ref{prop:majgap_0} \\
\hline			    
			 & - The law of $\tau_{X}^H$ may be defective   & - The law of $\tau_{X}^H$ may be defective  \\
  No condition & - \eqref{ineq:Laplace-transform-main-theorem} holds true with
  $\widetilde{\lambda}=\lambda$ 
  & - No improvement of Theorem~\ref{prop:majgap_0}\\
  		 on $\widetilde{b}$	& ~~~~and any $\lambda\geq0$& \\
  & - Possibly no weak convergence		& \\
 \hline
\end{tabular}
\end{center}
}

\section{Conclusion and perspectives}\label{sec:conclusion}

In this paper we have developed a sensitivity analysis w.r.t. the Hurst 
parameter of the driving noise for the probability distribution of 
functionals of solutions to stochastic differential equations, 
including the probability distribution of first hitting times,
when the Hurst parameter is close to $1/2$, that is, when the noise is 
close to the pure Brownian case.
Our estimates seem accurate.
As explained in the introduction, in practice they tend to justify the 
use of Markov Brownian models when estimated Hurst parameters remain 
close to~$\tfrac{1}{2}$.
In principle, by using similar analytical tools as above, it should be 
possible to get expansions in terms of $|H-\tfrac{1}{2}|$. However, the 
calculations would be still much more lengthy and heavy than above.
The following open questions deserve future works. 

\vspace{0.1cm}
It would be interesting to extend our results to SDEs driven by a 
Gaussian noise with general kernel $K$ and to estimate the sensitivity 
of first hitting time Laplace transforms in 
terms of the $L^2$ distance between $K$ and $K_{\frac{1}{2}}$.

\vspace{0.1cm}
The ellipticity condition \ref{hyp:h2'} may seem restrictive but it 
seems difficult to get rid of it. A natural attempt is as follows. When $H=\tfrac{1}{2}$ the 
SDE~(\ref{eq:fSDE_0}) can 
be written
in the following It\^o's form:
\begin{equation*}
X_t = x_0 + \int_0^t \left(b(X_s) + 
\tfrac{1}{2}\sigma(X_s)\sigma'(X_s)\right)\ ds + \int_0^t 
\sigma(X_s)\ dW_s.
\end{equation*}
Let $u$ be the solution to the following parabolic PDE 
\begin{equation*}
\begin{cases}
\frac{\partial }{\partial s}u(s,x) + (b(x)+\tfrac{1}{2}\sigma(x)\sigma'(x)) \frac{\partial}{\partial x} u(s,x) + \tfrac{1}{2}\sigma^2(x) \frac{\partial^2}{\partial x^2} u(s,x) = 0,
~~(s,x)\in [0,t)\times \R, \\
u(t,x) = \varphi(x),~~x\in \R.
\end{cases}
\end{equation*} 
Recall the calculation made in Section~\ref{subsec:proof_gapLawXt^H}.
For $\tfrac{1}{2}<H<1$ Itô's formula applied to $u(t,X^H_t)$ leads to
\begin{align*}
&\EE\left(u(t,X_t^H)\right) - u(0, x_0) = \\
& -\EE\int_0^t \tfrac{1}{2} (\sigma \sigma')(X_s^H)\ \partial_x 
u(s,X_s^H) \ ds + \alpha_H \EE\int_0^t \int_0^T |r-s|^{2H-2} 
\sigma(X_r^H) \sigma'(X_s^H) \partial_x u(s,X_s^H)~dr~ds\\
&-\EE\int_0^t \tfrac{1}{2} \sigma^2(X_s^H)\ \partial^2_{xx} u(s,X_s^H) 
~ds + \alpha_H \EE\int_0^t \int_0^T |r-s|^{2H-2} \sigma(X_r^H) 
\sigma(X_s^H) \partial^2_{xx} u(s,X_s^H)~dr~ds	\\
&+\alpha_H \EE\int_0^t \int_0^T |r-s|^{2H-2} (D_r X_s^H-\sigma(X_r^H)) 
\left(\partial^2_{xx} u(s,X_s^H) \sigma(X_s^H) + \partial_x u(s,X_s^H) 
\sigma'(X_s^H)\right)~dr~ds.
\end{align*}
First, note that the hypothesis \ref{hyp:h2'} helps to get sharp 
estimates on derivatives of $u$. Second, without this hypothesis we 
have not succeeded to obtain accurate enough bounds on the sup and 
$\mathcal{H}_H$ norms of $D_r X_s^H$ to deduce relevant sensitivity 
estimates w.r.t~$H$. When $H>\tfrac{1}{2}$, the estimates obtained 
in~\cite{HuNualart} on the supremum and Hölder norm of $X^H$ and $DX^H$ 
do not require the ellipticity of $\sigma$. However, they depend on the 
H\"older norm $|B^H|_{\alpha,0,T}$, 
where $\alpha\in(\tfrac{1}{2},H)$, which tends to infinity 
when 
$H\rightarrow \tfrac{1}{2}$.

\vspace{0.1cm}
A sensitivity analysis of the density of~$\tau^H_X$ would certainly be 
useful for applications. 
Our estimate on the Laplace transform of $\tau^H_X$ gives information 
on the robustness of this density 
around time~0 when $H$ is close to~$\tfrac{1}{2}$. This seems 
interesting since the simulations 
in~\cite{DelormeWiese} suggest that, when $H>\tfrac{1}{2}$, the 
density of $\sup_{t\in[0,1]} B^H_t$ is unbounded around~$0$. To go 
further, one should compute the inverse Laplace transform of
the formula for
$\EE\left(e^{-\lambda \tau^H_X}\right) - \EE\left(e^{-\lambda 
\tau_X^{1/2}}\right)$ given in Proposition~\ref{prop:gaps}.
Handling technical issues raised by the singularity of the inverse 
Laplace transform and by terms whose Malliavin derivatives are  
singular is out of the scope of the present paper.

\vspace{0.1cm}
The extension of our analysis to multidimensional SDEs and first exit 
times of domains is another interesting further direction of research.

\vspace{0.1cm}
Finally, as explained at the end of the Introduction, 
sharp sensitivity analyses around $H\neq\frac{1}{2}$ seems to be a 
challenging problem.

~

\paragraph{Acknowledgments.}
The authors thank an anonymous referee for her/his comments on the first version of the paper which have led us to add the section 6.

\smallskip


\appendixpage
\begin{appendices}


\section{Representation of $K_{H}^*$ on $|\HH|$} \label{app:KH-etoile}
Denote by $\mathcal{E}$ the set of simple functions on $[0,T]$. We 
recall that $|\HH| \subset \HH$ is the completion of $\mathcal{E}$ 
with respect to the norm~\eqref{eq:defAbsH} (respectively, 
\eqref{eq:defAbsH<}) when $H>\tfrac{1}{2}$ (respectively,
$H<\tfrac{1}{2}$).

Consider the operator $\widetilde{K}_{H}^*: 
\mathcal{E}\to L^2[0,T]$ defined by
\begin{align*}
\widetilde{K}_{H}^*\varphi(t) = K_{H}(T,t)\varphi(t) + \chi_{H} 
(H-\tfrac{1}{2}) \int_{t}^T 
\left(\frac{\theta}{t}\right)^{H-\frac{1}{2}} 
\left(\theta-t\right)^{H-\frac{3}{2}} \left(\varphi(\theta) - 
\varphi(t)\right) ~d\theta .
\end{align*}
In view of~\eqref{eq:KHbis} this operator coincides with 
$K_{H}^*$ on $\mathcal{E}$. As $|\HH|\subset \HH$ is a continuous 
embedding we have
\begin{align*}
\forall \varphi\in \mathcal{E},\quad 
\|\widetilde{K}^*_{H}\varphi\|_{L^2[0,T]} = \|\varphi\|_{\HH} \leq C 
\|\varphi\|_{|\HH|} ,
\end{align*}
which implies that $\widetilde{K}^*_{H}$ can be continuously extended
to an operator from~$|\HH|$ to~$L^2[0,T]$. 

Let us now prove that 
$\widetilde{K}^*_{H}$ and $K^*_{H}$ coincide on $|\HH|$. Let 
$\varphi\in |\HH|$ and let $(\varphi_{n})\in \mathcal{E}^\N$ be a 
sequence which converges to $\varphi$ in $|\HH|$ (and thus also 
converges in~$\HH$). We have
\begin{equation*}
\begin{split}
\|\widetilde{K}^*_{H}\varphi - K_{H}^*\varphi\|_{L^2[0,T]} &\leq 
\|\widetilde{K}^*_{H} (\varphi-\varphi_{n})\|_{L^2[0,T]}  + 
\|K_{H}^*(\varphi_{n}-\varphi)\|_{L^2[0,T]} \\
&\leq C \|\varphi-\varphi_{n}\|_{|\HH|} + \|\varphi-\varphi_{n}\|_{\HH}.
\end{split}
\end{equation*}
Since the right-hand side of the preceding inequality converges to~$0$
when~$n$ tends to infinity we conclude that $\widetilde{K}^*_{H}\varphi 
= 
K_{H}^*\varphi$. 
Therefore the representation~\eqref{eq:KHbis} of $K_{H}^*$ 
holds true for any $\varphi\in |\HH|$.

\section{On various deterministic integrals depending on $K_H$ and 
$\partial_\theta K_H$}\label{App:tecLem}

Given $\tfrac{1}{4}<H<1$ and $0<r<t$, set
\begin{equation} \label{def:A(r,v)}
\begin{cases}
\mathcal{A}(v,r,t) := \int_r^t 
|\partial_\theta K_H(\theta,v)|~K_H(\theta,r)~d\theta
~~~\text{for}~0<s<v<t, \\
\mathcal{A}^\sharp(v,r,t) := \int_v^t 
|\partial_\theta K_H(\theta,v)|
~K_H(\theta,r)~(\theta-v)^H~d\theta
~~~\text{for}~0<s<v<t, \\
\mathcal{A}^\flat(v,r,t) :=  \int_v^t 
|\partial_\theta K_H(\theta,v)|
~(K_H(\theta,r)-K_H(v,r))~d\theta
~~~\text{for}~0<s<v<t, \\
\mathcal{A}^\natural(v,r,t) :=  \int_s^t 
|\partial_\theta K_H(\theta,v)|~K_H(\theta,r)~d\theta
~~~\text{for}~0<v<s<t.
\end{cases}
\end{equation}

In many calculations we need to consider time intervals $0<s<t$ with 
$0<t-s<1$ and the integrals
$$ \int_s^t \int_v^t (\mathcal{A}(v,r,t))^2~dr~dv,~~~
\int_s^t \int_0^v (\mathcal{A}^\sharp(v,r,t))^2~dr~dv,~~\text{etc.} $$
We need to bound these integrals from above by a constant of the type 
$C_H~|H-\tfrac{1}{2}|^2~(t-s)^\delta~t^k$ 
for some $\delta>0$ and $k\geq0$. 
Getting such an estimate requires totally different arguments 
according as~$\tfrac{1}{2}<H<1$ or $\tfrac{1}{4}<H<\tfrac{1}{2}$ 
because of the difference of behaviour of the kernel $K_H$ in these two 
cases:

\begin{itemize}
\item \textbf{In the singular case $H<\tfrac{1}{2}$}
we have $(\frac{\theta}{r})^{H-\frac{1}{2}}<1$ for any $\theta>r$. 
The formulae~\eqref{def:KH} 
and \eqref{eq:deriv-K_H} respectively lead to
\begin{equation} \label{eq:K_H-simplifiee-H-petit}
0 \leq K_H(\theta,r) \leq 
C~\indi{\{\theta>r\}}~\left((\theta-r)^{H-\frac{1}{2}} + 
|H-\tfrac{1}{2}|~r^{\frac{1}{2}-H} \int_r^\theta 
\xi^{H-\frac{3}{2}}~(\xi-r)^{H-\frac{1}{2}}~d\xi \right)
\end{equation}
and
\begin{equation} \label{eq:deriv-K_H-simplifiee-H-petit}
-\partial_\theta K_H (\theta,v)
= \left|\partial_\theta K_H (\theta,v)\right| 
\leq \indi{\{\theta>v\}}~C~|H-\tfrac{1}{2}|~(\theta-v)^{H-\frac{3}{2}}.
\end{equation}

\item \textbf{In the regular case $H>\tfrac{1}{2}$} we have
\begin{equation} \label{eq:K_H-simplifiee-H-grand}
0 \leq K_H(\theta,r) 
\leq C~\indi{\{\theta>r\}}~\left(\frac{\theta}{r}\right)^{H-\frac{1}{2}}
~(\theta-r)^{H-\frac{1}{2}}
\end{equation}
and
\begin{equation} \label{eq:deriv-K_H-simplifiee-H-grand}
0 \leq \partial_\theta K_H (\theta,v) 
\leq C~\indi{\{\theta>v\}}~|H-\tfrac{1}{2}|
~\left(\frac{\theta}{v}\right)^{H-\frac{1}{2}}
~(\theta-v)^{H-\frac{3}{2}}.
\end{equation}                 

\end{itemize}

\subsection{Estimate on~$\int_s^t 
\int_v^t (\mathcal{A}(v,r,t))^2~dr~dv$: The regular case 
$\tfrac{1}{2}<H<1$}

\begin{proposition}
Let $\mathcal{A}(v,r,t)$ be defined as in~\eqref{def:A(r,v)}.
For any $\tfrac{1}{2} < H < 1$ 
 it holds that
\begin{equation} \label{ineq:L2-maj-A(r,v)-H-grand}
\int_s^t 
\int_v^t (\mathcal{A}(v,r,t))^2~dr~dv 
\leq C_H~|H-\tfrac{1}{2}|^2
~(t-s)^{3-2H}~t^{6H-3}.
\end{equation}
\end{proposition}

\begin{proof}
Since $H>\tfrac{1}{2}$, in view of~\eqref{eq:K_H-simplifiee-H-grand} 
and~\eqref{eq:deriv-K_H-simplifiee-H-grand} we have
$$ |\mathcal{A}(v,r,t)| \leq 
C~|H-\tfrac{1}{2}| \left(\frac{t^2}{v~r}\right)^{H-\frac{1}{2}}
\int_r^t 
(\theta-v)^{H-\frac{3}{2}} (\theta-r)^{H-\frac{1}{2}} d\theta
\leq 
C~|H-\tfrac{1}{2}|~\frac{t^{3H-\frac{3}{2}}}{(v~r)^{H-\frac{1}{2}}}
~\int_r^t (\theta-v)^{H-\frac{3}{2}}~d\theta. $$
We thus are led to consider
$$ \int_v^t \frac{1}{r^{2H-1}}~
\left(\int_r^t (\theta-v)^{H-\frac{3}{2}}~d\theta \right)^2 dr
\leq \int_v^t \frac{1}{(r-v)^{2H-1}}
\left(\int_r^t (\theta-v)^{H-\frac{3}{2}}~d\theta \right)^2 dr. $$
We need some care to get a bound on the right-hand side of the 
preceding inequality which does not explode when $H$ tends to 
$\frac{1}{2}$. By two successive integrations by parts we get that the 
right-hand side equals to
\begin{equation*}
\begin{split}
\tfrac{1}{1-H} \int_v^t 
(r-v)^{2-2H} \left(\int_r^t(\theta-v)^{H-\frac{3}{2}}~d\theta\right)
(r-v)^{H-\frac{3}{2}} dr
&= C_H \int_v^t 
(r-v)^{\frac{1}{2}-H} \int_r^t(\theta-v)^{H-\frac{3}{2}} d\theta~dr \\
&= C_H~(t-v).
\end{split}
\end{equation*}
As $t-v\leq t-s$ we deduce that
\begin{equation*}
\begin{split}
\int_s^t 
\int_v^t (\mathcal{A}(v,r,t))^2~dr~dv 
&\leq C_H~|H-\tfrac{1}{2}|^2~(t-s)
~t^{6H-3}~\int_s^t 
\frac{1}{v^{2H-1}}~dv \\
&\leq C_H~|H-\tfrac{1}{2}|^2~(t-s)
~t^{6H-3}~(t^{2-2H}-s^{2-2H}).
\end{split}
\end{equation*}
As $\tfrac{1}{2}<H<1$ one has $t^{2-2H}-s^{2-2H} \leq 
(t-s)^{2-2H}$.
That ends the proof.
\end{proof}

\subsection{Estimate on~$\int_s^t 
\int_v^t (\mathcal{A}(v,r,t))^2~dr~dv$: The irregular case 
$\tfrac{1}{4}<H<\tfrac{1}{2}$}

In the irregular case the calculations are longer than in the 
regular case. We start with an easy lemma.

\begin{proposition}
For any~$T>0$ we have
\begin{equation} \label{ineq:integrale-puissance-et-log}
\forall \gamma>-1,~~\int_0^T \theta^\gamma~
(\log(\theta))^2~d\theta \leq \frac{C}{\gamma+1}~T^{\gamma+1}
~(\log(T))^2 + \frac{C}{(\gamma+1)^3}~T^{\gamma+1}.
\end{equation}
For any $0<s<t$ we have
\begin{equation} \label{ineq:L2-maj-I(r,v)}
\forall \gamma>-1,~~\int_s^t\int_v^t 
(r-v)^\gamma~\left(\log(\frac{t-v}{r-v})\right)^2~dr~dv
\leq \frac{C}{(\gamma+1)^2}~(t-s)^{\gamma+2}~(\log(t-s))^2
+ \frac{C}{(\gamma+1)^4}~(t-s)^{\gamma+2}.
\end{equation}
\end{proposition}

\begin{proof}
Two successive integrations by parts lead to 
\begin{equation*}
\begin{split}
\int_0^T \theta^\gamma~(\log(\theta))^2~d\theta
&\leq \frac{1}{\gamma+1}~T^{\gamma+1}~(\log(T))^2
+\frac{2}{(\gamma+1)^2}~T^{\gamma+1}~|\log(T)|
+\frac{2}{(\gamma+1)^3}~T^{\gamma+1} \\
&\leq \frac{1}{\gamma+1}~T^{\gamma+1}~( |\log(T)| + 
\frac{1}{\gamma+1})^2.
\end{split}
\end{equation*}
The inequality~\eqref{ineq:integrale-puissance-et-log} follows.

To prove~\eqref{ineq:L2-maj-I(r,v)} we start with deducing
from~\eqref{ineq:integrale-puissance-et-log} that
$$ \int_v^t (r-v)^\gamma~(\log(r-v))^2~dr
= \int_0^{t-v} r^\gamma~(\log(r))^2~dr 
\leq \frac{C}{(\gamma+1)^3}~(t-v)^{\gamma+1}
~\left((\gamma+1)^2~(\log(t-v))^2+1 \right). $$
In addition,
$$ \int_s^t\int_v^t (r-v)^\gamma~(\log(t-v))^2~dr~dv 
= \frac{1}{\gamma+1} \int_s^t (t-v)^{\gamma+1}~(\log(t-v))^2~dv. $$
Therefore, as $\gamma+2>1$ the left-hand side 
of~\eqref{ineq:L2-maj-I(r,v)} is bounded from above by
$$ \frac{C}{\gamma+1}~\int_0^{t-s} \theta^{\gamma+1}
~(\log(\theta))^2~d\theta + \frac{C}{(\gamma+1)^4}(t-s)^{\gamma+2}. $$
It remains to again use~\eqref{ineq:integrale-puissance-et-log} (with 
$T=t-s$) to get~\eqref{ineq:L2-maj-I(r,v)}.
\end{proof}

We will need to consider the integral $I(v,r,t)$ defined 
for~$0\leq v<r<t$ by
\begin{equation} \label{def:I(r,v)}
I(v,r,t) := \int_r^t 
(\xi-v)^{H-\frac{3}{2}}~(\xi-r)^{H-\frac{1}{2}}~d\xi.
\end{equation} 
It will be decisive to bound~$I(v,r,t)$ from above by a function of 
$r$~and~$v$ which is 
square integrable and  involves constants which do not explode when 
$H$~tends to $\tfrac{1}{2}$. 

\begin{proposition} 
For any $\tfrac{1}{4}<H<\tfrac{1}{2}$ and $0\leq v<r<t$, let 
$I(v,r,t)$ be defined as in~\eqref{def:I(r,v)}. One has
\begin{equation} \label{ineq:maj-I(r,v)}
I(v,r,t) \leq C~(r-v)^{2H-1}~\left(\log(\frac{t-v}{r-v}) + 1 \right).
\end{equation} 
\end{proposition}

\begin{proof}
Use the change of variable $\xi = r+\frac{r-v}{\alpha}$. It comes:
$$ I(v,r,t) = (r-v)^{2H-1}~\int_{\frac{r-v}{t-r}}^\infty 
(1+\alpha)^{H-\frac{3}{2}}~\alpha^{-2H}~d\alpha. $$
Now, split the integration interval into $(\frac{r-v}{t-r}\wedge 1,1)$
and $(1,+\infty)$. As for the integral over $(1,+\infty)$ 
we observe that
$$ \int_1^\infty (1+\alpha)^{H-\frac{3}{2}}~\alpha^{-2H}~d\alpha
\leq \int_1^\infty\alpha^{-H-\frac{3}{2}}~d\alpha \leq C. $$
As for the integral over $(\frac{r-v}{t-r}\wedge 1,1)$,
for any~$0<z<1$ and $\tfrac{1}{4}<H<\tfrac{1}{2}$ one has
$$ \int_z^1 (1+\alpha)^{H-\frac{3}{2}}~\alpha^{-2H}~d\alpha
\leq \int_z^1 \alpha^{-2H}~d\alpha 
\leq \int_z^1 \alpha^{-1}~d\alpha
= \log(\frac{1}{z}). $$
For $z=\frac{r-v}{t-r}\wedge 1$ one has
$\log(\frac{1}{z}) \leq \log(\frac{t-v}{r-v})$.
The desired result follows.
\end{proof}

We now are in a position to get the main result in this subsection.

\begin{proposition} \label{lemma:dKH-v-KH-r}
Let $\mathcal{A}(v,r,t)$ be defined as in~\eqref{def:A(r,v)}.
For any $\tfrac{1}{4} < H < \tfrac{1}{2}$ one has
\begin{equation} \label{ineq:L2-maj-A(r,v)}
\int_s^t \int_v^t (\mathcal{A}(v,r,t))^2~dr~dv 
\leq C_H~|H-\tfrac{1}{2}|^2~(t-s)^{4H}~( (\log(t-s))^2+1). 
\end{equation}
\end{proposition}

\begin{proof}
In view of~\eqref{eq:K_H-simplifiee-H-petit},
\eqref{eq:deriv-K_H-simplifiee-H-petit} and~\eqref{def:I(r,v)} one has
$$ |\mathcal{A}(v,r,t)| \leq C~|H-\tfrac{1}{2}|~I(v,r,t)
+C~|H-\tfrac{1}{2}|^2~\int_r^t (\theta-v)^{H-\frac{3}{2}}
~r^{\frac{1}{2}-H}~I(0,r,\theta)~d\theta. $$
First, we use~\eqref{ineq:maj-I(r,v)} to bound~$I(v,r,t)$ from above. 
Second, we notice that~\eqref{ineq:maj-I(r,v)} implies 
$$ I(0,r,\theta) 
\leq C~r^{2H-1}~\left(\log(\frac{t}{r}) + 1 \right)
\leq C~r^{2H-1}~\left(\log(\frac{t-v}{r-v}) + 1 \right), $$
from which
\begin{equation*}
\begin{split}
|H-\tfrac{1}{2}|^2~\int_r^t (\theta-v)^{H-\frac{3}{2}}
~r^{\frac{1}{2}-H}~I(0,r,\theta)~d\theta 
&\leq 
C~|H-\tfrac{1}{2}|~r^{H-\frac{1}{2}}~\left(\log(\frac{t-v}{r-v})+1\right)
~(r-v)^{H-\frac{1}{2}} \\
&\leq 
C~|H-\tfrac{1}{2}|~(r-v)^{2H-1}~\left(\log(\frac{t-v}{r-v})+1\right).
\end{split}
\end{equation*}
We thus have
\begin{equation*} 
\mathcal{A}(v,r,t) \leq C~|H-\tfrac{1}{2}|~(r-v)^{2H-1}
~\left( \log(\frac{t-v}{r-v}) + 1 \right).
\end{equation*}
The inequality~\eqref{ineq:L2-maj-A(r,v)} then results from
the inequality~\eqref{ineq:L2-maj-I(r,v)} with $\gamma=4H-2$.
\end{proof}

\subsection{Estimate on~$\int_s^t 
\int_0^v (\mathcal{A}^\sharp(v,r,t))^2~dr~dv$}

\begin{proposition}
For $0<r<v<t$ let $\mathcal{A}^\sharp(v,r,t)$ be defined as 
in~\eqref{def:A(r,v)}.
\begin{enumerate}[label=(\roman*)]
\item In the regular case $\frac{1}{2}<H<1$ it holds that
\begin{equation} \label{ineq:L2-maj-Adiese(r,v)-H-grand}
\int_s^t \int_0^v (\mathcal{A}^\sharp(v,r,t))^2~dr~dv 
\leq C_H~|H-\tfrac{1}{2}|^2~(t-s)^{2H+1}~t^{4H-1}.
\end{equation}
\item In the irregular case $\frac{1}{4}<H<\frac{1}{2}$ it holds that
\begin{equation} \label{ineq:L2-maj-Adiese(r,v)-H-petit}
\int_s^t \int_0^v (\mathcal{A}^\sharp(v,r,t))^2~dr~dv 
\leq C_H~|H-\tfrac{1}{2}|^2~(t-s)^{4H}~t^{2H}.
\end{equation}
\end{enumerate}
\end{proposition}

\begin{proof}
Notice that
\begin{equation*}
\begin{split}
(\mathcal{A}^\sharp(v,r,t))^2 &= 2 \int_v^t 
\mathcal{A}^\sharp(v,r,\theta)
\Big| \partial_\theta K_H(\theta,v) \Big|
~K_H(\theta,r)~(\theta-v)^H~d\theta \\
&= 2 \int_v^t \int_v^{\theta}
\Big| \partial_\alpha K_H(\alpha,v) \Big|
~K_H(\alpha,r)~(\alpha-v)^H~d\alpha
~\Big| \partial_\theta K_H(\theta,v) \Big|
~K_H(\theta,r)~(\theta-v)^H~d\theta. 
\end{split}
\end{equation*}

Now, in view of~\eqref{eq-RH-KH}, for any 
$v<\alpha<\theta<t$ we have
\begin{equation} \label{ineq:int_KH-KH}
\int_0^v K_H(\alpha,r)~K_H(\theta,r)~dr
\leq \int_0^t K_H(\alpha,r)~K_H(\theta,r)~dr
= \tfrac{1}{2}~(\alpha^{2H} + \theta^{2H} - (\theta-\alpha)^{2H})
\leq t^{2H}.
\end{equation}
We deduce:
\begin{align} \label{ineq:int-0-v-A-diese-2-dr}
\int_0^v(\mathcal{A}^\sharp(v,r,t))^2~dr
&\leq 2~t^{2H}~\int_v^t \int_v^{\theta}
|\partial_\alpha K_H(\alpha,v)|
~(\alpha-v)^H~d\alpha~|\partial_\theta K_H(\theta,v)|
~(\theta-v)^H~d\theta \nonumber \\
&= t^{2H}~\Big(\int_v^t |\partial_\theta K_H(\theta,v)| 
~(\theta-v)^H~d\theta \Big)^2.
\end{align}

\textbf{In the regular case $\frac{1}{2}<H<1$} we therefore
can use~\eqref{eq:deriv-K_H-simplifiee-H-grand} to get
$$ \int_0^v(\mathcal{A}^\sharp(v,r,t))^2~dr 
\leq C_H~|H-\tfrac{1}{2}|^2~(t-v)^{4H-1}~\frac{t^{4H-1}}{v^{2H-1}}
\leq C_H~|H-\tfrac{1}{2}|^2~(t-s)^{4H-1}~\frac{t^{4H-1}}{v^{2H-1}}, $$
from which
$$ \int_s^t\int_0^v(\mathcal{A}^\sharp(v,r,t))^2~dr~dv
\leq C_H~|H-\tfrac{1}{2}|^2~(t-s)^{4H-1}~(t^{2-2H}-s^{2-2H})~t^{4H-1}. 
$$
It then remains to use $t^{2-2H}-s^{2-2H} \leq (t-s)^{2-2H}$ to 
obtain~\eqref{ineq:L2-maj-Adiese(r,v)-H-grand}.

\vspace{0.3cm}
\textbf{In the irregular case $\frac{1}{4}<H<\frac{1}{2}$} we 
can use~\eqref{eq:deriv-K_H-simplifiee-H-petit} to get
$$ \int_0^v(\mathcal{A}^\sharp(v,r,t))^2~dr 
\leq C_H~|H-\tfrac{1}{2}|^2~(t-v)^{4H-1}~t^{2H} 
\leq C_H~|H-\tfrac{1}{2}|^2~(t-s)^{4H-1}~t^{2H}, $$
from which~\eqref{ineq:L2-maj-Adiese(r,v)-H-petit} follows.
\end{proof}

\subsection{Estimate on~$\int_s^t 
\int_0^v (\mathcal{A}^\flat(v,r,t))^2~dr~dv$}

\begin{proposition}
For $0<r<v<t$ let $\mathcal{A}^\flat(v,r,t)$ be defined as 
in~\eqref{def:A(r,v)}.

\begin{enumerate}[label=(\roman*)]
\item In the regular case $\frac{1}{2}<H<1$ it holds that
\begin{equation} \label{ineq:L2-maj-mathcal-Abemol(r,v)-H-grand}
\int_s^t \int_0^v (\mathcal{A}^\flat(v,r,t))^2~dr~dv 
\leq C_H~|H-\tfrac{1}{2}|^2~(t-s)^{2H+1}~t^{2H-1}.
\end{equation}
\item In the irregular case $\frac{1}{4}<H<\frac{1}{2}$ it holds that
\begin{equation} \label{ineq:L2-maj-mathcal-Abemol(r,v)-H-petit}
\int_s^t \int_0^v (\mathcal{A}^\flat(v,r,t))^2~dr~dv 
\leq C~|H-\tfrac{1}{2}|^2~(t-s)^{4H}.
\end{equation}
\end{enumerate}
\end{proposition}

\begin{proof}
Notice that
\begin{equation*}
\begin{split}
(\mathcal{A}^\flat(v,r,t))^2 
&= 2 \int_v^t \mathcal{A}^\flat(v,r,\theta)
\Big| \partial_\theta K_H(\theta,v) \Big|
~|K_H(\theta,r) - K_H(v,r)|~d\theta \\
&\leq 2 \int_v^t \int_v^{\theta}
\Big| \partial_\alpha K_H(\alpha,v) \Big| 
~|K_H(\alpha,r) - K_H(v,r)|~d\alpha
~\Big|  \partial_\theta K_H(\theta,v) \Big| 
|K_H(\theta,r)-K_H(v,r)|~d\theta. 
\end{split}
\end{equation*}
Now, in view of~\eqref{eq:norme-L2-K_H}, for 
any $v<\alpha<\theta<t$ we have
\begin{equation*}
\begin{split}
\int_0^v |(K_H(\alpha,r) - K_H(v,r))~& (K_H(\theta,r)-K_H(v,r))|~dr \\
&\leq \Big(\int_0^t (K_H(\alpha,r) - K_H(v,r))^2~dr \Big)^{\tfrac{1}{2}}
~\Big(\int_0^t (K_H(\theta,r)-K_H(v,r))^2~dr\Big)^{\tfrac{1}{2}} \\
&\leq (\alpha - v)^H~(\theta - v)^H.
\end{split}
\end{equation*}
We deduce:
\begin{equation*}
\begin{split}
\int_0^v(\mathcal{A}^\flat(v,r,t))^2~dr
&\leq 2~\int_v^t \int_v^{\theta}
\Big|\partial_\alpha K_H(\alpha,v)\Big|
~(\alpha-v)^H~d\alpha~
\Big|\partial_\theta K_H(\theta,v)\Big|
~(\theta-v)^H~d\theta \\
&= \Big(\int_v^t
\Big| \partial_\theta K_H(\theta,v) \Big|
~(\theta-v)^H~d\theta \Big)^2.
\end{split}
\end{equation*}
In view of~\eqref{ineq:int-0-v-A-diese-2-dr} we deduce the desired 
inequalities by dividing the right-hand side 
of~\eqref{ineq:L2-maj-Adiese(r,v)-H-grand}
and~\eqref{ineq:L2-maj-Adiese(r,v)-H-petit} by $t^{2H}$.
\end{proof}

\subsection{Estimate on~$\int_0^s 
\int_0^t (\mathcal{A}^\natural(v,r,t))^2~dr~dv$}

\begin{proposition}
For $v$ and $r$ in $(0,t)$ let $\mathcal{A}^\natural(v,r,t)$ be defined 
as in~\eqref{def:A(r,v)}.

For any $\frac{1}{4}<H<1$ it holds that
\begin{equation} \label{ineq:L2-maj-mathcal-Abecarre(r,v)}
\int_0^s \int_0^t (\mathcal{A}^\natural(v,r,t))^2~dr~dv 
\leq C_H~|H-\tfrac{1}{2}|^2~(t-s)^{2H}~t^{2H}.
\end{equation}
\end{proposition}

\begin{proof}
Notice that
\begin{equation*}
\begin{split}
(\mathcal{A}^\natural(v,r,t))^2 
&= 2 \int_s^t \mathcal{A}^\natural(v,r,\theta)
\Big| \partial_\theta K_H(\theta,v) \Big|
~K_H(\theta,r)~d\theta \\
&= 2 \int_s^t \int_s^\theta
\Big| \partial_\alpha K_H(\alpha,v) \Big| 
~K_H(\alpha,r)~d\alpha
~\Big|  \partial_\theta K_H(\theta,v) \Big| 
K_H(\theta,r)~d\theta. 
\end{split}
\end{equation*}
We again use~\eqref{ineq:int_KH-KH} and~\eqref{eq:deriv-K_H}
 to get
\begin{equation*}
\begin{split}
\int_0^t(\mathcal{A}^\natural(v,r,t))^2~dr
&\leq 2~t^{2H}\int_s^t \int_s^\theta
\Big|\partial_\alpha K_H(\alpha,v)\Big|~d\alpha~
\Big|\partial_\theta K_H(\theta,v)\Big|~d\theta \\
&\leq 4~\chi_H^2~t^{2H}~(H-\tfrac{1}{2})^2 \int_s^t \int_s^\theta
(\theta\alpha)^{H-\frac{1}{2}}~v^{1-2H}~(\theta-v)^{H-\frac{3}{2}}
~(\alpha-v)^{H-\frac{3}{2}}~d\alpha~d\theta.
\end{split}
\end{equation*}
By changing the variable $v$ into $z=1-\frac{\theta}{\alpha}
\cdot \frac{\alpha-v}{\theta-v}$ one gets
$$ \int_0^s 
v^{1-2H}~(\theta-v)^{H-\frac{3}{2}}~(\alpha-v)^{H-\frac{3}{2}}~dv 
= (\alpha~\theta)^{\frac{1}{2}-H} (\theta-\alpha)^{2H-2} 
\int_0^{\frac{s}{\alpha} \cdot \frac{\theta-\alpha}{\theta-s}} 
z^{1-2H}~(1-z)^{H-\frac{3}{2}}~dz. $$
Therefore,
$$ \int_0^s\int_0^t(\mathcal{A}^\natural(v,r,t))^2~dr~dv
\leq C~(H-\tfrac{1}{2})^2~t^{2H} \int_s^t \int_s^\theta
(\theta-\alpha)^{2H-2} 
\int_0^{\frac{s}{\alpha} \cdot \frac{\theta-\alpha}{\theta-s}} 
z^{1-2H}~(1-z)^{H-\frac{3}{2}}~dz~d\alpha~d\theta. $$
We now combine the inequality 
$\frac{s}{\alpha} \cdot \frac{\theta-\alpha}{\theta-s}
\leq \frac{\theta-\alpha}{\theta-s}$ with the change of variables $x = 
\frac{\theta-\alpha}{\theta-s}$ to get
$$ \int_0^s\int_0^t(\mathcal{A}^\natural(v,r,t))^2~dr~dv
\leq C (H-\tfrac{1}{2})^2~t^{2H} \int_s^t (\theta-s)^{2H-1} \int_0^1 
x^{2H-2} \int_0^x z^{1-2H} (1-z)^{H-\frac{3}{2}}~dz~dx~d\theta. $$

To end the proof of Inequality~\eqref{ineq:L2-maj-mathcal-Abecarre(r,v)}
it remains to prove that
\begin{equation}\label{eq:proofAbec}
\int_0^1 x^{2H-2} \int_0^x z^{1-2H} 
(1-z)^{H-\frac{3}{2}}~dz~dx  \leq C_H.
\end{equation}
We have:
\begin{align*}
\int_0^1 x^{2H-2} \int_0^x z^{1-2H} (1-z)^{H-\frac{3}{2}}\, 
dz\, dx   &= \int_0^{\frac{1}{2}} x^{2H-2} \int_0^x
z^{1-2H} (1-z)^{H-\frac{3}{2}}~dz~dx  \\
&\quad + \int_{\frac{1}{2}}^1 x^{2H-2} \int_0^x z^{1-2H} 
(1-z)^{H-\frac{3}{2}}~dz~dx \\
&=: I_{1} + I_{2}.
\end{align*}
We have $I_1 \leq C_H$ since for any $x\leq \frac{1}{2}$, 
$$ \int_0^x z^{1-2H}~(1-z)^{H-\frac{3}{2}}~dz \leq 
C\int_0^x z^{1-2H}~dz = C_H~x^{2-2H}. $$
We now turn to $I_2$ which we split into the sum of
$$ I_{21} := \int_{\frac{1}{2}}^1 x^{2H-2} \int_0^{\frac{1}{2}} 
z^{1-2H}~(1-z)^{H-\frac{3}{2}}~dz~dx~~~~~\text{and}~~~~~
I_{22} := \int_{\frac{1}{2}}^1 x^{2H-2} \int_{\frac{1}{2}}^x 
z^{1-2H}~(1-z)^{H-\frac{3}{2}}~dz~dx. $$
On the one hand, the bound on $I_1$ leads to
$$ I_{21} \leq 2^{2-2H}~C_H. $$
On the other hand, we have
$$ I_{22} \leq 2^{2-2H} \int_{\frac{1}{2}}^1 
\int_{\frac{1}{2}}^x z^{1-2H}~(1-z)^{H-\frac{3}{2}}~dz~dx 
\leq 2^{2-2H}~(2^{2H-1}\vee 1)~
\int_{\frac{1}{2}}^1 (1-x)^{-\frac{1}{2}}
\int_{\frac{1}{2}}^x (1-z)^{H-1}~dz~dx \leq C. $$
We thus have obtained~\eqref{eq:proofAbec}.
\end{proof}

\subsection{On variants of $I(v,r,t)$: 
$\mathcal{I}(v,r,t)$, $\mathcal{I}^\sharp(v,r,t)$
and $\mathcal{I}^\natural(v,r,t)$}

\begin{proposition} \label{lemma:maj-I-reg(v,r,t)}
Let $\mathcal{I}(v,r,t)$ be  defined for~$0<v<r<t$ by
\begin{equation} \label{def:I-reg(v,r,t)}
\mathcal{I}(v,r,t) := \int_r^t 
(\theta-v)^{H-\frac{3}{2}}~(\theta-r)^{H+\frac{1}{2}}~d\theta.
\end{equation} 
For any $\tfrac{1}{4}<H<\tfrac{1}{2}$ one has
\begin{equation} \label{ineq:L2-maj-I-reg(v,r,t)}
\int_s^t \int_v^t (\mathcal{I}(v,r,t))^2~dr~dv \leq C~(t-s)^{4H+2}.
\end{equation}
\end{proposition}

\begin{proof}
As above, use the change of variable $\theta = r+\frac{r-v}{\alpha}$. 
It comes:
$$ \mathcal{I}(v,r,t) = (r-v)^{2H}~\int_{\frac{r-v}{t-r}}^\infty 
(1+\alpha)^{H-\frac{3}{2}}~\alpha^{-2H-1}~d\alpha. $$
 Notice that
$$ \forall z>0,~~\int_z^\infty 
(1+\alpha)^{H-\frac{3}{2}}~\alpha^{-2H-1}~d\alpha
\leq \int_z^\infty \alpha^{-2H-1}~d\alpha \leq \frac{C}{z^{2H}}. $$
For $z=\frac{r-v}{t-r}\wedge 1$ one has
$\frac{1}{z^{2H}} \leq \max(1,\frac{(t-r)^{2H}}{(r-v)^{2H}})$.
It follows that
\begin{equation*} 
\mathcal{I}(v,r,t) \leq C~((r-v)^{2H}+ (t-r)^{2H}).
\end{equation*}

The inequality~\eqref{ineq:L2-maj-I-reg(v,r,t)} then results from
$$ \int_s^t \int_v^t ((r-v)^{4H} + (t-r)^{4H})~dr~dv
\leq C~\int_s^t (t-v)^{4H+1}~dv. $$
\end{proof}

\begin{proposition} \label{lemma:maj-I-diese-reg(v,r,t)}
Let $\mathcal{I}^\sharp(v,r,t)$ be  defined for~$0<r<v<t$ by
\begin{equation} \label{def:I-diese(v,r,t)}
\mathcal{I}^\sharp(v,r,t) := \int_v^t 
(\theta-v)^{2H-\frac{3}{2}}~(\theta-r)^{H+\frac{1}{2}}~d\theta.
\end{equation} 

For any $\tfrac{1}{4}<H<\tfrac{1}{2}$ one has
\begin{equation} \label{ineq:L2-maj-Idiese(v,r,t)}
\int_s^t \int_0^v (\mathcal{I}^\sharp(v,r,t))^2~dr~dv 
\leq C_H~(t-s)^{4H}~t^{2H+2}.
\end{equation}
\end{proposition}

\begin{proof}
$$ \int_v^t 
(\theta-v)^{2H-\frac{3}{2}}~(\theta-r)^{H+\frac{1}{2}}~d\theta
\leq C_H~(t-r)^{H+\frac{1}{2}}~(t-v)^{2H-\frac{1}{2}}
\leq C_H~t^{H+\frac{1}{2}}~(t-s)^{2H-\frac{1}{2}}. $$
\end{proof}

\begin{proposition}
Let $\mathcal{I}^\natural(v,r,t)$ be  defined for~$0<v<s<t$ by
\begin{equation} \label{def:I-natural(v,r,t)}
\mathcal{I}^\natural(v,r,t) := \int_s^t 
\Big| \partial_\theta K_H(\theta,v) \Big|
~(\theta-r)^{H+\frac{1}{2}}~d\theta.
\end{equation} 

For any $\tfrac{1}{4}<H<\tfrac{1}{2}$ one has
\begin{equation} \label{ineq:L2-maj-Inatural(v,r,t)}
\int_0^s \int_0^t (\mathcal{I}^\natural(v,r,t))^2~dr~dv 
\leq (t-s)^{2H}~t^{2H+1}.
\end{equation}
\end{proposition}

\begin{proof}
We again notice that the map~$\partial_\theta K_H(\theta,v)$ is either 
positive or negative
(see~\eqref{eq:deriv-K_H-simplifiee-H-petit} and 
\eqref{eq:deriv-K_H-simplifiee-H-grand}).
Therefore,
$$ \int_0^t(\mathcal{I}^\natural(v,r,t))^2~dr
\leq t^{2H+1}~\Big(\int_s^t
\partial_\theta K_H(\theta,v)~d\theta \Big)^2
= t^{2H+1}~(K_H(t,v) - K_H(s,v))^2. $$
By again using~\eqref{eq:norme-L2-K_H} 
we deduce:
$$ \int_0^s \int_0^t (\mathcal{I}^\natural(v,r,t))^2~dr~dv 
\leq \int_0^t \int_0^t (\mathcal{I}^\natural(v,r,t))^2~dr~dv 
\leq (t-s)^{2H}~t^{2H+1}. $$
\end{proof}

\section{Bounds on $\wl$ and its derivatives (Proof of Proposition~\ref{prop:boundW}) }
\label{app:boundW}

The aim of this section is to show the following proposition.

\begin{customprop}{\ref{prop:boundW}}
For any $\lambda>0$, let $\wl(y)$ be defined as 
in~\eqref{eq:ODE_Y_ext}. Under the assumptions \ref{hyp:h1'} and 
\ref{hyp:h2'} on 
$b$ 
and $\sigma$ one has
\begin{equation} \label{maj:wl-old}
\forall y\in\R,~0\leq \wl(y) 
\leq e^{-|\mathbb{Y}-y|~\mathcal{R}(\lambda)},
\end{equation}
where $\mathcal{R}(\lambda)$ is defined as in~\eqref{eq:def_muRST}:
$\mathcal{R}(\lambda):=\sqrt{2\lambda +\mu^2} - \mu$.

In addition, the two first derivatives of $\wl$ satisfy the following 
estimates: There exists $C>0$ depending on $\mu$ only 
such that, for all real numbers $y$ and $\tilde{y}$,
\begin{gather}
|\wl'(y)| \leq C(1 + \lambda)~e^{-|\mathbb{Y}-y|~\mathcal{R}(\lambda)}, 
\label{maj:wl-d1-old} \\
|\wl''(y)| \leq C(1+\lambda) 
~e^{-|\mathbb{Y}-y|~\mathcal{R}(\lambda)} , \label{maj:wl-d2-old} \\
|\wl''(y) - \wl''(\tilde{y})| \leq C~(1+\lambda)^2~|y-\tilde{y}|~
\left( e^{-|\mathbb{Y}-y|~\mathcal{R}(\lambda)} + 
e^{-|\mathbb{Y}-\tilde{y}|~\mathcal{R}(\lambda)} \right).
\label{maj:diff-wl-d2-old}
\end{gather}
\end{customprop}

\begin{proof}
We successively consider $y<\mathbb{Y}$ and $y\geq \mathbb{Y}$.

\underline{The case $y<\mathbb{Y}$.} 

Let the Lamperti process $\mathbf{Y}$ be defined as 
in~\eqref{eq:fSDE_Lamperti-BM}.
Let $\mathbf{Y}^\uparrow$ be defined as: $\mathbf{Y}_t^\uparrow = y+ 
\mathbf{B}_t + \mu t$ where,
as above, $\mu := |\widetilde{b}|_\infty$.
Denote by $\tau^\uparrow_\mathbf{Y}$ the first time 
$\mathbf{Y}^\uparrow$ hits 
$\mathbb{Y}$. As $\mathbf{Y}_t \leq \mathbf{Y}_t^\uparrow$ a.s. 
for every~$t\geq 0$ one has
$\tau_\mathbf{Y}^\uparrow \leq \tau_\mathbf{Y}$~a.s., from which
\begin{equation*}
\EE\left(e^{-\lambda \tau_\mathbf{Y}}\right) 
\leq \EE\left(e^{-\lambda \tau_\mathbf{Y}^\uparrow}\right) 
= e^{\mu(\mathbb{Y}-y)- (\mathbb{Y}-y)\sqrt{2\lambda + \mu^2}},
\end{equation*} 
where the last equality can be found in e.g.~\cite{Borodin}. 
The inequality~\eqref{maj:wl-old} follows.

Let us now prove the estimate on $\wl'$.
We use a trick provided to us by P-E.~Jabin.
In view of~(\ref{eq:ODE_Y}) we have 
\begin{equation}\label{eq:wprime}
\forall \tilde{y}\leq y\leq\mathbb{Y},~~\wl'(y) = \wl'(\tilde{y})
-2\int_{\tilde{y}}^y \widetilde{b}(z) \wl'(z)\ dz + 2\lambda 
\int_{\tilde{y}}^y \wl(z)\ dz.
\end{equation}
Integrate w.r.t.~$\tilde{y}$ between $y-1$ and $y$ to obtain
\begin{equation*}
\wl'(y) = \wl(y) - \wl(y-1) + \int_{y-1}^y
\left(-2\int_{\tilde{y}}^y \widetilde{b}(z) \wl'(z)\ dz 
+ 2\lambda \int_{\tilde{y}}^y  \wl(z)\ dz\right)~d\tilde{y}.
\end{equation*}
From~\eqref{eq:ODE_Y_ext} it results that the function~$\wl$ is 
positive and increasing on the 
interval $(-\infty,\mathbb{Y})$. Consequently,
\begin{align*}
0\leq \wl'(y) &\leq \wl(y)
+ 2\mu \int_{y-1}^y \int_{\tilde{y}}^y \wl'(z)\ dz ~d\tilde{y}
+ 2\lambda \int_{y-1}^y\int_{\tilde{y}}^y 
\wl(z)~dz~d\tilde{y} \\
&\leq C(1+\lambda)~\wl(y).
\end{align*}
The desired inequality~\eqref{maj:wl-d1-old} follows 
from~\eqref{maj:wl-old}.

The inequality~\eqref{maj:wl-d2-old} follows from~\eqref{maj:wl-old}, 
\eqref{maj:wl-d1-old}, and the differential equation~(\ref{eq:ODE_Y}).

Finally, to get~\eqref{maj:diff-wl-d2-old} we start 
from~(\ref{eq:ODE_Y}):
\begin{equation*}
\wl''(y) - \wl''(\tilde{y}) = 2\lambda (\wl(y)-\wl(\tilde{y})) - 
(\widetilde{b}(y)-\widetilde{b}(\tilde{y}))~\wl'(y) 
- \widetilde{b}(\tilde{y})(\wl'(y)-\wl'(\tilde{y})).
\end{equation*}
First, for any $\tilde{y}<y$, in view of~\eqref{maj:wl-d1-old} we have
$$ \wl(y) - \wl(\tilde{y}) \leq C~(1+\lambda)~(y-\tilde{y})~
e^{-(\mathbb{Y}-y)~\mathcal{R}(\lambda)}. $$
Second, from~\eqref{maj:wl-d1-old} we deduce that 
$$ |(\widetilde{b}(y)-\widetilde{b}(\tilde{y}))~\wl'(y)| 
\leq 
C~(1+\lambda)~(y-\tilde{y})~e^{-(\mathbb{Y}-y)~\mathcal{R}(\lambda)}. $$
Finally, again use that $\wl'$ is positive
and satisfies~\eqref{eq:wprime} to get
\begin{align*}
|\wl'(\tilde{y}) - \wl'(y)| &\leq C \int_{\tilde{y}}^y \wl'(z) dz + 
2\lambda \int_{\tilde{y}}^y \wl(z) dz\\
&\leq C(\wl(y)-\wl(\tilde{y})) + 2\lambda(y-\tilde{y}) 
e^{-(\mathbb{Y}-y)~\mathcal{R}(\lambda)}\\
&\leq C~(1+\lambda)~(y-\tilde{y})~
e^{-(\mathbb{Y}-y)~\mathcal{R}(\lambda)}.
\end{align*}
It then remains to exchange the roles of $y$ and $\tilde{y}$ 
to obtain~\eqref{maj:diff-wl-d2-old}.

\vspace{0.3cm}
\underline{The case $y\geq\mathbb{Y}$.} 

In that case, we have that 
$\mathbb{Y}-(2\mathbb{Y}-y)=|\mathbb{Y}-y|$.
The desired 
estimates follow from the definition of $\wl$ on the interval 
$(\mathbb{Y},+\infty)$ (see~\eqref{eq:ODE_Y_ext}) and the calculations 
for 
the case $y<\mathbb{Y}$ which imply that $|\wl'(\mathbb{Y})| \leq 
C(1+\lambda)$ 
and $|\wl''(\mathbb{Y})| \leq C(1+\lambda)$.
\end{proof}

\section{Proof of Proposition~\ref{prop:moments}} \label{App:moments}

The proof of Proposition~\ref{prop:moments} relies on the following 
elementary lemma.

\begin{lemma}
\begin{enumerate}[label=(\roman*)]
Set $m:=\mathbb{Y}-y_0$ and $\mu:=|\widetilde{b}|_\infty$. Let $q>0$.
\item 
Let $Y^{H\uparrow}$ be the process defined as
\begin{equation} \label{def:YH+}
Y^{H\uparrow}_t = y_0 + \mu\,t + B^H_t.
\end{equation}
One has
\begin{equation} \label{esp-W_lambda-Y+}
\EE\Big(e^{-q(\mathbb{Y}-Y^{H\uparrow}_s)}
~\indi{Y^{H\uparrow}_s\leq\mathbb{Y}}\Big)
\leq C~\exp\Big(-\tfrac{1}{2}\frac{(m-\mu s)^2}{s^{2H}}
~\indi{\frac{m-\mu s}{s^{2H}} \leq q}
- \tfrac{1}{2}q(m-\mu s)~\indi{\frac{m-\mu s}{s^{2H}} > q}\Big).
\end{equation}
\item Let~$G$ be any standard Gaussian random variable. One has
\begin{equation} \label{ineq:moment-Wlambda-YH-App}
\begin{split}
\EE\,e^{-q|\mathbb{Y}-Y^H_s|}
&\leq
C~\exp\left(-\frac{1}{2}\frac{(m-\mu s)^2}{s^{2H}}
~\indi{\frac{m-\mu s}{s^{2H}} \leq q}
- \frac{q}{2}(m-\mu s)~\indi{\frac{m-\mu s}{s^{2H}} > q}\right) \\
&~~~+ C~\PP(G \geq \frac{m-\mu s}{s^H}).
\end{split}
\end{equation}
\end{enumerate}
\end{lemma}

\begin{proof}
We start with proving~\eqref{esp-W_lambda-Y+}.

Define the decreasing function $f$ on $\R_+$ by
$$ f(q) := \EE\left(e^{-q(\mathbb{Y}-Y^{H\uparrow}_s)}
~\indi{Y^{H\uparrow}_s\leq\mathbb{Y}}\right)
= \EE\left(\exp(-q(m-B^H_s-\mu s))~\indi{B^H_s+\mu s\leq m} \right). $$
Notice that
$$ \int^{\frac{m-\mu s}{s^H}}_{-\infty} e^{q s^H y-\frac{y^2}{2}}~dy
= e^{\frac{1}{2}q^2 s^{2H}}~
\int^{\frac{m-\mu s}{s^H}}_{-\infty} e^{-\frac{1}{2}(y-q s^H)^2}~dy
= e^{\frac{1}{2}q^2 s^{2H}}~
\int^{\frac{m-\mu s}{s^H}-q s^H}_{-\infty} e^{-\frac{z^2}{2}}~dz. $$
Therefore,
$$ f(q) = \frac{1}{\sqrt{2\pi}}\exp(-q(m-\mu s)+\tfrac{1}{2}q^2 s^{2H})
~\int^{\frac{m-\mu s}{s^H}-q s^H}_{-\infty} e^{-\frac{z^2}{2}}~dz. $$

\paragraph*{When $\frac{m-\mu s}{s^{2H}} \leq q$:}

As~$f$ is decreasing, one has
$$ f(q) \leq f\left(\frac{m-\mu s}{s^{2H}} \right)
= C~\exp\left(-\frac{(m-\mu s)^2}{s^{2H}}
+\tfrac{1}{2}\frac{(m-\mu s)^2}{s^{2H}}\right) = 
C~\exp\left(-\tfrac{1}{2}\frac{(m-\mu s)^2}{s^{2H}}\right). $$

\paragraph*{When $\frac{m-\mu s}{s^{2H}} > q$:}
One then has
$$ f(q) \leq \exp(-q(m-\mu s)+\tfrac{1}{2}q(m-\mu s))
= \exp\left(-\tfrac{1}{2}q(m-\mu s)\right). $$

We therefore have obtained~\eqref{esp-W_lambda-Y+}.

\vspace{0.3cm}

We now turn to~\eqref{ineq:moment-Wlambda-YH-App}. Observe that
\begin{equation*}
\begin{split}
\EE\,e^{-q|\mathbb{Y}-Y^H_s|} &=
\EE\left(e^{-q(\mathbb{Y}-Y^H_s)}~\indi{Y^{H\uparrow}_s\leq\mathbb{Y}}\right)
+ \EE\left(e^{-q|\mathbb{Y}-Y^H_s|}
~\indi{Y^{H\uparrow}_s\geq\mathbb{Y}}\right) \\
&\leq 
\EE\left(e^{-q(\mathbb{Y}-Y^{H\uparrow}_s)}~
\indi{Y^{H\uparrow}_s\leq\mathbb{Y}}\right)
+ \PP(Y^{H\uparrow}_s\geq\mathbb{Y}).
\end{split}
\end{equation*}
Letting $G$ be defined as in the 
statement of the proposition we thus have
$$ \EE\,e^{-q|\mathbb{Y}-Y^H_s|} \leq 
\EE\left(e^{-q(\mathbb{Y}-Y^{H\uparrow}_s)}~
\indi{Y^{H\uparrow}_s\leq\mathbb{Y}}\right) 
+ \PP\left(G \geq \frac{m-\mu s}{s^H}\right). $$
It then remains to use~\eqref{esp-W_lambda-Y+}.
\end{proof}

We now are in a position to prove Proposition~\ref{prop:moments} that 
we recall here.

\begin{customprop}{\ref{prop:moments}}
Let $\lambda >|\widetilde{b}'|_\infty$. Let $m:=\mathbb{Y}-y_0$,
$\mu:=|\widetilde{b}|_\infty$, $q:=p\mathcal{R}(\lambda)$
and 
$\widetilde{\lambda}:= \lambda - |\widetilde{b}'|_\infty$. 

One has
\begin{equation} \label{ineq:moment-Wlambda-YH-integ-App}
\rate_p(\mathbb{Y}-y_0,\lambda)
\leq C~\Big( e^{-\frac{q}{2}m}
+ e^{-\frac{\widetilde{\lambda}}{2} \Psi^H_q(m)} 
+\exp\big(-2^{-\frac{8}{3}}~m^{\frac{2}{1+2H}}~
\widetilde{\lambda}^{\frac{2H}{1+2H}}\big)
+ \exp\big(-\widetilde{\lambda} \tfrac{m}{2\mu} \big) \Big),
\end{equation}
where
\begin{equation} \label{def-App:Psi-H-App}
\Psi^H_q(m) := \frac{m}{\mu+q}
~\indi{\left[\left(\frac{m}{\mu+q}\right)^{2H-1} < 1\right]}
+ \left(\frac{m}{\mu+q}\right)^{\frac{1}{2H}}
~\indi{\left[\left(\frac{m}{\mu+q}\right)^{2H-1} \geq 1\right]}.
\end{equation}
\end{customprop}

\begin{proof}
In view of~\eqref{ineq:moment-Wlambda-YH-App} we have
$$ \sup_{s\in\R_+} e^{-\widetilde{\lambda}s} 
~\EE\,e^{-q|\mathbb{Y}-Y^H_s|}
\leq J_1(\widetilde{\lambda}) + J_2(\widetilde{\lambda}), $$
where
$$ J_1(\widetilde{\lambda}) :=
C~\sup_{s\in\R_+} e^{-\widetilde{\lambda}s} 
\exp\left(-\frac{1}{2}\frac{(m-\mu s)^2}{s^{2H}}
~\indi{\frac{m-\mu s}{s^{2H}} \leq q}
- \frac{q}{2}(m-\mu s)~\indi{\frac{m-\mu s}{s^{2H}} > q}\right) $$
and
$$ J_2(\widetilde{\lambda}) :=
C~\sup_{s\in\R_+} e^{-\widetilde{\lambda}s} 
~\PP(G \geq \frac{m-\mu s}{s^H}). $$

We start with estimating $J_1(\widetilde{\lambda})$. Observe that the 
map
$$ \phi(s) := q s^{2H} + \mu s - m $$ 
is increasing, which implies that there exists a 
unique $s_\ast$ such that $\phi(s_\ast)=0$, that is, such that
$$ \left[\frac{m-\mu s}{s^{2H}} \leq q\right] 
\Longleftrightarrow s\geq s_\ast. $$
It comes:
\begin{equation} \label{maj-I11-intermed}
J_1(\widetilde{\lambda}) \leq C
\sup_{s\in [0,s_\ast]} e^{-\widetilde{\lambda} s-\frac{1}{2}q(m-\mu s)}
+ C\sup_{s\in(s_\ast,\infty)} e^{-\widetilde{\lambda} s}.
\end{equation}
In order to bound the preceding expression from above we bound 
$s_\ast$ from above and below as follows.

\vspace{0.3cm}

\textbf{An upper bound for $s_\ast$.}
Noticing that $\phi(\frac{m}{\mu})>0$ we get $s_\ast<\frac{m}{\mu}$. 

\vspace{0.3cm}

\textbf{A lower bound for $s_\ast$.} We aim to get a~$s$ such that 
~$\phi(s) \leq 0$. We distinguish two cases: 
\begin{itemize}
\item If $\left(\frac{m}{\mu+q}\right)^{2H-1} < 1$ one has
$$ \phi\left(\frac{m}{\mu+q}\right) \leq q\frac{m}{\mu+q}
+\mu\frac{m}{\mu+q}-m = 0. $$
\item If $\left(\frac{m}{\mu+q}\right)^{2H-1} \geq 1$ one has
$$ \phi\left(\left(\frac{m}{\mu+q}\right)^{\frac{1}{2H}}\right) 
= q\frac{m}{\mu+q} + \mu\left(\frac{m}{\mu+q}\right)^{\frac{1}{2H}}
-m \leq q\frac{m}{\mu+q}+\mu\frac{m}{\mu+q}-m = 0. $$
\end{itemize}
To summarize,
$$ s_\ast \geq \Psi^H_q(m), $$
where $\Psi^H_q(m)$ is defined as in~\eqref{def-App:Psi-H-App}.

We now come back to~\eqref{maj-I11-intermed} and observe that
\begin{equation*}
\begin{split}
\sup_{s\in[0,s_\ast]} e^{-\widetilde{\lambda} s-\frac{1}{2}q(m-\mu s)}
&\leq \sup_{s\in [0,\frac{s_\ast}{2}]} e^{-\widetilde{\lambda} 
s-\frac{1}{2}q(m-\mu s)}
+ \sup_{s\in (\frac{s_\ast}{2}, s_\ast]} 
e^{-\widetilde{\lambda} s-\frac{1}{2}q(m-\mu s)} \\
&\leq e^{-\frac{1}{2}qm} + 
e^{-\frac{1}{2}q (m-\mu s_\ast) - \widetilde{\lambda} \frac{s_\ast}{2}} 
\\
&\leq e^{-\frac{1}{2}qm} + 
e^{-\frac{1}{2}\widetilde{\lambda}\Psi^H_q(m)},
\end{split}
\end{equation*}
and
$$ \sup_{s\in (s_\ast,\infty)} e^{-\widetilde{\lambda} s}
\leq e^{-\widetilde{\lambda} \Psi^H_q(m)}. $$

We are in a position to conclude that
\begin{equation}
\boxed{J_1(\widetilde{\lambda}) \leq C e^{-\frac{q}{2}m}
+ C e^{-\frac{\widetilde{\lambda}}{2} 
\Psi^H_q(m)}. }
\end{equation}

\paragraph*{An upper bound for $J_2(\widetilde{\lambda})$.}

$$J_2(\widetilde{\lambda})
\leq C \sup_{s\in [0,\tfrac{m}{2\mu}]} e^{-\widetilde{\lambda} s}
~\PP\left(G \geq \frac{m - \mu s}{s^H}\right) + 
C \sup_{s\in (\tfrac{m}{2\mu},\infty)} e^{-\widetilde{\lambda} s}
~\PP\left(G \geq \frac{m - \mu s}{s^H}\right) .
$$

Observe that
$$ \sup_{s\in (\tfrac{m}{2\mu},\infty)} e^{-\widetilde{\lambda} s}
\PP\left(G \geq \frac{m - \mu s}{s^H}\right)
\leq \exp\left(-\widetilde{\lambda} 
\tfrac{m}{2\mu} \right). $$
In addition,
\begin{equation*}
\begin{split}
\sup_{s\in [0,\tfrac{m}{2\mu}]} e^{-\widetilde{\lambda} s}
\PP\left(G \geq \frac{m - \mu s}{s^H}\right)
&\leq \sup_{s\in [0,\tfrac{m}{2\mu}]} 
\exp\left(-\widetilde{\lambda} s
- \frac{1}{2}\frac{(m - \mu s)^2}{s^{2H}}\right) \\
&\leq \sup_{s\in [0,\tfrac{m}{2\mu}]} 
\exp\left(-\widetilde{\lambda} s - \frac{m^2}{8 s^{2H}}\right).
\end{split}
\end{equation*}
The function $s\mapsto - \widetilde{\lambda} s - \frac{m^2}{8 s^{2H}}$ 
reaches its maximum at $s = \left(\frac{H 
m^2}{4\widetilde{\lambda}}\right)^{\frac{1}{1+2H}}$.
Therefore,
$$ \forall s>0,~~- \widetilde{\lambda} s - \frac{m^2}{8 s^{2H}}
\leq -\widetilde{\lambda} \left(\frac{H 
m^2}{4\widetilde{\lambda}}\right)^{\frac{1}{1+2H}}
- \frac{m^2}{8} \left(\frac{H 
m^2}{4\widetilde{\lambda}}\right)^{\frac{-2H}{1+2H}}
\leq - c~m^{\frac{2}{1+2H}}~\widetilde{\lambda}^{\frac{2H}{1+2H}} $$
where $c:=\min_{\frac{1}{4}<H<1}(\tfrac{H}{4})^{\frac{1}{1+2H}}
=2^{-\frac{8}{3}}$, 
from which
$$ \sup_{s\in [0,\tfrac{m}{2\mu}]} e^{-\widetilde{\lambda} s}
~\PP\left(G \geq \frac{m - \mu s}{s^H}\right) 
\leq \exp\left(-2^{-\frac{8}{3}}~m^{\frac{2}{1+2H}}~
\widetilde{\lambda}^{\frac{2H}{1+2H}}\right). $$

We conclude:
\begin{equation}
\boxed{J_2(\widetilde{\lambda}) \leq 
C~\exp\left(-2^{-\frac{8}{3}}~m^{\frac{2}{1+2H}}~
\widetilde{\lambda}^{\frac{2H}{1+2H}}\right)
+ C \exp\left(-\widetilde{\lambda}\tfrac{m}{2\mu} \right). } 
\end{equation}

That ends the proof of~\eqref{ineq:moment-Wlambda-YH-integ-App}.
\end{proof}

\section{Glossary} \label{App:glossary}
\begin{itemize}
\item The process $Y^H$ is defined in Proposition~\ref{prop:Lamperti}.
\item In the statement of Theorem~\ref{prop:majgap_0} one defines 
$\matcal{R}(\lambda) := \sqrt{2\lambda + \mu^2} - \mu$ and 
$\mu := |\widetilde{b}|_\infty := |\tfrac{b}{\sigma}|_\infty$.
\item The constants $mathbb{Y}$ and $y_0$ are defined at the beginning
of the section~\ref{sec:majgap_0}. In that section we
set $m:=\Theta-y_0$.
\item The function $\wl$ is defined by~\eqref{eq:ODE_Y_ext}. It 
satisfies the ODE~\eqref{eq:ODE_Y} on the interval $(-\infty,\Theta)$.
\end{itemize}

\end{appendices}

\end{document}